\documentclass[12pt,a4paper,twoside]{article}
\usepackage{amsmath}
\usepackage{amssymb}
\usepackage{dsfont}
\usepackage{graphicx}
\usepackage[utf8]{inputenc}
\usepackage[T1]{fontenc}
\usepackage{amsmath,amssymb,amsthm}
\usepackage{mathrsfs}
\usepackage{mathtools}
\usepackage{geometry}
\usepackage{hyperref}
\usepackage{cleveref}
\usepackage{xcolor}
\usepackage{enumitem}
\usepackage{url}
\usepackage{fancyhdr}
\usepackage{booktabs}
\usepackage{longtable}
\usepackage{float}
\usepackage{upgreek}
\usepackage{environ}
\usepackage[super]{nth}
\usepackage{hyperref}
\hypersetup{colorlinks, citecolor={black}, linkcolor={black}, urlcolor={black}}
\usepackage{cancel}
\usepackage{tikz}
\usetikzlibrary{patterns}
\NewEnviron{comment}{}
\theoremstyle{plain} 
\newtheorem{theorem}{Theorem}[section]
\newtheorem{lemma}[theorem]{Lemma}
\newtheorem{corollary}[theorem]{Corollary}
\newtheorem{proposition}[theorem]{Proposition}

\theoremstyle{definition} 
\newtheorem{definition}[theorem]{Definition}
\newtheorem*{notation}{Notation}
\newtheorem{remark}[theorem]{Remark}
\newtheorem{example}[theorem]{Example}

\usepackage{xparse} 

\newcounter{mytag}[section] 
\renewcommand{\themytag}{\thesection.\arabic{mytag}} 

\usepackage{xparse} 
\NewDocumentCommand{\mytag}{o}{%
    \refstepcounter{mytag}
    \tag{\themytag}%
    \IfNoValueF{#1}{\label{#1}}%
}

\setlist[enumerate]{label=(\roman*)}

\newcommand{\R}{\mathbb{R}}
\newcommand{\N}{\mathbb{N}}
\newcommand{\T}{\mathbb{T}}

\newcommand{\Z}{\boldsymbol{\mathcal{Z}}}
\newcommand{\Q}{\mathbb{Q}}
\newcommand{\X}{\boldsymbol{\mathcal{X}}}
\newcommand{\F}{\mathcal{F}}
\newcommand{\Y}{\boldsymbol{\mathcal{Y}}}
\newcommand{\U}{\boldsymbol{\mathcal{U}}}
\newcommand{\V}{\mathbb{V}}
\newcommand{\W}{\mathbf{W}}
\newcommand{\Wbar}{\overline{\mathbf{W}}}
\newcommand{\I}{\boldsymbol{\mathcal{I}}}

\renewcommand{\H}{\boldsymbol{\mathcal{H}}}

\let\oldupvarphi\upvarphi
\renewcommand{\upvarphi}{\boldsymbol{\oldupvarphi}}

\let\oldUpphi\Upphi
\renewcommand{\Upphi}{\boldsymbol{\oldUpphi}}

\newcommand{\remainder}{\mathfrak{R}}
\newcommand{\tildeY}{\Tilde{y}}

\newcommand{\conditionRDE}{(\mathfrak{C})}

\renewcommand{\epsilon}{\varepsilon}
\renewcommand{\phi}{\varphi}

\newcommand{\metricDynamicalSystem}{(\Omega,\mathcal{F},P,(\theta_t)_{t\in\R})}
\newcommand{\metricDynamicalSystemT}{(\Omega,\mathcal{F},P,(\theta_t)_{t\in\T})}

\newcommand{\toinf}{\rightarrow\infty}
\newcommand{\formallyEqual}{\approx^{\mathrm{f}}}
\newcommand{\FzeroPrimeZero}{F_0'(0)}
\newcommand{\FPrimeZero}{F'(0)}

\newcommand{\randomVariables}[1]{\mathrm{RV}(#1)}
\newcommand{\homogeneousPolynomials}[1]{H_{n,D}(#1)}
\newcommand{\driverRoughPath}[2]{\mathscr{C}^\alpha(#1,#2)}
\newcommand{\controlledRoughPath}[3]{\mathscr{D}^{2\alpha}_{#1}(#2,#3)}
\newcommand{\linearMaps}[2]{\mathscr{L}(#1,#2)}

\newcommand{\abs}[1]{\left| #1\right|}
\newcommand{\id}{\mathrm{id}}
\newcommand{\ad}[1]{\mathrm{ad}_{n}#1}

\newcommand{\expectationPathwise}[1]{\int_\Omega #1\,\mathrm{d}P(\omega)}
\newcommand{\norm}[2]{\left\|#1\right\|_{#2}}
\newcommand{\normCalpha}[1]{\norm{#1}{\mathscr{C}^\alpha}}
\newcommand{\normDalpha}[1]{\norm{#1}{\mathscr{D}^{2\alpha}_\W}}
\newcommand{\normDalphaOmega}[1]{\norm{#1}{\mathscr{D}^{2\alpha}_{\W(\omega)} }}
\newcommand{\spectrum}[1]{\mathrm{spec}\left(#1\right)}
\newcommand{\diam}[1]{\mathrm{diam}\left(#1\right) }

\newcommand{\projection}{\mathrm{pr}}
\newcommand{\range}{\mathrm{ran}}
\newcommand{\conv}{\mathrm{conv}}

\title{Normal Forms for Rough Differential Equations}
\author{Gideon Chiusole\thanks{Technical University of Munich. \texttt{g.chiusole@tum.de}}~, Christian Kuehn\thanks{Technical University of Munich. \texttt{ckuehn@ma.tum.de}} ~and Maximilian Semenowicz\thanks{Technical University of Munich. \texttt{maximilian.semenowicz@tum.de}}}
\date{June 12, 2026}

\begin{document}

\maketitle

\begin{abstract}
    We address the existence of normal forms for rough ordinary differential equations. We assume suitable smoothness and the hyperbolicity of an equilibrium point. In this context, we establish local formal equivalence of the two solution flows generated by a random nonlinear RDE and its linearized version. This provides the foundation for extending normal form theory to rough differential equations.
\end{abstract}




\tableofcontents

\section{Introduction}\label{sec:introduction}
Normal form theory, originally initiated in \cite{poincare1892methodes}, aims to simplify a given differential equation through a smooth coordinate transformation, ideally achieving linearity. We work locally in the phase space $\R^D$ near a fixed point at the origin. Concretely, given a smooth diffeomorphism $\phi:\R^D\to\R^D$ with $\phi(0)=0$, that we imagine as a map coming from a corresponding flow, the normal form problem is to find a coordinate transformation $H:\R^D\to\R^D$ with $H(0)=0$ and $\mathrm{D}H(0)=\id$ such that $\psi:=H^{-1}\circ\phi\circ H$ is as simple as possible. The normal form of $\phi$ near $0$, which is the simplest $\psi$ we are able to reach, may thus be viewed as a generalization of the Jordan canonical form to the nonlinear setting.\\

In this work, we study the question how to extend normal form theory to rough ordinary differential equations (RDEs)
\begin{align*}
\mathrm{d}Y_t=F_0(Y_t)\,\mathrm{d}t+F(Y_t)\,\mathrm{d}\W_t(\omega)\mytag\label{eq:introRODE}
\end{align*}
where $\W$ denotes a centered Gaussian rough path that is also a rough path cocycle. In particular, we assume that the origin $0$ is an equilibrium point and ask, whether there is a coordinate transformation of (\ref{eq:introRODE}) to its linearized version
\begin{align*}
\mathrm{d}Y_t=\mathrm{D}F_0(0)Y_t\,\mathrm{d}t+\mathrm{D}F(0)Y_t\,\mathrm{d}\W_t(\omega).
\end{align*}
The existing deterministic normal form framework \cite{wigginsDeterministic}, \cite{birkhoff1927dynamical}, \cite{guckenheimer1983nonlinear}, \cite{Kuznetsov2023} must be refined to accommodate the irregularity of rough paths and the randomness of stochastic processes. For more classical stochastic differential equations (SDEs) in the Stratonovich sense, there exist normal form results as present in \cite{ArnoldImkeller1998} and \cite{Arnold1998RDS}. The challenge is whether it is possible to transfer the theory to the rough setting, where much broader classes of driving/forcing processes are allowed.

The normal form problem we intend to tackle boils down to the search for near-identity random coordinate transforms $H:\Omega\times\R^D\to\R^D$ conjugating two random local flows, over a metric dynamical system $\metricDynamicalSystem$ with shift $\theta_t:\omega(\cdot)\mapsto\omega(\cdot+t)-\omega(t)$ on the Wiener space. In our case, these flows are generated by random rough differential equations whose driver $(\W(\omega))_\omega$ forms a rough path cocycle in the sense of \cite{BAILLEUL_riedel_20175792}, linking rough path theory with random dynamical systems theory. The central obstruction to linearization, also appearing in deterministic and stochastic settings, is the possible failure of hyperbolicity. Said hyperbolicity for random dynamical systems is calculated via the Lyapunov spectrum $\spectrum{\Phi}$ of a linearized flow $\Phi$. The  multiplicative ergodic theorem ensures under broad hypotheses the existence and regularity of a Lyapunov spectrum \cite{Oseledets1968English}, \cite{Filip2017MET}, \cite{Viana_2014}, \cite{Arnold1998RDS}. If $0\in\spectrum{\Phi}$, then hyperbolicity fails and there is a resonance effect generically blocking the transformation to a normal form. Hence, ruling out resonance of all orders is an essential assumption.

As expected, extending normal form theory from SDEs to RDEs introduces further analytical challenges. The effort is justified, though, for many real-world systems are driven by rough, non-martingale noise - such as fractional Brownian motion with Hurst parameter $\in(0,1/2)$, see  \cite{Biagini2008}, \cite{Ndiaye2026Averaging} - where standard stochastic calculus fails \cite{Coviello2011StochasticCalculus}, \cite{protter2005stochastic}, \cite{BeigelboeckSiorpaes2014BichtelerDellacherie}. The two main obstacles we overcome in this work are:
\begin{itemize}
 \item Classical SDE normal form theory uses martingale tools such as the Burkholder-Davis-Gundy inequality; these are unavailable for RDEs. We develop pathwise estimates and Gaussian tail bounds in their place.
 \item Coordinate changes are themselves rough in time, as they are built from solutions to RDEs. In particular, they carry the structure of rough paths - a richer object compared to its deterministic or SDE counterparts. This requires again additional technical tools to be developed.
 \end{itemize}

Although we have not discussed all the technical details of the setup yet, let us still state the main result of this work to clarify the final goal.

\begin{theorem}[Formal Linearization of an RDE]\label{main thm in introduction}
    Let  $(\W(\omega))_\omega\subseteq\driverRoughPath{\R}{\R^M}$ be a centered Gaussian rough path that is also a weakly geometric rough path cocycle. Furthermore, let
 $F_0\in C^\infty(\R^D,\R^D)\cap C^3_{\mathrm{b}}(\R^D,\R^D)$ and
     $F\in C^\infty(\R^D,\linearMaps{\R^{M}}{\R^D})\cap C^3_{\mathrm{b}}(\R^D,\linearMaps{\R^{M}}{\R^D})$ with $F_0(0)=0$ and $F(0)=0$.
     We consider the RDE (\ref{eq:introRODE})
    in $\R^D$. If the linear cocycle generated by the linearized equation
    \begin{align*}
        \mathrm{d}Y_t=\mathrm{D}F_0(0)Y_t\,\mathrm{d}t+\mathrm{D}F(0)Y_t\,\mathrm{d}\W_t(\omega)
    \end{align*}
    is non-resonant of all orders $\in\N_{\geq2}$, that is, if the linear RDEs
    \begin{align*}
        \mathrm{d}Y_t =\ad \mathrm{D}F_0(0)Y_t \,\mathrm{d}t+ \ad \mathrm{D}F(0)Y_t \,\mathrm{d}\W_t(\omega)
    \end{align*}
    generate hyperbolic linear cocycles for all $n\in\N_{\geq2}$,
    then there exists a random coordinate transform $H:\Omega\times\R^D\to\R^D$ that formally locally conjugates the solution flows of the initial RDE and its linearization, rendering these two equations formally equivalent.

    In addition to that, the formal Taylor expansion $H(\omega,x)\formallyEqual x+\sum_{n=2}^\infty H^n(\omega,x)$ is uniquely determined and is composed of coefficients that are controlled rough paths
    \begin{align*}
        t\mapsto \H^n(\theta_t\omega):=(H^n(\theta_t\omega),\dot H^n(\theta_t\omega))\in\controlledRoughPath{\W(\omega)}{\R}{H_{n,D}(\R^D)}
    \end{align*}
    themselves. Both the path and the Gubinelli derivative meet the temperedness condition.
\end{theorem}

At this point, we should clarify what we mean by "formally". We say that Theorem \ref{main thm in introduction} holds true formally because we do not take into account the question of convergence/divergence of the (therefore a priori formal) power series defining the random coordinate transform $H$. Instead, we manipulate the Taylor series coefficient by coefficient and work with underlying algebraic relations. This formal procedure here happens in analogy to the deterministic and, especially, to the stochastic counterpart.
Be aware that proper convergence is already a delicate issue in the deterministic setting. In fact, \cite{Krikorian2022} proves that for a large amount of coefficient functions, the deterministic Birkhoff normal form is divergent. \\

The paper is structured as follows. Section 2 recalls some fundamentals of random dynamical systems presented in \cite{Arnold1998RDS}, and Section 3 provides a concise reminder of rough path theory in the sense of Gubinelli \cite{GUBINELLI200486}, following \cite{FrizHairer2020}.

Sections 4 and 5 establish the analytical framework for the rough normal form. Section 4 develops the necessary calculus for coordinate transformations in the rough path setting, with a central contribution being a rough Itô-Wentzell-type formula stated in Theorem \ref{rde lem ito-wentzell}, which handles coefficients exhibiting rough time-dependence themselves. Section 4 also bridges rough path theory with random dynamical systems by verifying that RDEs driven by rough path cocycles generate cocycles in the sense of random dynamical system, and that linearized rough flows are eligible for applications of the multiplicative ergodic theorem.

Section 5 presents the main normal form construction, building upon the SDE approach of \cite[Section 8.5]{Arnold1998RDS} while addressing the analytical challenges posed by rough drivers. A formal power series ansatz of the random coordinate transform $H(\omega,x)\formallyEqual x+\sum_{n=2}^\infty H^n(\omega,x)$ is substituted into the to-be-satisfied conjugation relation, reducing the normal form problem to finding invariant measures corresponding to a recursive hierarchical system of affine RDEs. An important technical task is to extend the moment conditions coined in \cite[Lemma 8.5.3]{Arnold1998RDS} to the rough setting and to prove that these moment bounds propagate through the hierarchy. Doing so, we inherently have to study Gaussian tail estimates first examined in \cite{CLL}. Under the non-resonance assumption, we establish that each stage of the conjugation relation admits a unique tempered stationary solution, constructed via integral series over the Oseledets decomposition. The result is the unique tempered random coordinate transformation $H$ that formally linearizes the original RDE.\\

For the reader's convenience, we attach a table of generic but frequently used notation in Table \ref{tab:notation} in Appendix \ref{appendix notation}.

\section{Background: Random Dynamical Systems}\label{sec:rds}

The theory of random dynamical systems provides a framework for the study of dynamical systems (RDS) subject to randomness. A random dynamical system is defined as a measurable cocycle over a metric dynamical system, where the latter models the temporal evolution of the random input itself, and is therefore regarded as the underlying base flow of the probability space. Thereby, this formulation extends classical deterministic dynamical systems to a random setting.
For completeness and to fix the notation, we review the most fundamental concepts of RDS with direct importance to our work. For a detailed examination of the broad field of RDS, we refer the reader to the classic monograph \cite{Arnold1998RDS}.\\

    Let $(\Omega,\mathcal{F},P)$ denote a probability space. We use Borel $\sigma$-algebras when dealing with measurability questions in the context of Euclidean spaces. We write $\T$ to denote time; in our work, that means $\T\in\{\R,\R_{\geq0}\}$.

    \begin{definition}[{\cite[Definition 1.1.1]{Arnold1998RDS}}]\label{sde def metric dyn system}
        Let $(\theta_t:\Omega\rightarrow\Omega)_{t\in\T}$ be a family of $P$-invariant transformations, that is, they satisfy $P=P\circ \theta_t^{-1}$ for all $t\in\T$, which has the following properties:
        \begin{enumerate}
            \item $(t,\omega)\mapsto\theta_t\omega$ is 
            measurable,
            \item $\theta_0=\id$,
            \item $\theta_{t+s}=\theta_t\theta_s$ for all $s,t\in\T$.
        \end{enumerate}
        The tuple $\metricDynamicalSystemT$ is called a metric dynamical system.
    \end{definition}

    \begin{definition}[{\cite[Definition 1.1.2]{Arnold1998RDS}}]\label{sde def random dynamical system}
 A random dynamical system on a finite-dimensional Banach space $\V$ over a metric dynamical system $\metricDynamicalSystemT$, or short cocycle, is a measurable map
        \begin{align*}
            \phi:\T\times \Omega\times \V\rightarrow \V,\;(t,\omega,x)\mapsto\phi(t,\omega,x)
        \end{align*}
        which has the following properties:
        \begin{enumerate}
            \item $\phi(0,\omega,\cdot)=\id$ for all $\omega\in\Omega$,
            \item $\phi(t+s,\omega,x)=\phi(t,\theta_s\omega,\phi(s,\omega,x))$ for all $x \in \V$, $\omega\in\Omega$, and $s,t\in \T$,
            \item $\phi(\cdot,\omega,\cdot):\T\times \V\rightarrow \V$ is continuous for all $\omega\in\Omega$.
        \end{enumerate}
    We refer to (ii) as the cocycle property or cocycle identity.
    \end{definition}

    Next, we introduce the important concept of sub-exponential growth of a random variable with respect to a metric dynamical system, which appears in various places in RDS theory.

    \begin{definition}[{\cite[Definition 3.1]{ArnoldImkeller1998}}]
        We call a $\V$-valued random variable $X$ tempered with respect to the metric dynamical system $\metricDynamicalSystem$ if
        \begin{align*}
            \lim_{\abs{t}\rightarrow \infty} \frac{\log^+\abs{X(\theta_t\omega)}}{\abs{t}}=0
        \end{align*}
        holds for almost every $\omega\in\Omega$, where $\log^+ x:=\max\{\log x,0\}$ for $x\in\R_{\geq0}$.
    \end{definition}

    \begin{definition}[{\cite[Section 3.3]{ArnoldImkeller1998}}]
        A random probability measure $\mu$ on $(\Omega\times\R^D,\mathcal{F}\otimes \mathrm{Borel}(\R^D))$, given by $\,\mathrm{d}\mu(\omega,x)=\mathrm{d}\mu_\omega(x)\,\mathrm{d}P(\omega)$, with marginal $P$ on $(\Omega,\mathcal{F})$ is called invariant for the cocycle $\phi$ if $\phi(t,\omega)_*\mu_\omega=\mu_{\theta_t\omega}$ holds true almost surely and for all $t\in\R$, where the subscript $_*$ denotes the pushforward.
    \end{definition}

    Last, we recall Oseledets' celebrated multiplicative ergodic theorem (usually abbreviated as MET). We present only the statements that are the most important to our work, but see {\cite[Section 3.4]{Arnold1998RDS}}, especially \cite[Theorem 3.4.1 and Theorem 3.4.11]{Arnold1998RDS}, for a thorough account of MET variants.

    \begin{theorem}[Multiplicative Ergodic Theorem; {\cite[Theorem 3.4.11]{Arnold1998RDS}}]\label{thm MET}
        Let $\Phi$ be a linear cocycle over $\metricDynamicalSystem$. Assume the following integrability condition:
        \begin{align*}
            \int_\Omega\log^+\left(\sup_{t\in[0,1]}\abs{\Phi(t,\omega,0)}\right)+\log^+\left(\sup_{t\in[0,1]}\abs{\Phi(t,\omega,0)^{-1}}\right)\mathrm{d}P(\omega) <\infty.
        \end{align*}
        Then, there exists a $\theta$-invariant set $\Omega_0\in\mathcal{F}$ with $P(\Omega_0)=1$ on which there is the random Oseledets splitting of $\R^D$, which has the following properties:
        \begin{enumerate}
            \item For all $\omega\in\Omega_0$, we have the splitting into the Oseledets spaces given by \begin{align*}\R^D= \bigoplus_{r=1}^{R(\omega)}E_r(\omega).\end{align*}
            The mapping $\omega\mapsto E_r(\omega)$ is measurable for each $r=1,\dots,R(\omega)$. The Oseledets subspaces $E_r(\omega)$ possess dimension $\dim(E_r(\omega))=d_r(\omega)$, and are invariant under the flow in the sense that $\Phi(t,\omega)E_r(\omega)=E_r(\theta_t\omega)$ for all $t\in\R$ and $r=1,\dots,R(\omega)$. Equivalently, if $\projection^r(\omega):\R^D\to E_r(\omega)$ denotes the projection onto $E_r(\omega)$ along $\bigoplus_{r'=1, r'\neq r}^{R(\omega)}E_{r'}(\omega)$, then $ \Phi(t,\omega)\projection^r(\omega)=\projection^r(\theta_t\omega)\Phi(t,\omega)$ holds true for all $t\in\R$ and $r=1,\dots,R(\omega)$.
            \item  For every $x\in\R^D\setminus\{0\}$, the limit $\lim_{\abs{t}\rightarrow\infty}\log\abs{\Phi(t,\omega)x}/\abs{t}$ exists and fulfills for $r=1,\dots,R(\omega)$ the condition
            \begin{align*}
                x\in E_r(\omega)\setminus\{0\}\text{ if and only if } \lim_{\abs{t}\rightarrow\infty} \frac{\log\abs{\Phi(t,\omega)x}}{|t|}=\lambda_r(\omega),
            \end{align*}
            where the $\lambda_r(\omega)$ are the so-called Lyapunov exponents constituting the Lyapunov spectrum $\spectrum{\Phi}:=\{(\lambda_r(\cdot),d_r(\cdot)):r=1,\dots,R(\cdot)\}$, see also {\cite[Definition 3.3.8]{Arnold1998RDS}}. Be aware that sometimes we also just refer to the Lyapunov exponents when talking about the Lyapunov spectrum $\spectrum{\Phi}:=\{\lambda_r(\cdot):r=1,\dots,R(\cdot)\}$.
        \end{enumerate}
    \end{theorem}

\section{Background: Rough Path Theory}\label{sec:rough-path-theory}
    \subsection{Rough Paths}
        
\begin{comment}
In this preliminary section, we provide a brief introduction to the realm of Rough Path Theory originally initiated in \cite{Lyons1994}. However, we do not only understand Rough Path Theory in Lyons's original setting, which is highly developed in \cite{Friz_Victoir_2010}. Rather, we adopt the viewpoint proposed in \cite{GUBINELLI200486}, where another notion of controlled rough paths was introduced, which turns out to be a natural object for our purposes that include integration at heart. For a detailed account of the topic, the standard reference is \cite{FrizHairer2020}.

 In our work, we exclusively deal with rough paths with "irregularity parameter" $\alpha\in(1/3,1/2]$ that refers to Hölder continuity. The reason is that while this regularity range makes room for a real extension of Young's integration theory initiated in \cite{Young1936}, it also constitutes the simplest and in contemporary literature most prominently featured case. From now on, we shall refrain from introducing $\alpha\in(1/3,1/2]$ in every statement, for any upcoming occurrence of $\alpha$ shall mean the regularity regime. Many of the upcoming estimations implicitly depend upon $\alpha$, which we usually suppress as the regularity regime $\alpha$ always exists in the background.
 Upon further notice, $T\in\R_{>0}$ denotes a finite time horizon. Recall that for any path $X:\mathrm{dom}(X)\to\V$ and $s,t\in\mathrm{dom}(X)\subseteq\R$ with $s<t$, we abbreviate increments $X_{s,t}:=X_t-X_s$, as is customary in Rough Path Theory.
 \end{comment}
Rough path theory, initiated in \cite{Lyons1994}, provides a robust analytical framework for studying differential equations driven by highly irregular signals. In particular, the theory extends classical integration and differential equation techniques beyond the semimartingale setting, allowing one to give pathwise meaning to equations driven by rough noise such as fractional Brownian motion \cite{FrizHairer2020}. Central to the theory is the enhancement of a driving signal by its iterated integrals which encode higher-order oscillatory information and, therefore, eliminate obstructions we encounter with standard martingale methods.
 In this section, we introduce the key concepts and notation regarding rough paths, following \cite{FrizHairer2020}, to which we refer the reader for a detailed account. We shall refrain from introducing the parameter $\alpha\in(1/3,1/2]$ controlling the regularity of a rough path in every statement. Upon further notice, $T\in\R_{>0}$ denotes a finite time horizon.

        \begin{definition}[{\cite[Definition 2.1]{FrizHairer2020}}]\label{rde def driver rough path}
             Put $\Delta:=\{(s,t)\in[0,T]^2:s<t\}$. An $\alpha$-Hölder rough path (over $\V$) is a pair $\W:=(W,\mathbb{W})$, where the first-order path $W_\cdot:[0,T]\rightarrow \V$ satisfies \begin{align*}
                \norm{W}{\alpha} := \sup_{(s,t)\in\Delta} \frac{\abs{W_{s,t}}}{(t-s)^\alpha} <\infty,
            \end{align*}
            the second-order process $\mathbb{W}_{\cdot,\cdot}:[0,T]^2\rightarrow \V\otimes \V$ fulfills
            \begin{align*}
                \norm{\mathbb{W}}{2\alpha} := \sup_{(s,t)\in\Delta} \frac{\abs{\mathbb{W}_{s,t}}}{(t-s)^{2\alpha}} <\infty,
            \end{align*}
            and
            Chen's relation, which reads $W_{s,u}\otimes W_{u,t}=\mathbb{W}_{s,t}-\mathbb{W}_{s,u}-\mathbb{W}_{u,t}$ for all $s,u,t\in[0,T]$ with $s<u<t$, is met.
            We denote the space of $\alpha$-Hölder rough paths by $\mathscr{C}^\alpha([0,T],\V)$. Last, we impose a metric on $\mathscr{C}^\alpha([0,T],\V)$ defined as
            \begin{align*}
                \rho_{\mathscr{C}^\alpha}: \mathscr{C}^\alpha([0,T],\V) \times \mathscr{C}^\alpha([0,T],\V) \rightarrow \R,\; (\W,\widetilde{\W})\mapsto \norm{W-\widetilde{W}}{\alpha} + \norm{\mathbb{W}-\widetilde{\mathbb{W}}}{2\alpha}.
            \end{align*}
            In turn, $\rho_{\mathscr{C}^\alpha}$ gives rise to the semi-norm
            \begin{align*}
                \norm{\cdot}{\mathscr{C}^\alpha}: \mathscr{C}^\alpha([0,T],\V) \rightarrow \R,\;\W\mapsto \norm{W}{\alpha} + \norm{\mathbb{W}}{2\alpha}.
            \end{align*}
        \end{definition}
    Rough paths are typically defined on compact intervals, which is mainly due to Hölder continuity being satisfied only on compact intervals as opposed to infinite intervals. A meaningful extension to the entire real line $\R$ is given as follows:
        \begin{definition}[{\cite[Section 1]{BAILLEUL_riedel_20175792}}]\label{rde def rough driver on R} Set $\mathscr{I}_0:=\{I\subseteq\R: I\text{ is a compact interval with }0\in I\}$.
        We denote by $\mathscr{C}^\alpha(\R,\V)$ the space of pairs $\W:=(W,\mathbb{W})$, where $W_\cdot:\R\rightarrow \V$ and $\mathbb{W}_{\cdot,\cdot}:\R^2\rightarrow \V\otimes \V$, such that $\W|_I:=(W|_I,\mathbb{W}|_{I\times I})\in\driverRoughPath{I}{\V}$ for each $I\in\mathscr{I}_0$. We equip $\driverRoughPath{\R}{\V}$ with the coarsest topology for which the maps $\W\mapsto \W|_I$ are continuous for all $I\in\mathscr{ I}_0$.
        \end{definition}

    When we lift a stochastic process to a rough path, we often call it an enhanced or enriched rough path.
    While Definition \ref{rde def driver rough path} covers driving signals for rough differential equations (RDEs), we still have to deal with a suitable notion for the corresponding sought-after unknowns.

        \begin{definition}[{\cite[Definition 4.6]{FrizHairer2020}}]\label{def controlled rough path}
            Let $\W\in \mathscr{C}^\alpha([0,T],\V_1)$ and $\Y:=(Y,Y')\in C^\alpha([0,T],\V_2)\times C^\alpha([0,T],\mathscr{L}(\V_1,\V_2))$. We say that $\Y$ is an ($\alpha$-Hölder continuous) controlled rough path controlled by $\W$ if the remainder $\remainder\Y$, implicitly defined for $s,t\in\Delta$ by
           \begin{align*}
                Y_{s,t}= Y'_sW_{s,t}+\mathfrak{R} \Y_{s,t},
            \end{align*}
            is regular in the sense that
            \begin{align*}
                \norm{\mathfrak{R}\Y}{2\alpha}:= \sup_{(s,t)\in\Delta} \frac{\abs{\mathfrak{R}\Y_{s,t}}}{(t-s)^{2\alpha}} <\infty,
            \end{align*}
            where $\Delta:=\{(s,t)\in[0,T]^2:s<t\}$.
            The space of ($\alpha$-Hölder continuous) controlled rough paths controlled by $\W$ is denoted by $\mathscr{D}_\W^{2\alpha}([0,T],\V_2)$. The path $Y'$ indeed plays the role of a derivative of $Y$ with respect to $\W$ and is called the Gubinelli derivative. We endow $\mathscr{D}_\W^{2\alpha}([0,T],\V_2)$ with the semi-norm
            \begin{align*}
                \norm{\cdot}{\mathscr{D}_\W^{2\alpha}}: \mathscr{D}_\W^{2\alpha}([0,T],\V_2) \rightarrow \R,\; \Y\mapsto \norm{Y'}{\alpha}+\norm{\mathfrak{R}\Y}{2\alpha}.
            \end{align*}

            In analogy to Definition \ref{rde def rough driver on R}, we also regard $\Y\in\controlledRoughPath{\W}{\R}{\V_2}$, given $\W\in\driverRoughPath{\R}{\V_1}$, where this means that   $\Y|_I\in\controlledRoughPath{\W|_I}{I}{\V_2}$  for each compact interval $I\subseteq\R$ with $0\in I$.
        \end{definition}

        \begin{notation}
            Let $I\subseteq\R$ either be a compact interval with $0\in I$ or $I=\R$.
            We use bold font to denote rough paths. For rough drivers in $\driverRoughPath{I}{\V_1}$, we stick to $\W\in\driverRoughPath{I}{\V_1}$, exclusively. In turn, we denote controlled rough paths in $\controlledRoughPath{\W}{I}{\V_2}$ by caligraphic symbols like  $\Y$ for Latin letters and $\upvarphi$ for Greek letters. Sometimes, we refer to the explicit pair, too. For the Gubinelli derivative of  $\Y\in\controlledRoughPath{\W}{I}{\V_2}$, we usually write $Y'$ in Lagrange's style. If we want to stress that we are dealing with a time-derivative as opposed to differentiation with respect to any other potential argument, we also employ Newton's notation $\dot Y$.
        \end{notation}

        We shall below require the fundamental property of a rough path being (weakly) geometric. The power of this property lies in the closer resemblance to well-known calculi.
        \begin{definition}[{\cite[Section 2.2]{FrizHairer2020}}, {\cite[Definition 2.2]{castrequini2022itowentzellformularoughpaths}}]\label{rde def geometric rough path}
            Let $\W\in\driverRoughPath{[0,T]}{\V}$ and $\Delta:=\{(s,t)\in[0,T]^2:s<t\}$. We call $\W$ weakly geometric if
            \begin{align*}
                \mathrm{Sym}(\mathbb{W}_{s,t})=\frac{1}{2}W_{s,t}\otimes W_{s,t}
            \end{align*}
            holds true for all $(s,t)\in\Delta$. Here, given $\mathbb{X}\in\V_0\otimes\V_0$, we set the symmetric part $2\mathrm{Sym}(\mathbb{X}):=\mathbb{X}+\mathbb{X}^T$ and the antisymmetric part $2\mathrm{Anti}(\mathbb{X}):=\mathbb{X}-\mathbb{X}^T$. If otherwise $\W$ is not weakly geometric, we measure the deviation by the bracket of $\W$ that is given by
            \begin{align*}
                [\W]_{s,t}:=W_{s,t}\otimes W_{s,t}-2\mathrm{Sym}(\mathbb{W}_{s,t})
            \end{align*}
            for $(s,t)\in\Delta$. Notice that $[\W]\in C^{2\alpha}(\Delta,\V\otimes\V)$.

            If there exists a sequence of smooth paths $(W^n)_{n\in\N}\subseteq C^\infty([0,T],\V)$ with sequence of canonical lifts $(\W^n)_{n\in\N}\subseteq\driverRoughPath{[0,T]}{\V}$ defined as in \cite[Exercise 2.1]{FrizHairer2020} such that $\W^n\rightarrow\W$, that is, $\rho_{\mathscr{C}^\alpha}(\W^n, \W)\to0$, then we call $\W$ geometric.
        \end{definition}

        We realize that any geometric rough path is already weakly geometric, as a result of the lifts of smooth paths being weakly geometric. Weakly geometric rough paths will play an important role to us since, like Stratonovich calculus, they imitate first-order calculus quite well, see \cite[Section 2.2]{FrizHairer2020}. And, the (enhanced Gaussian) rough paths we will be interested in conveniently/purposefully turn out to be weakly geometric.  \\

        We have now encountered the general underlying concepts. Let us transition to the main operations that we have to carry out with (controlled) rough paths. The most basic yet integral operations are controlled integration and composition with regular maps.

        \begin{theorem}[Gubinelli; {\cite[Theorem 4.10]{FrizHairer2020}}]\label{theorem gubinelli estimates} Let $\W\in\mathscr{C}^\alpha([0,T],\V_1)$ and $\Y\in\mathscr{D}_\W^{2\alpha}([0,T],\mathscr{L}(\V_1,\V_2))$. Then, for every $s,t\in[0,T]$ with $s<t$, and finite partitions $\mathrm{P}\subseteq[s,t]$, the limit
        \begin{align*}
            \int_s^t Y_r\,\mathrm{d}\W_r:=\lim_{\abs{\mathrm{P}}\rightarrow0} \sum_{[u,v]\in \mathrm{P}} Y_u W_{u,v} + Y'_u \mathbb{W}_{u,v}
        \end{align*}
        exists in $\V_2$. In addition to that, $\I\in\mathscr{D}_\W^{2\alpha}([0,T],\V_2) $, where the integral path $\I:=(I,I')$ is defined by $I:=\int_0^\cdot Y_r\,\mathrm{d}\W_r$ and $I'=Y$. Furthermore, for some constant $C=C(\alpha)\in\R$,  the following two estimates hold true:
        \begin{enumerate}
            \item For every $s,t\in[0,T]$ with $s<t$, we have
            \begin{align*}
                \abs{\int_s^t Y_r\,\mathrm{d}\W_r - Y_sW_{s,t}  - Y'_s\mathbb{W}_{s,t}}\leq C(\norm{W}{\alpha}\norm{\mathfrak{R}\Y}{2\alpha} +\norm{\mathbb{W}}{2\alpha}\norm{Y'}{\alpha})\abs{t-s}^{3\alpha}.
            \end{align*}
            \item The map
            $
             \mathscr{D}_\W^{2\alpha}([0,T],\mathscr{L}(\V_1,\V_2)) \rightarrow \mathscr{D}_\W^{2\alpha}([0,T],\V_2),\; \Y \mapsto \I
            $
            is continuous and linear and satisfies the bound
            \begin{align*}
                \norm{\I}{\mathscr{D}_\W^{2\alpha}} \leq \norm{Y}{\alpha} + \norm{Y'}{\sup}\norm{\mathbb{W}}{2\alpha} + CT^\alpha(\norm{W}{\alpha}\norm{\mathfrak{R}\Y}{2\alpha} +\norm{\mathbb{W}}{2\alpha}\norm{Y'}{\alpha}).
            \end{align*}
        \end{enumerate}
        In particular, thanks to \cite[(4.20)]{FrizHairer2020}, (i) and (ii) have the direct consequences
        \begin{enumerate}[label=(\roman*')]
                \item $\abs{\int_s^t Y_r\,\mathrm{d}\W_r}\leq \abs{Y_sW_{s,t}}+\abs{Y'_s\mathbb{W}_{s,t}} + C \norm{\W}{\mathscr{C}^\alpha}\norm{\Y}{\mathscr{D}^{2\alpha}_\W}\abs{t-s}^{3\alpha}$, and
                \item $\norm{\I}{\mathscr{D}^{2\alpha}_\W} \leq
                C(1+\normCalpha{\W})(\abs{Y_0'}+T^\alpha\normDalpha{\Y})$.
            \end{enumerate}
        \end{theorem}

        \begin{notation}
            We often write $\int_0^\cdot\Y_r\,\mathrm{d}\W_r:=\I$ for the controlled rough path that comes from controlled integration as in Theorem \ref{theorem gubinelli estimates}.
        \end{notation}

        \begin{lemma}[{\cite[Lemma 7.3]{FrizHairer2020}}]\label{rde lem composing with regular fct}
            Let $\W\in\driverRoughPath{[0,T]}{\V_1}$ and $\Y\in\controlledRoughPath{\W}{[0,T]}{\V_2}$, and let $f\in C^2_{\mathrm{b}}(\V_2,\V_3)$. Then, $f(\Y):=(f(Y_\cdot),\mathrm{D}f(Y_\cdot)Y'_\cdot)\in\controlledRoughPath{\W}{[0,T]}{\V_3}$. Furthermore, there exists $C=C(T,\alpha)\in\R$ such that
            \begin{align*}
                \normDalpha{f(\Y)}\leq C\max\left\{1,\abs{Y_0'}+\normDalpha{\Y}\right\} \norm{f}{C^2_{\mathrm{b}}}(1+\norm{W}{\alpha})^2\left(\abs{Y_0'}+\normDalpha{\Y}\right).
            \end{align*}
            At last, a uniform choice of $C$ over $T\in(0,1]$ is valid.
        \end{lemma}
    \subsection{Rough Differential Equations}

We employ an extension of the theory described in \cite[Chapter 8]{FrizHairer2020} to RDEs with drift (as opposed to driftless RDEs). Nevertheless, we still refer to \cite{FrizHairer2020} as the extension is straightfoward. For example, one
could include explicit time in an enlarged driving rough path and alter the coefficient function accordingly. Of course, whenever we speak of an RDE like
    \begin{align*}
        \mathrm{d}Y_t=F_0(Y_t)\,\mathrm{d}t+F(Y_t)\,\mathrm{d}\W_t,
    \end{align*}
    assuming well-definedness, we really mean the corresponding integral equation
     \begin{align*}
        Y_t=Y_0+\int_0^tF_0(Y_s)\,\mathrm{d}s+\int_0^tF(Y_s)\,\mathrm{d}\W_s.
    \end{align*}
    A common well-posedness result is the following:

        \begin{theorem}[{\cite[Theorem 8.3 and Theorem 8.5]{FrizHairer2020}}, {\cite[Theorem 2.12]{Geng2021_RoughPathsNotes}}]\label{rde thm solutions to rdes}
            Fix some initial value $y_0\in \V_2$. Let $\W\in\driverRoughPath{[0,T]}{\V_1}$, and let $F\in C^3(\V_2,\linearMaps{\V_1}{\V_2} )$ and $F_0\in C^3(\V_2,\V_2)$.
            \begin{enumerate}
                \item  \underline{Existence and Uniqueness}:   Then, there are some $T_*\in(0,T]$ and some unique $\Y\in\controlledRoughPath{\W}{[0,T_*]}{\V_2}$ with $Y'=F(Y)$ such that we have
            \begin{align*}
                Y_t=y_0+\int_0^tF_0(Y_s)\,\mathrm{d}s+\int_0^tF(Y_s)\,\mathrm{d}\W_s
            \end{align*}
            for all $t\in[0,T_*]$. Moreover, if $F$ and $F_0$ are linearly growing or $\in C^3_{\mathrm{b}}$, then $T_*=T$.
            \item \underline{Continuity of the Solution Map, also Known as Universal Limit Theorem}: Let $\Y$ be as in (i). In addition, let $\widetilde{\W}\in\driverRoughPath{[0,T]}{\V_1}$ and let $\widetilde{\Y}\in\controlledRoughPath{\widetilde{\W}}{[0,T]}{\V_2}$ satisfy
            \begin{align*}
                \widetilde{Y}_t=\Tilde{y}_0+\int_0^tF_0(\widetilde{Y}_s)\,\mathrm{d}s+\int_0^tF(\widetilde{Y}_s)\,\mathrm{d}\widetilde{\W}_s,
            \end{align*}
            where $t\in [0,T_*]$ and $\Tilde{y}_0\in\V_2$. Then, it holds
            \begin{align*}
                &\quad\norm{Y'-\widetilde{Y}'}{\alpha}+\norm{\remainder \Y-\remainder\widetilde{\Y}}{2\alpha} \\
                &\leq Q\left(\alpha,T,\norm{F}{C^3_{\mathrm{b}}},\norm{F_0}{C^3_{\mathrm{b}}},\normCalpha{\W},\normCalpha{\widetilde{\W}}\right)\left(\rho_{\mathscr{C}^\alpha}(\W,\widetilde{\W})+\abs{y_0-\Tilde{y}_0}\right),
            \end{align*}
        where $Q:\R_{\geq0}^6\to\R_{\geq0}$ denotes an increasing polynomial. Furthermore,
        \begin{align*}
            \norm{Y-\widetilde{Y}}{\alpha}\leq Q\left(\alpha,T,\norm{F}{C^3_{\mathrm{b}}},\norm{F_0}{C^3_{\mathrm{b}}},\normCalpha{\W},\normCalpha{\widetilde{\W}}\right)\left(\rho_{\mathscr{C}^\alpha}(\W,\widetilde{\W})+\abs{y_0-\Tilde{y}_0}\right).
        \end{align*}
            \end{enumerate}
        \end{theorem}

        In fact, our requirements on the drift vector field $F_0$ are rather restrictive and allow for easy relaxation, see \cite[Exercise 8.5]{FrizHairer2020}. However, the exact regularity assumptions are of no importance to us, so, for ease of presentation, we align the conditions on $F_0$ and $F$. Now that we have fixed the rough path fundamentals, we advance to a detailed investigation of RDEs, placing more emphasis on links to our normal form procedure. Specifically, we are interested in the connection to the theory of RDS as presented in \cite{Arnold1998RDS}.

\section{Rough Itô-Wentzell Formula, Link from Rough Differential Equations to Random Dynamical Systems}\label{sec:rde}
    In this section, we fix a finite time horizon $T\in\R_{>0}$. Unless otherwise stated, all rough paths and RDEs are considered on the time domain $[0,T]$. Although it would be convenient to choose a fixed driver $\W\in\driverRoughPath{[0,T]}{\R^M} $, too, we refrain from doing so to avoid imposing unnecessary a priori-requirements on all results. Instead, we describe the respective conditions on $\W$ in each statement.

    \subsection{A Rough Chain Rule and a Rough Itô-Wentzell-type Formula}
        Let us quickly present the classical Gubinelli derivative product rule.
 \begin{lemma}[Rough Product Rule; {\cite[Corollary 7.4]{FrizHairer2020}}]\label{rde lem product rule}
 Let $\W\in\driverRoughPath{[0,T]}{\R^M}$, and
        let $\Upphi \in\controlledRoughPath{\W}{[0,T]}{\linearMaps{\R^D}{\R^K}}$ and $\Y\in\controlledRoughPath{\W}{[0,T]}{\linearMaps{\R^L}{\R^D}}$. Then, we have $\Upphi\Y\in\controlledRoughPath{\W}{[0,T]}{\linearMaps{\R^L}{\R^K}}$ through $(\Phi Y)_t:=\Phi_tY_t$ and the product rule $(\Phi Y)'_t:=\Phi_t'Y_t+\Phi_tY'_t$ with the interpretation $\Phi_t'Y_t\in\linearMaps{\R^M}{\linearMaps{\R^L}{\R^K}}$ via $\Phi_t'Y_tv:= \Phi_t'vY_t$ for $v\in\R^M$.
        Moreover,  \begin{align*}
            \normDalpha{\Upphi\Y}\leq \norm{\Phi'}{\sup}\norm{Y}{\alpha}(1+\norm{W}{\alpha})+2\norm{Y}{\sup} \normDalpha{\Upphi} +2\norm{\Phi}{\sup}\normDalpha{\Y} +\norm{\Phi}{\alpha}\norm{Y'}{\sup}.
        \end{align*}
    \end{lemma}

    \begin{remark}
        Setting $L=1$ and $D=K$, it follows immediately from Lemma \ref{rde lem product rule} that a similar result holds true for the regular matrix-vector application $\Upphi\Y\in\controlledRoughPath{\W}{[0,T]}{\R^D}$ given controlled rough paths $\Upphi\in\controlledRoughPath{\W}{[0,T]}{\linearMaps{\R^D}{\R^D}}$ and $\Y\in\controlledRoughPath{\W}{[0,T]}{\R^D}$.
    \end{remark}
Our approach to the normal form procedure is structurally based upon the SDE approach discussed in \cite[Section 8.5]{Arnold1998RDS}. For this reason, we need analogues of cornerstone results of stochastic analysis, namely the Stratonovich chain rule and the Itô formula. This section deals with the generalization of these statements.

\begin{definition}\label{rde def multi-indices}
            For $n\in\N$, we define the set
            \begin{align*}
                \N_n^D:=\left\{ \tau\in \N_0^D: \abs{\tau}:=\sum_{d=1}^D\tau_d=n\right\}
            \end{align*}
            of $D$-dimensional multi-indices of length $n$. For $\tau\in\N_n^D$, we abbreviate 
            $X^\tau:=\prod_{d=1}^DX_d^{\tau_d} $. Furthermore, we denote by
            \begin{align*}
                H_{n,D}(\V):=\left\{f=\sum_{\tau\in\N_n^D}  f_\tau X^\tau:f_\tau\in \V \right\}
            \end{align*}
            the vector space of homogeneous polynomials of order $n$ in $D$ variables with values in $\V$. Note that we adopt the standard notation from algebra to write an abstract polynomial $f$ in terms of indeterminates $X=(X_1,\dots,X_D)$. When dealing with a specific polynomial function $f(\cdot)$, we still refer to it as $f$.
        \end{definition}

    \begin{lemma}\label{rde lem polynomials simeq R^D}
        For any $n,D\in\N$, we have $\homogeneousPolynomials{\R^D}\simeq\R^{\overline{D}} $ for some suitable $\overline{D}\in\N$.
    \end{lemma}
    \begin{proof}
        For the straightforward proof, see \cite[Section 8.1]{Arnold1998RDS}.
    \end{proof}

 \begin{comment}
 \begin{lemma}\label{rde lem preliminary 0}
        Let $\W\in\driverRoughPath{[0,T]}{\R^M}$ and $\H\in\controlledRoughPath{\W}{[0,T]}{\homogeneousPolynomials{\V}}$. Then, every coefficient of $\H=(H,H')$ is $\in\controlledRoughPath{\W}{[0,T]}{\V}$.
    \end{lemma}
    \begin{proof}
        For $\sigma\in\N_n^D$, define $\projection^\sigma:\homogeneousPolynomials{\V}\rightarrow \V,\,f=\sum_{\tau\in\N_n^D}f_\tau X^\tau \mapsto f_\sigma$. Because $\projection^\sigma$ is a linear map between finite-dimensional Banach spaces, it is clearly smooth. Consequently, $\projection^\sigma\H\in\controlledRoughPath{\W}{[0,T]}{\V}$. But $ (\projection^\sigma H)_t=H_{\sigma,t}$ and $(\projection^\sigma H)'_t=H'_{\sigma,t}$ for all $t\in [0,T]$. 
    \end{proof}
    \begin{remark}
        Opposite to Lemma \ref{rde lem preliminary 0}, given a finite sequence $(\H_\tau)_{\tau\in\N_n^D}\subseteq\controlledRoughPath{\W}{[0,T]}{\V}$, it is straightforward to show that $\sum_{\tau\in\N_n^D}\H_\tau X^\tau\in\controlledRoughPath{\W}{[0,T]}{\homogeneousPolynomials{\V}}$. Therefore,
        Lemma \ref{rde lem preliminary 0} allows us to characterize  $\H\in\controlledRoughPath{\W}{[0,T]}{\homogeneousPolynomials{\V}}$ in terms of its coefficients, that is,
        \begin{align*}
            H_t=\sum_{\tau\in\N_n^D}H_{\tau,t}X^\tau \text{ and } H_t'=\sum_{\tau\in\N_n^D}H'_{\tau,t}X^\tau
        \end{align*}
        with $\mathfrak{R}\H_{s,t}=\sum_{\tau\in\N_n^D}\mathfrak{R}\H_{\tau,s,t}$.
    \end{remark}
    \end{comment}

\begin{lemma}[Rough Chain Rule]\label{rde lem preliminary 1}
Let $\W\in\driverRoughPath{[0,T]}{\R^M}$ and let $F:[0,T]\times\R^D\rightarrow\R^D$ be continuous and satisfy $F(t,\cdot)\in C^\infty(\R^D,\R^D)$ for all $t\in[0,T]$, with formal Taylor series expansion $F(t,x)\formallyEqual \sum_{n=0}^\infty F^n(t,x)$. Assume that, for every $n\in\N_{0}$ and any fixed $x\in\R^D$, we have $t\mapsto (F^n(t,x),\dot F^n(t,x))\in\controlledRoughPath{\W}{[0,T]}{\R^D}$. Here,
\begin{align*}
F^n(t,x):=\sum_{\tau\in\N_n^D}F^n_\tau (t)x^\tau
\text{ and }
\dot F^n(t,x):=\sum_{\tau\in\N_n^D}\dot F^n_\tau (t)x^\tau,
\end{align*}
with $\mathfrak{R}(F^n(\cdot,x),\dot F^n(\cdot,x))_{s,t}=\sum_{\tau\in\N_n^D} \mathfrak{R}(F^n_\tau(\cdot),\dot F^n_\tau(\cdot))_{s,t}x^\tau$. Then, the following statements hold true:
\begin{enumerate}
\item On a formal level,
\begin{align*}
t\mapsto (F(t,x),\dot F(t,x)):\formallyEqual\left(\sum_{n=0}^\infty F^n(t,x),\sum_{n=0}^\infty \dot F^n(t,x)\right)\in\controlledRoughPath{\W}{[0,T]}{\R^D}.
\end{align*}
\item For every $n\in\N_{0}$, we have
\begin{align*}
t\mapsto\left(\mathrm{D}_xF^n(t,x),\mathrm{D}_x\dot F^n (t,x)\right)\in\controlledRoughPath{\W}{[0,T]}{\linearMaps{\R^D}{\R^D}}.
\end{align*}
\item On a formal level,
\begin{align*}
t\mapsto (\mathrm{D}_xF(t,x),(\mathrm{D}_x F(\cdot,x))'_t):\formallyEqual\left(\sum_{n=1}^\infty \mathrm{D}_xF^n(t,x),\sum_{n=1}^\infty \mathrm{D}_x\dot F^n (t,x)\right)
\end{align*}
is an element of $\controlledRoughPath{\W}{[0,T]}{\linearMaps{\R^D}{\R^D}}.$
\end{enumerate}
Let now in addition $\Y\in\controlledRoughPath{\W}{[0,T]}{\R^D}$.
\begin{enumerate}\setcounter{enumi}{3}
\item Then, for every $n\in\N_{0}$, we have
\begin{align*}
t\mapsto(F^n(t, Y_t),\dot F^n(t,Y_t)+\mathrm{D}_xF^n(t,Y_t)Y_t')\in\controlledRoughPath{\W}{[0,T]}{\R^D}.
\end{align*}
\item On a formal level,
\begin{align*}
t\mapsto&(F(t,Y_t),\dot F(t,Y_t)+\mathrm{D}_xF(t,Y_t)Y_t')\\
&\quad:\formallyEqual\left(\sum_{n=0}^\infty F^n(t,Y_t),\sum_{n=0}^\infty \dot F^n(t,Y_t)+\mathrm{D}_xF^n(t,Y_t)Y_t'\right)
\end{align*}
is an element of $\controlledRoughPath{\W}{[0,T]}{\R^D}.$
\item For every $n\in\N_{0}$, we have that
\begin{align*}
t\mapsto\left(\mathrm{D}_xF^n(t,Y_t)Y_t',\mathrm{D}_x\dot F^n (t,Y_t)Y_t'+\mathrm{D}_x^2F^n(t,Y_t)(Y_t'\otimes Y_t')+\mathrm{D}_xF^n(t,Y_t)Y_t''\right)
\end{align*}
is an element of $\controlledRoughPath{\W}{[0,T]}{\linearMaps{\R^M}{\R^D}}.$
\item On a formal level, with the usual abbreviation of the sum,
\begin{align*}
t\mapsto&\left(\mathrm{D}_xF(t,Y_t)Y_t',\mathrm{D}_x\dot F (t,Y_t)Y_t'+\mathrm{D}_x^2F(t,Y_t)(Y_t'\otimes Y_t')+\mathrm{D}_xF(t,Y_t)Y_t''\right)
\end{align*}
is an element of $\controlledRoughPath{\W}{[0,T]}{\linearMaps{\R^M}{\R^D}}$.
\end{enumerate}
Assume now that already $t\mapsto(F(t,x),\dot F(t,x))\in\controlledRoughPath{\W}{[0,T]}{\R^D}$ for any fixed $x$. Then, the statements (v) and (vii) hold true even non-formally. To be precise:
\begin{enumerate}[label=(\roman*')]\setcounter{enumi}{4}
    \item We have that
    $
        t\mapsto (F(t,Y_t),\dot F(t,Y_t)+\mathrm{D}_xF(t,Y_t)Y_t')\in\controlledRoughPath{\W}{[0,T]}{\R^D}
    $ and
    \setcounter{enumi}{6}
    \item that $t\mapsto\big(\mathrm{D}_xF(t,Y_t)Y_t',\mathrm{D}_x\dot F (t,Y_t)Y_t'+\mathrm{D}_x^2F(t,Y_t)(Y_t'\otimes Y_t')+\mathrm{D}_xF(t,Y_t)Y_t''\big)$ is an element of $\controlledRoughPath{\W}{[0,T]}{\linearMaps{\R^M}{\R^D}}$.
\end{enumerate}
\end{lemma}

\begin{proof}
The proof is a tedious but routine calculus exercise we refrain from performing here. It is important to realize that it is valid to characterize  $\H\in\controlledRoughPath{\W}{[0,T]}{\homogeneousPolynomials{\V}}$ in terms of its coefficients $\H_\tau\in\controlledRoughPath{\W}{[0,T]}{\V}$, that is,
        \begin{align*}
            H=\sum_{\tau\in\N_n^D}H_{\tau,\cdot}X^\tau \text{ and } H'=\sum_{\tau\in\N_n^D}H'_{\tau,\cdot}X^\tau
        \end{align*}
        with $\mathfrak{R}\H_{s,t}=\sum_{\tau\in\N_n^D}\mathfrak{R}\H_{\tau,s,t}$. See also \cite[Section 3]{castrequini2022itowentzellformularoughpaths}.
\end{proof}

With the chain rule done, we still must establish an Itô-type formula that suits our purposes. The issue we face is the fact that the Itô formula stated in \cite[Theorem 7.7]{FrizHairer2020} is not strong enough, because it does not cover coefficients that explicitly depend upon time in a rough way. Thus, the current objective is to expand the result into such a setting. Our model is \cite[Theorem 4.1]{castrequini2022itowentzellformularoughpaths}, whose proof we upgrade as required by our normal form calculations taking place below.
\begin{theorem}[Rough Itô-Wentzell Formula]\label{rde lem ito-wentzell}
    Fix a driver $\W\in\driverRoughPath{[0,T]}{\R^M} $. Let $(\eta(\cdot,x),\dot\eta(\cdot,x))\in\controlledRoughPath{\W}{[0,T]}{\linearMaps{\R^M}{\R^D}}$ for all $x\in\R^D$. Assume the following conditions on $(\eta,\dot\eta)$:
    \begin{enumerate}
        \item The maps $\R^D\to\controlledRoughPath{\W}{[0,T]}{\linearMaps{\R^M}{\R^D}},\,x\mapsto(\eta(\cdot,x),\dot\eta(\cdot,x))$ and $\eta:[0,T]\times\R^D\to\linearMaps{\R^M}{\R^D}$ are continuous and $\eta(t,\cdot)\in C^\infty(\R^D,\linearMaps{\R^M}{\R^D})$ for every $t\in[0,T]$.
        \item The map $\mathrm{D}_x\eta:[0,T]\times\R^D\to\linearMaps{\R^D}{\linearMaps{\R^M}{\R^D}} $ is continuous and $\mathrm{D}_x\eta(t,\cdot)\in C^\infty(\R^D,\linearMaps{\R^D}{\linearMaps{\R^M}{\R^D}})$ for every $t\in[0,T]$.
        \item The map $\R^D\to\controlledRoughPath{\W}{[0,T]}{\linearMaps{\R^D}{\linearMaps{\R^M}{\R^D}}},\,x\mapsto (t\mapsto(\mathrm{D}_x\eta(t,x),(\mathrm{D}_x\eta(\cdot,x))'_t))$ is continuous.
    \end{enumerate}
    Let moreover $\eta_0:[0,T]\times\R^D\to\R^D$ be continuous and satisfy $\eta_0(t,\cdot)\in C^\infty(\R^D,\R^D)$ for all $t\in[0,T]$. Then, define for all $x\in\R^D$ the controlled rough path $\H(x):=(H(\cdot,x),\dot H(\cdot,x))\in\controlledRoughPath{\W}{[0,T]}{\R^D}$ by
    \begin{align*}
        H(t,x):=H(0,x)+\int_0^t\eta_0(s,x)\,\mathrm{d}s+\int_0^t\eta(s,x)\,\mathrm{d}\W_s.
    \end{align*}
    Assume that $\H$ enjoys the following properties:
    \begin{enumerate}
        \item For any $t\in[0,T]$, we have $H(t,\cdot)\in C^\infty(\R^D,\R^D)$.
        \item The maps $(t,x)\mapsto H(t,x)$, $(t,x)\mapsto \mathrm{D}_xH(t,x)$, $(t,x)\mapsto \mathrm{D}_x^2H(t,x)$ and $(t,x)\mapsto \mathrm{D}_x^3H(t,x)$ are continuous.
    \end{enumerate}
    As last object, we introduce $\Y\in\controlledRoughPath{\W}{[0,T]}{\R^D}$ by setting
        \begin{align*}
            Y_t:=Y_0+\int_0^t\Upsilon_0(s)\,\mathrm{d}s+\int_0^t\Upsilon_s\,\mathrm{d}\W_s,
        \end{align*}
    where $\Upsilon_0:[0,T]\to\R^D$ is continuous, and $(\Upsilon,\Upsilon')\in\controlledRoughPath{\W}{[0,T]}{\linearMaps{\R^M}{\R^D}}$.
    Then, the Itô-Wentzell-type formula
    \begin{align*}
          H(t,Y_t)&=H(0,Y_0)+ \int_0^t  \eta(s,Y_s)+ \mathrm{D}_x H(s,Y_s)\Upsilon_s\,\mathrm{d}\W_s\\
        &\quad+\int_0^t\eta_0(s,Y_s)+\mathrm{D}_xH(s,Y_s)\Upsilon_0(s)\,\mathrm{d}s\\
        &\quad +\int_0^t\mathrm{D}_x\eta(s,Y_s)Y_s'\,\mathrm{d}[\W]_s +\frac{1}{2}\int_0^t\mathrm{D}_x^2H(s,Y_s)(Y_s'\otimes Y_s')\,\mathrm{d}[\W]_s
    \end{align*}
    holds true, where the last two integrals are understood in the Young sense, see \cite{Young1936}. In particular, if $\W$ is weakly geometric, then these two Young integrals vanish due to $[\W]=0$.
\end{theorem}
To understand the depth of what is happening, it is crucial to be aware of the fact that a solution to an RDE is considered to have two Gubinelli derivatives in the sense that its Gubinelli derivative possesses a meaningful Gubinelli derivative itself. Let us illuminate this by reviewing the generic RDE $\,\mathrm{d}Z_t=F_0(Z_t)\,\mathrm{d}t+F(Z_t)\,\mathrm{d}\W_t$, assuming that everything is well-defined. The solution $\Z\in\controlledRoughPath{\W}{[0,T]}{\R^D}$ actually reads $\Z=(Z,F(Z))$. In addition to that, we know that $(F(Z),\mathrm{D}F(Z)Z')=(F(Z),\mathrm{D}F(Z)F(Z))\in\controlledRoughPath{\W}{[0,T]}{\linearMaps{\R^M}{\R^D}}$, which explains why it makes sense to talk about $Z''=\mathrm{D}F(Z)F(Z)$.

Now, if the integral representations of $H(\cdot,x)$ and $Y$ in Theorem \ref{rde lem ito-wentzell} appear as mild solution formulations of suitable RDEs, meaning that $(\eta(\cdot,x),\dot\eta(\cdot,x))=(\dot H(\cdot,x),\ddot H(\cdot,x))$ and $(\Upsilon,\Upsilon')=(Y',Y'')$, where we could write out the Gubinelli derivatives explicitly given the respective (drift and) diffusion vector fields, then the rough Itô-Wentzell formula is effectively a "mild solution formula"-type result for a certain composition of RDE solvers in terms of their respective drift and diffusion vector fields plus higher-order correction terms in form of Young integrals.

\begin{proof}[Proof of Theorem \ref{rde lem ito-wentzell}]
    To begin with, notice that $Y'=\Upsilon$ and $\dot H(\cdot,x)=\eta(\cdot,x)$ by definition. That being said, let $s,t\in[0,T]$ with $s<t$. We start by expanding
    \begin{align*}
        H(t,Y_t)-H(s,Y_s)=H(t,Y_t)-H(t,Y_s)+H(t,Y_s)-H(s,Y_s).
    \end{align*}
    Let us examine both summands separately.

    $\underline{\textit{Regarding the first summand}}$. A simple spatial Taylor expansion produces
    \begin{align*}
        H(t,Y_t)-H(t,Y_s)=\mathrm{D}_xH(t,Y_s)Y_{s,t} +\frac{1}{2}\mathrm{D}_x^2H(t,Y_s)(Y_{s,t}\otimes Y_{s,t})+ \mathrm{O}(\abs{t-s}^{3\alpha}).
    \end{align*}
    We make use of the two assertions
    \begin{align*}
        \mathrm{D}_xH(t,Y_s)Y_{s,t} =\mathrm{D}_xH(s,Y_s)Y_{s,t} +\mathrm{D}_x\eta (s,Y_s)Y_s'(W_{s,t}\otimes W_{s,t}) + \mathrm{O}(\abs{t-s}^{3\alpha})
    \end{align*}
    and
    \begin{align*}
        \mathrm{D}_x^2H(t,Y_s)(Y_{s,t}\otimes Y_{s,t})=2 \mathrm{D}_x^2H(s,Y_s)(Y_s'\otimes Y_s')\mathbb{W}_{s,t} +\mathrm{D}_x^2H(s,Y_s)(Y_s'\otimes Y_s')[\W]_{s,t}+\mathrm{O}(\abs{t-s}^{3\alpha}),
    \end{align*}
    whose justifications we postpone to Appendix \ref{appendix rde proof of ito-wentzell}. Substituting these into the Taylor expansion yields
    \begin{align*}
        H(t,Y_t)-H(t,Y_s)&=\mathrm{D}_xH(s,Y_s)Y_{s,t} +\mathrm{D}_x\eta(s,Y_s)Y_s'(W_{s,t}\otimes W_{s,t})\\
        &\quad +\mathrm{D}_x^2H(s,Y_s)(Y_s'\otimes Y_s')\mathbb{W}_{s,t} +\frac{1}{2}\mathrm{D}_x^2H(s,Y_s)(Y_s'\otimes Y_s')[\W]_{s,t}+\mathrm{O}(\abs{t-s}^{3\alpha}).
    \end{align*}
    According to the definition of $Y$ and Theorem \ref{theorem gubinelli estimates}, we have
    \begin{align*}
        Y_{s,t}=\Upsilon_sW_{s,t}+\Upsilon_s'\mathbb{W}_{s,t}+\int_s^t\Upsilon_0(u)\,\mathrm{d}u+\mathrm{O}(\abs{t-s}^{3\alpha}),
    \end{align*}
    inserting which lets us advance the previous line to
    \begin{align*}
        H(t,Y_t)-H(t,Y_s)&=\mathrm{D}_xH(s,Y_s)\Upsilon_sW_{s,t} + \mathrm{D}_xH(s,Y_s)\Upsilon_s'\mathbb{W}_{s,t}+\mathrm{D}_xH(s,Y_s)\int_s^t\Upsilon_0(u)\,\mathrm{d}u \\
        &\quad+\mathrm{D}_x\eta(s,Y_s)Y_s'(W_{s,t}\otimes W_{s,t})+\mathrm{D}_x^2H(s,Y_s)(Y_s'\otimes Y_s')\mathbb{W}_{s,t}\\
        &\quad+\frac{1}{2}\mathrm{D}_x^2H(s,Y_s)(Y_s'\otimes Y_s')[\W]_{s,t}+\mathrm{O}(\abs{t-s}^{3\alpha}).\mytag\label{rde tag ito-wentzell 1}
    \end{align*}
    Let us now take a different point of view. By Theorem \ref{theorem gubinelli estimates} and Lemma \ref{rde lem preliminary 1},
    \begin{align*}
        \int_s^t\mathrm{D}_xH(u,Y_u)\Upsilon_u\,\mathrm{d}\W_u&=\mathrm{D}_xH(s,Y_s)\Upsilon_sW_{s,t}+\Big((\mathrm{D}_x H(\cdot,Y_s))'_s\Upsilon_s\\
        &\quad+\mathrm{D}_x^2H(s,Y_s)(Y_s'\otimes \Upsilon_s)+\mathrm{D}_xH(s,Y_s)\Upsilon_s'\Big)\mathbb{W}_{s,t}+\mathrm{O}(\abs{t-s}^{3\alpha})\\
        &=\mathrm{D}_xH(s,Y_s)\Upsilon_sW_{s,t}+\mathrm{D}_x \eta(s,Y_s)\Upsilon_s\mathbb{W}_{s,t}^T\\
        &\quad+\mathrm{D}_x^2H(s,Y_s)(Y_s'\otimes Y_s')\mathbb{W}_{s,t}+\mathrm{D}_xH(s,Y_s)\Upsilon_s'\mathbb{W}_{s,t}+\mathrm{O}(\abs{t-s}^{3\alpha}),\mytag\label{rde tag ito-wentzell 1.1}
    \end{align*}
    where, in the last equation, we substituted $Y'=\Upsilon$ and, more importantly,
    recognized the decisive information $(\mathrm{D}_x H(\cdot,Y_s))'_s\Upsilon_s\mathbb{W}_{s,t}=\mathrm{D}_x\eta(s,Y_s)\Upsilon_s\mathbb{W}_{s,t}^T$. 
    Next, we subtract (\ref{rde tag ito-wentzell 1.1}) from (\ref{rde tag ito-wentzell 1}) and recall $W_{s,t}\otimes W_{s,t}=[\W]_{s,t}+2\mathrm{Sym}(\mathbb{W}_{s,t})$ and $Y'=\Upsilon$ to deduce
    \begin{align*}
        &\quad H(t,Y_t)-H(t,Y_s)-\int_s^t\mathrm{D}_xH(u,Y_u)\Upsilon_u\,\mathrm{d}\W_u\\
        &=\mathrm{D}_x\eta(s,Y_s)Y_s'[\W]_{s,t}+\frac{1}{2}\mathrm{D}_x^2H(s,Y_s)(Y_s'\otimes Y_s')[\W]_{s,t}+\mathrm{D}_xH(s,Y_s)\int_s^t\Upsilon_0(u)\,\mathrm{d}u\\
        &\quad-\mathrm{D}_x\eta(s,Y_s)Y_s'(\mathbb{W}^T_{s,t}-2\mathrm{Sym}(\mathbb{W}_{s,t}))+\mathrm{O}(\abs{t-s}^{3\alpha}).\mytag\label{rde tag ito-wentzell 2}
    \end{align*}
    Furthermore, standard rules for Riemann integration provide
    \begin{align*}
        \int_s^t\mathrm{D}_xH(u,Y_u)\Upsilon_0(u)\,\mathrm{d}u-\mathrm{D}_xH(s,Y_s)\int_s^t\Upsilon_0(u)\,\mathrm{d}u=\mathrm{o}(\abs{t-s}),
    \end{align*}
    so that, ultimately, equation (\ref{rde tag ito-wentzell 2}) becomes
    \begin{align*}
        &\quad H(t,Y_t)-H(t,Y_s)-\int_s^tD_xH(u,Y_u)\Upsilon_u\,\mathrm{d}\W_u\\
        &=\mathrm{D}_x\eta(s,Y_s)Y_s'[\W]_{s,t}+\frac{1}{2}\mathrm{D}_x^2H(s,Y_s)(Y_s'\otimes Y_s')[\W]_{s,t}+\int_s^tD_xH(u,Y_u)\Upsilon_0(u)\,\mathrm{d}u\\
        &\quad-\mathrm{D}_x\eta(s,Y_s)Y_s'(\mathbb{W}^T_{s,t}-2\mathrm{Sym}(\mathbb{W}_{s,t}))+\mathrm{o}(\abs{t-s}),\mytag\label{rde tag ito-wentzell 3}
    \end{align*}
    where we made use of $\mathrm{O}(\abs{t-s}^{3\alpha})=\mathrm{o}(\abs{t-s})$. Let us now tackle the other term.\\

    $\underline{\textit{Regarding the second summand}}$.
    Thanks to the definition of $H$ and Theorem \ref{theorem gubinelli estimates}, we again infer
    \begin{align*}
        H(t,x)-H(s,x)=\eta(s,x)W_{s,t}+\dot \eta(s,x)\mathbb{W}_{s,t}+\int_s^t\eta_0(u,x)\,\mathrm{d}u+\mathrm{O}(\abs{t-s}^{3\alpha}).
    \end{align*}
    Meanwhile, Theorem \ref{theorem gubinelli estimates} and Lemma \ref{rde lem preliminary 1} applied to the integral of the concatenation of $\eta$ and $Y$ instead of just $\eta$ produce
    \begin{align*}
        \int_s^t \eta(u,Y_u)\,\mathrm{d}\W_u=\eta(s,Y_s)W_{s,t}+\left(\dot \eta(s,Y_s)+\mathrm{D}_x\eta(s,Y_s)Y_s'\right)\mathbb{W}_{s,t}+\mathrm{O}(\abs{t-s}^{3\alpha}).
    \end{align*}
    For $x=Y_s$, subtracting the second line from the first line results in
    \begin{align*}
        H(t,Y_s)-H(s,Y_s)-\int_s^t \eta(u,Y_u)\,\mathrm{d}\W_u=\int_s^t\eta_0(u,Y_s)\,\mathrm{d}u-\mathrm{D}_x\eta(s,Y_s)Y_s'\mathbb{W}_{s,t}+\mathrm{O}(\abs{t-s}^{3\alpha}).
    \end{align*}
    Again, with $\int_s^t\eta_0(u,Y_s)\,\mathrm{d}u=\eta_0(s,Y_s)\abs{t-s}+\mathrm{o}(\abs{t-s})$ and $\int_s^t\eta_0(u,Y_u)\,\mathrm{d}u=\eta_0(s,Y_s)\abs{t-s}+\mathrm{o}(\abs{t-s})$, and $\mathrm{O}(\abs{t-s}^{3\alpha})=\mathrm{o}(\abs{t-s})$, a final substitution leads to
    \begin{align*}
        H(t,Y_s)-H(s,Y_s)-\int_s^t \eta(u,Y_u)\,\mathrm{d}\W_u=\int_s^t\eta_0(u,Y_u)\,\mathrm{d}u-\mathrm{D}_x\eta(s,Y_s)Y_s'\mathbb{W}_{s,t}+\mathrm{o}(\abs{t-s}).\mytag\label{rde tag ito-wentzell 4}
    \end{align*}

    At this point, we find ourselves in a position to finish the proof by putting together the identities that we have derived thus far. Indeed, writing $\mathrm{P}\subseteq[0,t]$ for partitions, identifying a telescopic sum in (1) and inserting (\ref{rde tag ito-wentzell 3}) and (\ref{rde tag ito-wentzell 4}) in (2), we end up with
    \begin{align*}
        H(t,Y_t)-H(0,Y_0)
        &\overset{(1)}{=}\sum_{[u,v]\in\mathrm{P}} (H(v,Y_v)-H(v,Y_u)) +(H(v,Y_u)-H(u,Y_u))\\
        &\overset{(2)}{=}\sum_{[u,v]\in\mathrm{P}}\bigg(\bigg(\int_u^v\mathrm{D}_xH(w,Y_w)\Upsilon_w\,\mathrm{d}\W_w+\mathrm{D}_x\eta(u,Y_u)Y_u'[\W]_{u,v}\\
        &\quad\quad+\frac{1}{2}\mathrm{D}_x^2H(u,Y_u)(Y_u'\otimes Y_u')[\W]_{u,v}+\int_u^v\mathrm{D}_xH(w,Y_w)\Upsilon_0(w)\,\mathrm{d}w\\
        &\quad\quad-\mathrm{D}_x\eta(u,Y_u)Y_u'(\mathbb{W}^T_{u,v}-2\mathrm{Sym}(\mathbb{W}_{u,v}))+\mathrm{o}(\abs{v-u})\bigg)\\
        &\quad+\bigg(\int_u^v \eta(w,Y_w)\,\mathrm{d}\W_w+\int_u^v\eta_0(w,Y_w)\,\mathrm{d}w-\mathrm{D}_x\eta(u,Y_u)Y_u'\mathbb{W}_{u,v}+\mathrm{o}(\abs{v-u})\bigg)\bigg)\\
        &\overset{(3)}{=}\sum_{[u,v]\in\mathrm{P}} \int_u^v\eta(w,Y_w)+\mathrm{D}_xH(w,Y_w)\Upsilon_w\,\mathrm{d}\W_w\\
        &\quad+\sum_{[u,v]\in\mathrm{P}}\int_u^v\eta_0(w,Y_w)+\mathrm{D}_xH(w,Y_w)\Upsilon_0(w)\,\mathrm{d}w\\
        &\quad+\sum_{[u,v]\in\mathrm{P}}\mathrm{D}_x\eta(u,Y_u)Y_u'[\W]_{u,v}+\frac{1}{2}\mathrm{D}_x^2H(u,Y_u)(Y_u'\otimes Y_u')[\W]_{u,v}\\
        &\quad+\sum_{[u,v]\in\mathrm{P}}\mathrm{o}(\abs{v-u})\\
        &\xrightarrow{(4)}\int_0^t \eta(w,Y_w)+ \mathrm{D}_x H(w,Y_w)\Upsilon_w\,\mathrm{d}\W_w\\
        & \quad+\int_0^t\eta_0(w,Y_w)+\mathrm{D}_xH(w,Y_w)\Upsilon_0(w)\,\mathrm{d}w\\
        &\quad +\int_0^t\mathrm{D}_x\eta(w,Y_w)Y_w'\,\mathrm{d}[\W]_w +\frac{1}{2}\int_0^t\mathrm{D}_x^2H(w,Y_w)(Y_w'\otimes Y_w')\,\mathrm{d}[\W]_w,
    \end{align*}
    as desired. In (3), we employed linearity of the controlled integration, as well as $\mathbb{W}_{u,v}+\mathbb{W}^T_{u,v}=2\mathrm{Sym}(\mathbb{W}_{u,v}) $. Last, in (4), we summed over all controlled Gubinelli integrals and took the limit $|\mathrm{P}|\rightarrow0$ to obtain the Young integrals, as well as $\sum_{[u,v]\in\mathrm{P}}\mathrm{o}(\abs{v-u})\rightarrow0 $.
\end{proof}
     \subsection{Random Dynamical Systems from a Rough Point of View}
            \begin{definition}[{\cite[Definition 2]{BAILLEUL_riedel_20175792}}]
            Let $\metricDynamicalSystem$ be a metric dynamical system. A random rough path $(\W(\omega))_\omega\subseteq\mathscr{C}^\alpha(\R,\V)$ is called a rough path cocycle if $\W_{t_1+s,t_2+s}(\omega)=\W_{t_1,t_2}(\theta_s\omega)$ holds true for every $\omega\in\Omega$ and all $s,t_1,t_2\in\R$, $t_1<t_2$.

     When the randomness is known from the context, we sometimes just say that $\W$ is a rough path cocycle.
        \end{definition}

 \begin{notation}
     Henceforward, we will almost exclusively work with random rough paths $(\W(\omega))_\omega\subseteq\driverRoughPath{\R}{\R^M}$. While we will constantly indicate the random input $\omega$ in the driver $\W(\omega)\in\driverRoughPath{\R}{\R^M}$ in the notation, we will not display it always explicitly in the context of controlled rough paths $\Y\in\controlledRoughPath{\W(\omega)}{\R}{\linearMaps{\R^M}{\R^D}}$. The dependence of $\Y=\Y(\omega)$ upon $\omega$ will stem from the controlling driver $\W(\omega)$ and will always be clear from the context.
 \end{notation}

We start this section by generalizing a statement discussed in \cite[Lemma 3.7]{kuehn2021roughcentermanifolds} and \cite[Theorem 19]{BAILLEUL_riedel_20175792} that deals with the fact that RDEs driven by rough path cocycles generate, similar to the situation with SDEs, a random dynamical system even on all of $\R$. But before, we must quote a useful time-shift rule for controlled integration that shall be featured multiple times throughout this work.

\begin{lemma}[Time-Shift Rule for Integration; {\cite[Lemma 3.7]{kuehn2021roughcentermanifolds}}]\label{rde lem time-shift}
Let $(\W(\omega))_{\omega}\subseteq \driverRoughPath{\R}{\R^M}$ be a rough path cocycle and let $\Z\in\controlledRoughPath{\W(\omega)}{\R}{\R^D}$ and $\U\in\controlledRoughPath{\W(\omega)}{\R}{\linearMaps{\R^M}{\R^D}}$. Then, for any $s,t\in\R$, we have
\begin{align*}
    \begin{cases}
        \int_0^t Z_{u+s}\,\mathrm{d}u=\int_s^{t+s}Z_u\,\mathrm{d}u, &t\in\R_{\geq0}\\
        \int_t^0Z_{u+s}\,\mathrm{d}u=\int_{t+s}^sZ_u\,\mathrm{d}u,&t\in\R_{<0}
    \end{cases},
\end{align*}
and also
\begin{align*}
    \begin{cases}
        \int_0^t U_{u+s}\,\mathrm{d}\W_u(\theta_{s}\omega)=\int_s^{t+s}U_u\,\mathrm{d}\W_u(\omega), &t\in\R_{\geq0}\\
        \int_t^0U_{u+s}\,\mathrm{d}\W_u(\theta_s\omega)=\int_{t+s}^sU_u\,\mathrm{d}\W_u(\omega),&t\in\R_{<0}
    \end{cases}.
\end{align*}
\end{lemma}

\begin{lemma}\label{rde Lem RDEs generate cocycles}
    Let $(\W(\omega))_{\omega}\subseteq \driverRoughPath{\R}{\R^M}$ be a rough path cocycle. Then, the RDE  $\,\mathrm{d}Y_t=G_0(Y_t)\,\mathrm{d}t+G(Y_t)\,\mathrm{d}\W_t(\omega)$ with $Y_0=y_0\in\R^D$ generates a cocycle on the entire time axis $\R$. Here, $G_0\in  C^\infty(\R^D,\R^D)\cap C^3_{\mathrm{b}}(\R^D,\R^D)$ and $G\in  C^\infty(\R^D,\linearMaps{\R^M}{\R^D})\cap C^3_{\mathrm{b}}(\R^D,\linearMaps{\R^M}{\R^D})$ enjoy the following properties taken from \cite[Section 3.1]{mildGronwall2025} that ensure global solubility:
    \begin{enumerate}
        \item The drift $G_0$ is Lipschitz and linearly growing.
        \item The derivative of $\mathrm{D}G(\cdot)G(\cdot)$ is bounded.
    \end{enumerate}
\end{lemma}
\begin{proof}
Denote the solution flow of the RDE for times $\in\R_{\geq0}$ by $\phi$. It was shown on several occasions that $\phi:\R_{\geq0}\times \Omega\times \R^D\to\R^D$ is a cocycle, see, for instance, \cite[Theorem 19]{BAILLEUL_riedel_20175792}, \cite[Lemma 3.7]{kuehn2021roughcentermanifolds} and \cite[Lemma 5.8]{mildGronwall2025}.
    Thus, the task upon us is to extend $\phi$ to the negative axis $\R_{<0}$.

 Remembering the connection between RDEs in Lyons's original sense and in Gubinelli's sense, as described in \cite[Section 8.8]{FrizHairer2020}, we know from \cite[Proposition 11.11]{Friz_Victoir_2010} that $\phi(t,\omega)$ is a $C^\infty$-diffeomorphism for every $t\in\R_{\geq0}$ and every $\omega$. Hence, reviewing \cite[Theorem 1.1.6]{Arnold1998RDS}, it is natural to extend $\phi$ to the negative axis by defining $\phi(t,\omega):=\phi(-t,\theta_t\omega)^{-1}$ for each $t\in\R_{<0}$ and every $\omega$, which gives rise to a map $\phi:\R\times\Omega\times\R^D\to\R^D$. Let us verify that the extended $\phi$ remains a cocycle.

    First, we study the cocycle property, dividing the $(s,t)$-plane into six areas:
    \begin{comment}
    as illustrated in the figure below.
    \begin{center}
        \begin{tikzpicture}[scale=0.5]

        \draw[thick][->] (-5,0) -- (5,0) node[right] {{\scriptsize$s$}};
        \draw[thick][->] (0,-5) -- (0,5) node[above] {{\scriptsize$t$}};

        \draw[thick] (-4,4) -- (4,-4);

        \node at (2.5,2.5) {{\scriptsize $A_1$}};

        \node at (-2,4) {{\scriptsize$A_2$}};

        \node at (-4,2) {{\scriptsize$A_3$}};

        \node at (4,-2) {{\scriptsize$A_6$}};

        \node at (2,-4) {{\scriptsize$A_5$}};

        \node at (-2.5,-2.5) {{\scriptsize$A_4$}};

        \end{tikzpicture}
    \end{center}
    As seen,
    \end{comment}
    $A_1:=\{(s,t)\in\R^2:s,t\geq 0\}$, $A_2:=\{(s,t)\in\R^2:s\leq0,t\geq 0,t\geq-s\}$, $A_3:=\{(s,t)\in\R^2:s\leq0,t\geq 0,t\leq-s\}$, $A_4:=\{(s,t)\in\R^2:s,t\leq0\}$, $A_5:=\{(s,t)\in\R^2:s\geq0,t\leq0,t\leq -s\}$, and $A_6:=\{(s,t)\in\R^2:s\geq0,t\leq0,t\geq-s\}$.
    We handle each area separately in a case distinction. For $(s,t)\in A_1$, we already know that the cocycle identity holds true. Let now $(s,t)\in A_2$. Using the cocycle identity for pairs $\in A_1$, we compute
    \begin{align*}
    \phi(t,\theta_s\omega)\circ \phi(s,\omega)&=\phi(t+s-s,\theta_s\omega)\circ \phi(-s,\theta_s\omega)^{-1}\\
    &=\phi(t+s,\theta_{-s+s}\omega)\circ \phi(-s,\theta_s\omega)\circ\phi(-s,\theta_s\omega)^{-1}=\phi(t+s,\omega).
        \end{align*}
    We work out the remaining cases for pairs $\in A_3,A_4,A_5,A_6$ in a similar fashion.

    At last, we realize that the measurability and continuity matters are straightforward consequences of the respective properties of $\phi|_{t\in\R_{\geq0}}$.
    \begin{comment}
    , when analyzing the partial maps
    \begin{align*}
        (t,\omega,y_0)\mapsto(-t,\theta_t\omega,y_0)\mapsto \phi(-t,\theta_{t}\omega,y_0)\mapsto\phi(-t,\theta_t\omega)^{-1}(y_0).
    \end{align*}
    \end{comment}
    This finishes the proof.
\end{proof}

It turns out to be quite fruitful to further investigate the form of $\phi$ instead of leaving it in its abstract definition given in Lemma \ref{rde Lem RDEs generate cocycles}. Of course, $\phi|_{t\in\R_{\geq 0}}$ corresponds to solution flows of the RDE, but, a priori, we do not have a meaningful interpretation of $\phi|_{t\in\R_{<0}}$ with regard to the RDE. Let us close this gap.
\begin{lemma}\label{rde Lem rdes cocycles but mild formulation}
    Under the conditions of Lemma \ref{rde Lem RDEs generate cocycles}, we have a mild formulation of $\phi:\R\times\Omega\times\R^D\to\R^D$ on all of $\R$ given as
    \begin{align*}
         \phi(t,\omega,y_0)= \begin{cases}
             y_0+\int_0^t G_0(\phi(u,\omega,y_0))\,\mathrm{d}u+\int_0^t G(\phi(u,\omega,y_0))\,\mathrm{d}\W_u(\omega), &t\in\R_{\geq 0}\\
             y_0-\int_t^0 G_0(\phi(u,\omega,y_0))\,\mathrm{d}u-\int_t^0 G(\phi(u,\omega,y_0))\,\mathrm{d}\W_u(\omega), &t\in\R_{<0}
         \end{cases}.
    \end{align*}
    That means that, for times $t\in\R_{<0}$, we set $\phi(t,\omega,y_0)$ to be the unique position in space such that $\phi(t,\omega,y_0)+\int_t^0 G_0(\phi(u,\omega,y_0))\,\mathrm{d}u+\int_t^0 G(\phi(u,\omega,y_0))\,\mathrm{d}\W_u(\omega)=y_0$.
\end{lemma}
\begin{proof}
    For non-negative times $\in\R_{\geq0}$, the statement is clear by the integral formulation of the solution flow. So, let $t\in\R_{<0}$. We check that the RHS evaluated at $\phi(-t,\theta_t\omega,y_0)$ equals $y_0$, so that the RHS really is $\phi(-t,\theta_t\omega)^{-1}=\phi(t,\omega)$, as defined in Lemma \ref{rde Lem RDEs generate cocycles}. Indeed, we calculate
    \begin{align*}
        \text{RHS}|_{\phi(-t,\theta_t\omega,y_0)}
        &= \phi(-t,\theta_t\omega,y_0)-\int_t^0 G_0(\phi(u,\omega,\phi(-t,\theta_t\omega,y_0)))\,\mathrm{d}u\\
        &\quad-\int_t^0 G(\phi(u,\omega,\phi(-t,\theta_t\omega,y_0)))\,\mathrm{d}\W_u(\omega)\\
        &\overset{(1)}{=}y_0+\int_0^{-t}G_0(\phi(u,\theta_t\omega,y_0))\,\mathrm{d}u+\int_0^{-t}G(\phi(u,\theta_t\omega,y_0))\,\mathrm{d}\W_u(\theta_t\omega)\\
        &\quad-\int_t^0 G_0(\phi(u,\omega,\phi(-t,\theta_t\omega,y_0)))\,\mathrm{d}u-\int_t^0 G(\phi(u,\omega,\phi(-t,\theta_t\omega,y_0)))\,\mathrm{d}\W_u(\omega)\\
        &\overset{(2)}{=}y_0+\int_t^{0}G_0(\phi(u-t,\theta_t\omega,y_0))\,\mathrm{d}u+\int_t^{0}G(\phi(u-t,\theta_t\omega,y_0))\,\mathrm{d}\W_u(\omega)\\
        &\quad-\int_t^0 G_0(\phi(u-t,\theta_t\omega,y_0))\,\mathrm{d}u-\int_t^0 G(\phi(u-t,\theta_t\omega,y_0))\,\mathrm{d}\W_u(\omega)=y_0,
    \end{align*}
    where, in (1), we wrote down the mild formulation of the solution flow $\phi|_{t'\in\R_{\geq0}}$, and, in (2), we employed the time-shift rule from Lemma \ref{rde lem time-shift} for the first two integrals and used the fact that $\phi$ enjoys the cocycle property, as proven in Lemma \ref{rde Lem RDEs generate cocycles}, in the last two integrands.
\end{proof}

\begin{remark}\label{rde rem regarding generation of cocycles and backward integration}
    We point out an interesting detail: Upon comparing Lemma \ref{rde Lem rdes cocycles but mild formulation} and \cite[Theorem 2.3.39]{Arnold1998RDS}, we find that the backward integration at the heart of the latter suddenly vanishes in the former, although we would expect both theories to be more or less parallel in our discussion. Yet, we also know from \cite[Section 5.4]{FrizHairer2020} that the time-direction in controlled integration has less of an impact on the outcome, because the second-order process brings in robust extra information missing in classical stochastic integration. We elaborate on this topic in Appendix \ref{appendix rde further details on generation and backward integration}. We explain that the mentioned issue only exists when conducting a superficial analysis, whereas upon closer inspection, we realize that there is no real difference.
\end{remark}

Note that linear RDEs generate linear cocycles. Therefore, it makes sense to be interested in whether the MET, Theorem \ref{thm MET}, is applicable to such cocycles. But, before we get to the link to the MET, let us shortly deviate to another important tool: Duhamel's formula, also known as the variation-of-constants formula, for solutions to affine differential equations. Now that we know about the existence of suitable fundamental matrices for linear RDEs, let us derive a version of Duhamel's formula for affine RDEs.

\begin{remark}\label{rde rem matrix-valued rdes}
    Let $(\W(\omega))_\omega\subseteq\driverRoughPath{\R}{\R^M}$ be a rough path cocycle, and let $B_0\in\linearMaps{\R^D}{\R^D}$ and $B\in\linearMaps{\R^D}{\linearMaps{\R^M}{\R^D}}$.
    Throughout this work, we will frequently encounter matrix-valued RDEs of the form $\,\mathrm{d}J_t=B_0J_t\,\mathrm{d}t+BJ_t\,\mathrm{d}\W_t(\omega)$, that have a phase space given by $\linearMaps{\R^D}{\R^D}$, on a regular basis. The importance of these equations results from the observation that linear cocycles $\Phi$ generated by linearized equations $\,\mathrm{d}Y_t=B_0Y_t\,\mathrm{d}t+BY_t\,\mathrm{d}\W_t(\omega)$ actually give rise to solutions of the above matrix-valued RDE themselves. However, before showing this, we must understand the equation, which requires careful interpretation: While the drift $B_0J_t\in\linearMaps{\R^D}{\R^D}$ is harmless, we must remember $\linearMaps{\R^D}{\linearMaps{\R^M}{\R^D}}\simeq \linearMaps{\R^D\otimes\R^M}{\R^D}$ in order to give a meaning to $BJ_t\in\linearMaps{\R^M}{\linearMaps{\R^D}{\R^D}}$ via $BJ_tv:=B(\cdot\otimes v)J_t$ for $v\in\R^M$.  More elegantly, we understand $B\in\linearMaps{\R^D}{\linearMaps{\R^M}{\R^D}}\simeq\linearMaps{\R^M}{\linearMaps{\R^D}{\R^D}}$ via $Bv:=\sum_{m=1}^MB_mv^m$, where $B_m\in\linearMaps{\R^D}{\R^D}$ for $m=1,\dots,M$, so that
    \begin{align*}
        BJ_tv=B(\cdot\otimes v)J_t=\sum_{m=1}^MB_mJ_tv^m.    \end{align*}
    \begin{comment}
    Now, by definition, for any $y_0\in\R^D$ and for all $t\in[0,T]$, we have the identity
    \begin{align*}
        0&=\Phi(t,\omega)y_0-y_0-\int_0^tB_0\Phi(s,\omega)y_0\,\mathrm{d}s-\int_0^tB\Phi(s,\omega)y_0\,\mathrm{d}\W_s(\omega)\\
        &=\left(\Phi(t,\omega)-\id-\int_0^tB_0\Phi(s,\omega)\,\mathrm{d}s-\int_0^tB\Phi(s,\omega)\,\mathrm{d}\W_s(\omega)\right)y_0,
    \end{align*}
    which is readily verified due to linearity properties and the fact that $y_0$ is independent of time. This already implies
    \begin{align*}
        \Phi(t,\omega)=\id+\int_0^tB_0\Phi(s,\omega)\,\mathrm{d}s+\int_0^tB\Phi(s,\omega)\,\mathrm{d}\W_s(\omega).
    \end{align*}
    Of course, the argument works the other way around, too,
    \end{comment}
    This quickly results in a one-to-one correspondence between linear vector-valued and matrix-valued RDEs via the fundamental matrix $\Phi$. It is clear that the corresponding solution to $\,\mathrm{d}J_t=B_0J_t\,\mathrm{d}t+BJ_t\,\mathrm{d}\W_t(\omega)$ with $J_0=\id$ written out properly is \begin{align*}\Upphi(\cdot,\omega):=
    (\Phi(\cdot,\omega),B\Phi(\cdot,\omega))\in\controlledRoughPath{\W(\omega)}{[0,T]}{\linearMaps{\R^D}{\R^D}}.\end{align*}
\end{remark}

    \begin{lemma}[The Inverse of a Linear Cocycle Solves an RDE]\label{rde lem phi inverse is also controlled rough path}
Let $(\W(\omega))_\omega\subseteq\driverRoughPath{\R}{\R^M}$ be a weakly geometric rough path cocycle.
        In addition, let $B_0\in\linearMaps{\R^D}{\R^D}$ and $B\in\linearMaps{\R^{D}}{\linearMaps{\R^M}{\R^{D}}}$, and denote by $\Phi$ the linear cocycle generated by the linear RDE $\,\mathrm{d}Y_t=B_0Y_t\,\mathrm{d}t+B Y_t\,\mathrm{d}\W_t(\omega)$ in $\R^{D}$. Assume that $\Phi(t,\omega)$ is invertible for every $t\in[0,T]$. Then, the following statements hold true:
        \begin{enumerate}
            \item We have $t\mapsto(\Phi(t,\omega)^{-1}, -\Phi(t,\omega)^{-1}\Phi'(t,\omega)\Phi(t,\omega)^{-1})\in\controlledRoughPath{\W(\omega)}{[0,T]}{\linearMaps{\R^{D}}{\R^{D}}}$. Here, we interpret $\Phi(t,\omega)^{-1} \Phi'(t,\omega) \Phi(t,\omega)^{-1}\in\linearMaps{\R^M}{\linearMaps{\R^{D}}{\R^{D}}}$ via \\$\Phi(t,\omega)^{-1} \Phi'(t,\omega) \Phi(t,\omega)^{-1}v:=\Phi(t,\omega)^{-1} \Phi'(t,\omega)v \Phi(t,\omega)^{-1}$ for $v\in\R^M$.
            \item This controlled rough path is the solution to the linear matrix-valued RDE
            \begin{align*}
            \mathrm{d}J_t=-J_tB_0\,\mathrm{d}t-J_tB\,\mathrm{d}\W_t(\omega)
            \end{align*} in $\linearMaps{\R^{D}}{\R^{D}}$ with initial condition $J_0=\id$, where $J_tB\in\linearMaps{\R^M}{\linearMaps{\R^D}{\R^D}}$ via $J_tBv:=J_tB(\cdot\otimes v)=J_t\sum_{m=1}^MB_mv^m$ for $v\in\R^M$ and suitable $B_m\in\linearMaps{\R^D}{\R^D}$, $m=1,\dots,M$.
            \item Moreover, we have \begin{align*}t\mapsto\left(\Phi(t,\omega)^{-T},-B^T\Phi(t,\omega)^{-T}\right)\in\controlledRoughPath{\W(\omega)}{[0,T]}{\linearMaps{\R^D}{\R^D}},
            \end{align*}
            where $B^T:=(B_1^T,\dots,B^T_M)$.
        \item This controlled rough path solves the linear matrix-valued RDE \begin{align*}\,\mathrm{d}J_t=-B_0^TJ_t\,\mathrm{d}t-B^TJ_t\,\mathrm{d}\W_t(\omega)\end{align*} with initial condition $J_0=\id$, which lives in $\linearMaps{\R^D}{\R^D}$. Here, $B^TJ_t\in\linearMaps{\R^M}{\linearMaps{\R^D}{\R^D}}$ via $B^TJ_tv:=B^T(\cdot\otimes v)J_t=\sum_{m=1}^MB_m^Tv^mJ_t$ for $v\in\R^M$.
        \end{enumerate}
    \end{lemma}
    \begin{proof}
        Assertions \textit{(i)} and \textit{(iii)} follow immediately from composition of controlled rough paths with regular functions as in Lemma \ref{rde lem composing with regular fct}, whereas \textit{(ii)} and \textit{(iv)} are tedious but routine calculations; a key hint is \cite[Remark 8.14]{FrizHairer2020}.
    \end{proof}

    \begin{corollary}[Rough Duhamel Formula]\label{rde corollary duhamels formula}        Let $(\W(\omega))_\omega\subseteq\driverRoughPath{\R}{\R^M}$ be a weakly geometric rough path cocycle.
        We consider the non-autonomous affine RDE \begin{align*}
        \mathrm{d}Y_t=(B_0Y_t+Z_t)\,\mathrm{d}t+(BY_t+U_t)\,\mathrm{d}\W_t(\omega)\end{align*} 
        in $\R^D$ with initial condition $Y_0=y_0\in\R^D$, given linear maps $B_0\in\linearMaps{\R^D}{\R^D}$ and $B\in\linearMaps{\R^D}{\linearMaps{\R^M}{\R^D}}$, and $\Z\in\controlledRoughPath{\W(\omega)}{[0,T]}{\R^D}$ and $\U\in \controlledRoughPath{\W(\omega)}{[0,T]}{\linearMaps{\R^{M}}{\R^D}}$. Denote by $\Phi$ the linear cocycle generated by the linearized RDE $\,\mathrm{d}Y_t=B_0Y_t\,\mathrm{d}t+B Y_t\,\mathrm{d}\W_t(\omega)$, and by $\phi(\cdot,\omega)$ the solution flow generated by the initial affine RDE. If $\Phi(t,\omega)$ is invertible for all $t\in[0,T]$, then we have the formula
        \begin{align*}
           \phi(\cdot,\omega)y_0=\Phi(\cdot,\omega)y_0+\Phi(\cdot,\omega)\int_0^\cdot\Phi(s,\omega)^{-1}Z_s\,\mathrm{d}s+\Phi(\cdot,\omega)\int_0^\cdot\Phi(s,\omega)^{-1}U_s\,\mathrm{d}\W_s(\omega).
        \end{align*}
    \end{corollary}
    \begin{proof}
        To begin with, note that $\Upphi(\cdot,\omega)^{-1}\U\in\controlledRoughPath{\W(\omega)}{[0,T]}{\linearMaps{\R^M}{\R^D}}$ and $\Upphi(\cdot,\omega)^{-1}\Z\in\controlledRoughPath{\W(\omega)}{[0,T]}{\R^D}$ according to Lemma \ref{rde lem phi inverse is also controlled rough path} and  Lemma \ref{rde lem product rule}, where the controlled rough path $\Upphi(\cdot,\omega)^{-1}\in\controlledRoughPath{\W(\omega)}{[0,T]}{\linearMaps{\R^D}{\R^D}}$ is explained in Lemma \ref{rde lem phi inverse is also controlled rough path}.
        We intend to apply Theorem \ref{rde lem ito-wentzell}. To this end, set \begin{align*}
            \mathcal{X}:=\left(y_0+\int_0^\cdot\Phi(s,\omega)^{-1}Z_s\,\mathrm{d}s+\int_0^\cdot\Phi(s,\omega)^{-1}U_s\,\mathrm{d}\W_s(\omega),\Phi(\cdot,\omega)^{-1}U\right)\in\controlledRoughPath{\W(\omega)}{[0,T]}{\R^D}.
        \end{align*}
        In addition, let $H:[0,T]\times\R^D\rightarrow\R^D,\,(t,x)\mapsto\Phi(t,\omega)x$, so that, for every $x\in\R^D$, the controlled rough path $t\mapsto(H(t,x),\dot H(t,x))$ satisfies the mild solution formulation
        \begin{align*}
            H(t,x)=x+\int_0^tB_0H(s,x)\,\mathrm{d}s+\int_0^tB H(s,x)\,\mathrm{d}\W_s(\omega).
        \end{align*}
        It is easy to see that all the requirements imposed in Theorem \ref{rde lem ito-wentzell} are met. Therefore, for all $t\in[0,T]$, we obtain, by definition of $H$, that
        \begin{align*}
            \Phi(t,\omega)X_t=H(t,X_t)
            &=y_0+\int_0^tB\Phi(s,\omega)X_s+\Phi(s,\omega)\Phi(s,\omega)^{-1}U_s\,\mathrm{d}\W_s(\omega)\\
            &\quad+\int_0^tB_0\Phi(s,\omega)X_s+\Phi(s,\omega)\Phi(s,\omega)^{-1}Z_s\,\mathrm{d}s\\
            &=y_0+\int_0^tB_0\Phi(s,\omega)X_s+Z_s\,\mathrm{d}s+\int_0^tB\Phi(s,\omega)X_s+U_s\,\mathrm{d}\W_s(\omega).
        \end{align*}
        Thus, the Duhamel representation indeed provides a solution to the initial affine RDE. We conclude by the fact that linear-growth RDEs possess unique solutions.
    \end{proof}

    \begin{remark}
        All of the above results are easily extended to integrals that take place on the non-positive axis $\R_{\leq 0}$. For example, we highlight that, in analogy to Lemma \ref{rde Lem rdes cocycles but mild formulation}, Duhamel's formula as in Corollary \ref{rde corollary duhamels formula} possesses for $t\in\R_{<0}$ the version
        \begin{align*}
            \phi(t,\omega)y_0=\Phi(t,\omega)y_0-\Phi(t,\omega)\int_t^0\Phi(s,\omega)^{-1}Z_s\,\mathrm{d}s-\Phi(t,\omega)\int_t^0\Phi(s,\omega)^{-1}U_s\,\mathrm{d}\W_s(\omega).
        \end{align*}
    \end{remark}
 With Duhamel's formula established, let us return to the connection between the linear cocycle generated by a linear RDE and the MET.

\begin{lemma}\label{rde lem linear RDE cocycle atuomatically satisfy IC of MET}
    Let $(\W(\omega))_\omega\subseteq\driverRoughPath{\R}{\R^M}$ be a weakly geometric rough path cocycle that satisfies $\expectationPathwise{\normCalpha{\W(\omega)}^{1/\alpha}}<\infty$ on compact intervals.
    Furthermore, let $\Phi$ be a linear cocycle generated by the linear RDE $\,\mathrm{d}Y_t=B_0Y_t\,\mathrm{d}t+BY_t\,\mathrm{d}\W_t(\omega)$ in $\R^D$, where $B_0\in\linearMaps{\R^D}{\R^D}$ and $B\in\linearMaps{\R^D}{\linearMaps{\R^M}{\R^D}}$. Then, $\Phi$ meets the integrability condition of the MET, Theorem \ref{thm MET}, that is,
    \begin{align*}
        \expectationPathwise{\sup_{t\in[0,1]}\log^+\abs{\Phi(t,\omega)}+\sup_{t\in[0,1]}\log^+\abs{\Phi(t,\omega)^{-1}}} <\infty.
    \end{align*}
    In particular, the MET automatically applies to $\Phi$.
\end{lemma}

This extremely convenient fact should not come as a surprise, however, for this result mirrors the SDE situation exactly. In that situation, any linear cocycle generated by a linear SDE instantly satisfies the integrability condition of the MET, as well, see \cite[Remark 6.2.12]{Arnold1998RDS}.

\begin{proof}[Proof of Lemma \ref{rde lem linear RDE cocycle atuomatically satisfy IC of MET}]
    The proof consists of an application (rather, two applications) of the rough Gronwall statement \cite[Proposition 8.13]{FrizHairer2020}. However, first, we need to transform the RDE to work with a driftless equation. To this end, we perform the standard enlargement of the driving signal by embedding time as a new entry to end up with a new driver rough path $\overline{\W}\in\driverRoughPath{[0,1]}{\R^{1+M}}$ and the new driftless RDE
    \begin{align*}
        \mathrm{d}Y_t=(B_0,B)Y_t\,\mathrm{d}\Wbar_t(\omega),
    \end{align*}
    which is equivalent to the original RDE with drift.
\end{proof}

    This completes the section on necessary preliminaries required as a basic setting to even start to analyze rough normal forms.

\section{Normal Form Theory for Rough Differential Equations}\label{sec:rde-normal-forms}
    In the rest of this work,  $(\W(\omega))_\omega\subseteq\driverRoughPath{\R}{\R^M}$ denotes a fixed enhanced centered Gaussian rough path that is also a weakly geometric rough path cocycle. For details on the construction and central properties, we refer to the classic monographs \cite[Chapter 10]{FrizHairer2020} and \cite[Chapter 15]{Friz_Victoir_2010}, and to the recent work \cite[Section 3.3]{mildGronwall2025}. For instance, the reader might think of enriched fractional Brownian motion with Hurst parameter $\in(1/3,1/2]$ as a prominent example.

Recall from the classical construction that the underlying probability space $(\Omega,\F,P)$ is given by the two-sided Wiener space, that is, $\Omega=C^0(\R,\R^M)$, $\F$ is the Borel $\sigma$-algebra and $P$ is the probability measure determined by finite-dimensional Gaussian distributions. In addition, we equip $(\Omega,\F,P)$ with the filtration $(\F^t_{-\infty})_{t\in\R} $, where, for any $t\in\R$, the $\sigma$-algebra $\F^t_{-\infty} $ denotes the completion of $\sigma(\{W_s:s\in(-\infty,t]\})$, and with the base shift $\theta:\R\times\Omega\to\Omega,\,(t,\omega)\mapsto\theta_t\omega:=\omega(\cdot+t)-\omega(t)$. After, if necessary, restricting the filtered probability space $(\Omega,\F,(\F^t_{-\infty})_{t\in\R},P)$, we are allowed to speak about all $\omega$ instead of almost all $\omega$.
Keep in mind also the important aspect from the construction of the enhancement that it preserves adaptedness, that is, $\mathbb{W}_{s,t}$ and, hence, $\W_t$ is $\F^t_{-\infty}$-measurable for every $s,t\in\R$ with $s<t$, so the second-order process $\mathbb{W}$ and, thus, the entire rough path $\W$ is non-anticipative.\\

Finally, we officially rewind the central object our work revolves around.
    \begin{definition}[{\cite[Definition 8.2.1]{Arnold1998RDS}}]\label{sde def random coordinate }
            We name a measurable map $H:\Omega\times \R^D\rightarrow \R^D$ a random coordinate transformation if $H(\omega,\cdot)$ is a $C^\infty$-diffeomorphism and $H(\omega,0)=0$ and $\mathrm{D}_xH(\omega,0)=\id$ hold true for almost every $\omega\in\Omega$.
        \end{definition}

In the following, we assume that the reader is familiar with the discussion of the SDE case presented in \cite[Section 8.5]{Arnold1998RDS}. Keep in mind that we intend to parallel the SDE discussion, only accounting for additional technical obstacles. Thus, our first target is to transform the rough normal form problem into the examination of a certain hierarchical system of affine RDEs and its invariant measures.

    \subsection{Arriving at the Hierarchical System of Affine Rough Differential Equations}

 We fix drift and diffusion coefficient functions
 $F_0\in C^\infty(\R^D,\R^D)\cap C^3_{\mathrm{b}}(\R^D,\R^D)$ and
     $F\in C^\infty(\R^D,\linearMaps{\R^{M}}{\R^D})\cap C^3_{\mathrm{b}}(\R^D,\linearMaps{\R^{M}}{\R^D})$ with $F_0(0)=0$ and $F(0)=0$,
     and consider the generic random RDE
    \begin{align*}
        \mathrm{d}Y_t=F_0(Y_t)\,\mathrm{d}t+F(Y_t)\,\mathrm{d}\W_t(\omega)
        \mytag\label{rde tag standard RDE}
    \end{align*}
    in $\R^D$. As a reminder, in analogy to the SDE situation \cite[Remark 2.3.27]{Arnold1998RDS} and to Lemma \ref{rde Lem rdes cocycles but mild formulation}, the RDE (\ref{rde tag standard RDE}) is to be read as
    \begin{align*}
        Y_t=\begin{cases}
            Y_0+\int_0^tF_0(Y_s)\,\mathrm{d}s+\int_0^tF(Y_s)\,\mathrm{d}\W_s(\omega),&t\in\R_{\geq0}\\
            Y_0-\int_t^0F_0(Y_s)\,\mathrm{d}s-\int_t^0F(Y_s)\,\mathrm{d}\W_s(\omega),&t\in\R_{<0}
        \end{cases}.
    \end{align*}
    As $\W$ is a rough path cocycle, according to Lemma \ref{rde Lem RDEs generate cocycles}, the RDE (\ref{rde tag standard RDE}) generates a cocycle $\phi$ on $\R$ over the metric dynamical system $\metricDynamicalSystem$.
    In addition to that, let $\psi$ be a cocycle generated by the linearized RDE $\,\mathrm{d}Y_t=\mathrm{D}F_0(0)Y_t\,\mathrm{d}t+\mathrm{D}F(0)Y_t\,\mathrm{d}\W_t(\omega)$. While full linearization represents an ambitious and optimistic target, we will deduce that under certain circumstances this wish is indeed realistic. Furthermore, this simplifies the problem: If we did not anticipate full linearization, $\psi$ would, a priori, only be generated by some RDE such as $\,\mathrm{d}Y_t=G_0(\theta_t\omega,Y_t)\,\mathrm{d}t+G(\theta_t\omega,Y_t)\,\mathrm{d}\W_t(\omega)$, where the randomness of drift and diffusion introduces additional difficulties.
    Our normal form procedure seeks a random coordinate transform $H:\Omega\times\R^D\rightarrow\R^D$ that locally conjugates $\phi$ and $\psi$, so our objective is to find some random coordinate transform $H$ that fulfills
    \begin{align*}
        \phi(t,\omega,H(\omega,x))=H(\theta_t\omega,\psi(t,\omega,x)) 
        \mytag\label{rde tag conjugate first}
    \end{align*}
    locally in a neighborhood of the origin.
    In particular, $H $ should be almost surely smooth in space, so we work with a formal Taylor expansion
    \begin{align*}
        H(\omega,x)\formallyEqual x+\sum_{n=2}^\infty H^n(\omega,x),
    \end{align*}
    where $H^n(\omega)\in\homogeneousPolynomials{\R^D}$.
    Let us make a bold guess\footnote{This idea does not come out of thin air. In fact, we consider the direct analogue to the SDE case.} about the nature of the $H^n(\omega)$: We postulate that the $H^n(\omega)$ give rise to stationary solutions to the RDEs
    \begin{align*}
        \mathrm{d}H^n(\theta_t\omega)&=\big(\ad \mathrm{D}F_0(0)H^n(\theta_t\omega)+K^n_0(\theta_t\omega)\big)\,\mathrm{d}t\\
        &\quad+\big(\ad \mathrm{D}F(0)H^n(\theta_t\omega)+K^n(\theta_t\omega)\big)\,\mathrm{d}\W_t(\omega)\\
        &=:\Gamma_0^n(\theta_t\omega,H^n(\theta_t\omega))\,\mathrm{d}t+\Gamma^n(\theta_t\omega,H^n(\theta_t\omega))\,\mathrm{d}\W_t(\omega
        )\mytag\label{rde tag bold guess}
    \end{align*}
    that live in $\homogeneousPolynomials{\R^D}\simeq :\R^{D_n}$, for each $n\in\N_{\geq2}$. Here, the linear maps $\mathrm{ad}_n\mathrm{D}F_0(0)$ and $\mathrm{ad}_n\mathrm{D}F(0)$ are given by
    \begin{align*}
        \mathrm{ad}_n\mathrm{D}F_0(0):H_{n,D}(\R^D)\rightarrow H_{n,D}(\R^D),\; H \mapsto \mathrm{D}F_0(0)H-\mathrm{D}_xH\mathrm{D}F_0(0),
    \end{align*}
    where $\mathrm{D}F_0(0)H-\mathrm{D}_xH\mathrm{D}F_0(0)=(x\mapsto \mathrm{D}F_0(0)H(x)-\mathrm{D}_xH(x)\mathrm{D}F_0(0)x)$, and
    \begin{align*}
        \mathrm{ad}_n\mathrm{D}F(0):H_{n,D}(\R^D)\rightarrow H_{n,D}(\linearMaps{\R^M}{\R^D}),\; H \mapsto \mathrm{D}F(0)H-\mathrm{D}_xH\mathrm{D}F(0),
    \end{align*}
    where $\mathrm{D}F(0)H-\mathrm{D}_xH\mathrm{D}F(0)=(x\mapsto \mathrm{D}F(0)H(x)-\mathrm{D}_xH(x)\mathrm{D}F(0)x)$.
    And, the to-be-specified additive terms are $t\mapsto (K^n_0(\theta_t\omega),\dot K^n_0(\theta_t\omega))\in \controlledRoughPath{\W(\omega)}{\R}{ H_{n,D}(\R^D)}$ and $t\mapsto (K^n(\theta_t\omega),\dot K^n(\theta_t\omega))\in\controlledRoughPath{\W(\omega)}{\R}{ H_{n,D}(\linearMaps{\R^M}{\R^D})} $. It is equivalent to say that, for every $x\in\R^D$, each $H^n(\omega,x)$ solves
    \begin{align*}
        \mathrm{d}H^n(\theta_t\omega,x)&=\big(\ad \mathrm{D}F_0(0)H^n(\theta_t\omega,x)+K^n_0(\theta_t\omega,x)\big)\,\mathrm{d}t\\
        &\quad+\big(\ad \mathrm{D}F(0)H^n(\theta_t\omega,x)+K^n(\theta_t\omega,x)\big)\,\mathrm{d}\W_t(\omega)\\
        &=\Gamma_0^n(\theta_t\omega,H^n(\theta_t\omega,x))\,\mathrm{d}t+\Gamma^n(\theta_t\omega,H^n(\theta_t\omega,x))\,\mathrm{d}\W_t(\omega
        ),
    \end{align*} which instead is an RDE in $\R^D$.
    We emphasize that, in particular, for every $n\in\N_{\geq2}$ and any fixed $x\in\R^D$, we have $t\mapsto\H^n(\theta_t\omega,x)\in\controlledRoughPath{\W(\omega)}{\R}{\R^D}$, with path $t \mapsto H^n(\theta_t\omega,x)=\sum_{\tau\in\N_n^D}H^n_{\tau}(\theta_t\omega)x^\tau$ and  Gubinelli derivative $t\mapsto \dot H^n(\theta_t\omega,x)=\sum_{\tau\in\N_n^D}\dot H^n_{\tau}(\theta_t\omega)x^\tau$. This is significant and special, for it means that the coordinate transform $H$, being composed of the $H^n$, is in fact an enriched object $\H$, constituted of the $\H^n$, itself.\\

    Now, we transition to providing a justification that illustrates the usefulness of our $H^n(\omega)$. We conduct the upcoming computations for $t\in\R_{\geq0}$, noting that the extension to $t\in\R_{<0}$ is analogous. Writing out $H(\theta_t\omega,\psi(t,\omega,x))\formallyEqual\psi(t,\omega,x)+\sum_{n=2}^\infty H^n(\theta_t\omega,\psi(t,\omega,x))$,
    with the mild solution formulation employed twice and Theorem \ref{rde lem ito-wentzell} applied to each $H^n(\omega)$ on the RHS, the conjugate identity (\ref{rde tag conjugate first}) equivalently reads on a formal level
    \begin{align*}
        &\quad H(\omega,x)+\int_0^tF_0(H(\theta_u\omega,\psi(u,\omega,x)))\,\mathrm{d}u + \int_0^tF(H(\theta_u\omega,\psi(u,\omega,x)))\,\mathrm{d}\W_u(\omega)\\
        &\overset{\mathclap{(\ref{rde tag conjugate first})}}{=}H(\omega,x)+\int_0^tF_0(\phi(u,\omega,H(\omega,x)))\,\mathrm{d}u + \int_0^tF(\phi(u,\omega,H(\omega,x)))\,\mathrm{d}\W_u(\omega)\\
        &\overset{\mathclap{(\ref{rde tag conjugate first})}}{=}H(\omega,x)+\int_0^t  \mathrm{D}F_0(0)\psi(u,\omega,x)\,\mathrm{d}u+\int_0^t\mathrm{D}F(0)\psi(u,\omega,x)\,\mathrm{d}\W_u(\omega)\\
        &\quad+\sum_{n=2}^\infty\int_0^t  \Gamma^n(\theta_u\omega,H^n(\theta_u\omega,\psi(u,\omega,x)))+ \mathrm{D}_x H^n(\theta_u\omega,\psi(u,\omega,x))\mathrm{D}F(0)\psi(u,\omega,x)\,\mathrm{d}\W_u(\omega)\\
        &\quad+\sum_{n=2}^\infty\int_0^t \Gamma_0^n(\theta_u\omega,H^n(\theta_u\omega,\psi(u,\omega,x)))+\mathrm{D}_xH^n(\theta_u\omega,\psi(u,\omega,x))\mathrm{D}F_0(0)\psi(u,\omega,x)\,\mathrm{d}u,\mytag\label{rde tag conjugate in Gubinelli level}
    \end{align*}
    where we already inserted $[\W(\omega)]=0$.

    Next, we write $F_0$ and $F$ in terms of their respective formal Taylor expansion:
    \begin{align*}
        F_0(x)\formallyEqual \mathrm{D}F_0(0)x+\sum_{n=2}^\infty\frac{1}{n!}\mathrm{D}^nF_0(0)(x,\dots,x)=: \mathrm{D}F_0(0)x + \sum_{n=2}^\infty F^n_0(x)
    \end{align*}
    and
    \begin{align*}
    F(x)\formallyEqual \mathrm{D}F(0)x+\sum_{n=2}^\infty\frac{1}{n!}\mathrm{D}^nF(0)(x,\dots,x)=:  \mathrm{D}F(0)x+\sum_{n=2}^\infty F^n(x),
    \end{align*}
    with $F^n_0\in H_{n,D}(\R^D)$ and $F^n\in H_{n,D}(\linearMaps{\R^{M}}{\R^D})$.
    Let us insert these, as well as the formal Taylor expansion of $H$, into (\ref{rde tag conjugate in Gubinelli level}) and separate each total order $N\in\N_0$ in the spatial argument $Y_u$, where we abbreviate $\Y:=\boldsymbol{\uppsi}(\cdot,\omega,x)$. Moreover, recalling that all $\Gamma^N_0$ and $\Gamma^N$ are linear in their space component (we purposefully refrain from separating the constant-order additive terms of $\Gamma^N_0$ and $\Gamma^N$ here), and defining 
    $\underline{N}_n:=\{\tau\in\N^n:\sum_{i=1}^n\tau_i=N\}$ for
    $N\in\N_{\geq2}$ and $n=1,\dots,N$, 
    we formally arrive at
    \begin{align*}
        &\quad H(\omega,x)+\int_0^t\mathrm{D}F_0(0)Y_u\,\mathrm{d}u+\sum_{N=2}^\infty \int_0^t\bigg(F^N_0(Y_u) + \sum_{n=2}^{N-1}\sum_{\tau\in\underline{N}_n}F^n_0(H^{\tau_1}(\theta_u\omega,Y_u),\dots,H^{\tau_n}(\theta_u\omega,Y_u))\\&\quad\quad+\mathrm{D}F_0(0)H^N(\theta_u\omega,Y_u)\bigg)\,\mathrm{d}u\\
        &\quad+\int_0^t\mathrm{D}F(0)Y_u\,\mathrm{d}\W_u(\omega) +\sum_{N=2}^\infty \int_0^t\bigg(F^N(Y_u) +\sum_{n=2}^{N-1} \sum_{\tau\in\underline{N}_n} F^n(H^{\tau_1}(\theta_u\omega,Y_u),\dots,H^{\tau_n}(\theta_u\omega,Y_u)) \\ 
        &\quad\quad+\mathrm{D}F(0)H^N(\theta_u\omega,Y_u)\bigg)\,\mathrm{d}\W_u(\omega)\\
        &=H(\omega,x)+\int_0^t\mathrm{D}F_0(0)Y_u\,\mathrm{d}u+\sum_{N=2}^\infty\int_0^t\bigg(\Gamma_0^N(\theta_u\omega,H^N(\theta_u\omega,Y_u))+ \mathrm{D}_xH^N(\theta_u\omega,Y_u)\mathrm{D}F_0(0)Y_u\bigg)\,\mathrm{d}u\\
        &\quad+\int_0^t\mathrm{D}F(0)Y_u\,\mathrm{d}\W_u(\omega)+\sum_{N=2}^\infty\int_0^t\bigg(\Gamma^N(\theta_u\omega,H^N(\theta_u\omega,Y_u))+\mathrm{D}_xH^N(\theta_u\omega,Y_u)\mathrm{D}F(\theta_u\omega,0)Y_u\bigg)\,\mathrm{d}\W_u(\omega).
    \end{align*}
    Of course, this equation holds true - and thus the conjugate identity (\ref{rde tag conjugate first}) is satisfied - if it does order-wise for each $N\in\N_0$.
    The constant order and the linear order are already equal on both sides. After writing out the definitions of $\Gamma^N_0$ and $\Gamma^N$, the to-be-analyzed equations for the remaining orders $N\in\N_{\geq2}$ become
    \begin{align*}
        &\quad\int_0^t\bigg(F^N_0(Y_u) + \sum_{n=2}^{N-1} \sum_{\tau\in\underline{N}_n} F^n_0(H^{\tau_1}(\theta_u\omega,Y_u),\dots,H^{\tau_n}(\theta_u\omega,Y_u))
        +\mathrm{D}F_0(0)H^N(\theta_u\omega,Y_u)\bigg)\,\mathrm{d}u\\
        &\quad+\int_0^t\bigg(F^N(Y_u) + \sum_{n=2}^{N-1} \sum_{\tau\in\underline{N}_n} F^n(H^{\tau_1}(\theta_u\omega,Y_u),\dots,H^{\tau_n}(\theta_u\omega,Y_u))
        +\mathrm{D}F(0)H^N(\theta_u\omega,Y_u)\bigg)\,\mathrm{d}\W_u(\omega)\\
        &=\int_0^t\bigg(\mathrm{ad}_N\mathrm{D}F_0(0)H^N(\theta_u\omega,Y_u)+K^N_0(\theta_u\omega,Y_u)+ \mathrm{D}_xH^N(\theta_u\omega,Y_u)\mathrm{D}F_0(0)Y_u\bigg)\,\mathrm{d}u\\
        &\quad+\int_0^t\bigg(\mathrm{ad}_N\mathrm{D}F(0)H^N(\theta_u\omega,Y_u)+K^N(\theta_u\omega,Y_u) + \mathrm{D}_xH^N(\theta_u\omega,Y_u)\mathrm{D}F(0)Y_u\bigg)\,\mathrm{d}\W_u(\omega).\mytag\label{rde tag conjugate id fulfilled}
    \end{align*}
    We readily confirm that the equations (\ref{rde tag conjugate id fulfilled}) indeed hold true for our choice of $H^N(\omega)$: After rearranging and recalling the definitions of $\mathrm{ad}_N\mathrm{D}F_0(0)$ and $\mathrm{ad}_N\mathrm{D}F(0)$, it is left to define the additive terms $K^N_0(\theta_t\omega,Y_t)$ and $K^N(\theta_t\omega,Y_t)$ as the combination of the left-over mixed terms. Specifically, $K^N_0(\theta_t\omega,Y_t)$ is a deterministic order-$N$ homogeneous polynomial in $Y_t$ depending upon the set $\{F^N_0(Y_t),F^n_0(Y_t),H^n(\theta_t\omega,Y_t):n=2,\dots,N-1\}$, whereas $K^N(\theta_t\omega,Y_t)$ is a deterministic order-$N$ homogeneous polynomial in $Y_t$ depending upon the set $\{F^N(Y_t),F^n(Y_t),H^n(\theta_t\omega,Y_t):n=2,\dots,N-1\}$.
     We emphasize two further consequences: First, $K^2_0\equiv F^2_0$ and $K^2\equiv F^2$ are deterministic, and, second, the randomness in $K^N_0(\omega,Y_t)$ and $K^N(\omega,Y_t)$ only enters via the random Taylor terms $H^n(\omega,Y_t)$ of lower orders $n=2,\dots,N-1$.

     This entire derivation showcases that the choice in (\ref{rde tag bold guess}) exactly provides the RDEs whose solutions give rise to the sought-after random coordinate transform which formally fulfills the conjugate relation (\ref{rde tag conjugate first}) for the initial RDE (\ref{rde tag standard RDE}) and its linearization \begin{align*}\,\mathrm{d}Y_t=\mathrm{D}F_0(0)Y_t\,\mathrm{d}t+\mathrm{D}F(0)Y_t\,\mathrm{d}\W_t(\omega),\end{align*}which is the desired (formal) normal form statement for RDEs. The open question we wish to answer affirmatively reads: Do stationary solutions to (\ref{rde tag bold guess}) exist for all $n\in\N_{\geq2}$?\\

    In summary, in order to arrive at a formal normal form result, we must study the RDEs (\ref{rde tag bold guess}) which are given by
    \begin{align*}
        \mathrm{d}H^n(\theta_t\omega)&=\left(\mathrm{ad}_n\mathrm{D}F_0(0)H^n(\theta_t\omega)+K^n_0(\theta_t\omega) \right)\mathrm{d}t\\
        &\quad+\left(\mathrm{ad}_n\mathrm{D}F(0)H^n(\theta_t\omega)+K^n(\theta_t\omega)\right)\mathrm{d}\W_t(\omega)
    \end{align*}
    and live in $\R^{D_n}$ for all $n\in\N_{\geq 2}$.

    We wish to find stationary solutions, which are solutions of the form $t\mapsto\phi_n(t,\omega)H^n(\omega)=H^n(\theta_t\omega)$, where $\phi_n(\cdot,\omega)$ is the solution flow generated by the $n$-th RDE, and
     we intend to investigate these RDEs by successively solving the stages of the following hierarchical system of affine RDEs:
    \begin{align*}
    \mathrm{d}Y_t^1 &= \left( A^1_0 Y_t^1 + P^1_0 \right)\mathrm{d}t+ \left( A^1 Y_t^1 + P^1 \right)\mathrm{d}\W_t(\omega),  \quad Y_0^1 = y_1 \in \mathbb{R}^{D_1},  \\
    \mathrm{d}Y_t^2 &=\left( A^2_0 Y_t^2 + P^2_0(Y_t^1) \right)\mathrm{d}t+ \left( A^2 Y_t^2 + P^2(Y_t^1) \right)\mathrm{d}\W_t(\omega), \quad Y_0^2 = y_2 \in \mathbb{R}^{D_2}, \\
    &\vdots \\
    \mathrm{d}Y_t^n &=\left( A^n_0 Y_t^n + P^n_0(Y_t^1, \dots, Y_t^{n-1}) \right)\mathrm{d}t+  \left( A^n Y_t^n + P^n(Y_t^1, \dots, Y_t^{n-1}) \right)\mathrm{d}\W_t(\omega), \\
    &\quad\quad\quad\quad\quad\quad\quad\quad\quad\quad\quad\quad\quad\quad\quad\quad\quad\quad\quad\quad\quad\quad\quad\quad\quad Y_0^n = y_n \in \mathbb{R}^{D_n}.\\
    &\vdots\mytag\label{rde hierarchical system basic}
    \end{align*}
    Of course, (\ref{rde hierarchical system basic}) mirrors (\ref{rde tag bold guess}). The linear maps $A_0^n\in\linearMaps{\R^{D_n}}{\R^{D_n}}$ and $A^n\in\linearMaps{\R^{D_n}}{\linearMaps{\R^{M}}{\R^{D_n}}}$ and the polynomials $P^n_0:\R^{D_1}\times\dots\times\R^{D_{n-1}}\rightarrow\R^{D_n}$ and $P^n:\R^{D_1}\times\dots\times\R^{D_{n-1}}\rightarrow\linearMaps{\R^{M}}{\R^{D_n}}$ are deterministic maps, with the exception that $P^1_0\in\R^{D_1}$ and $P^1\in\linearMaps{\R^{M}}{\R^{D_1}}$ are fixed objects.
    To enhance visual appeal, the hierarchical system starts with the first index $1$ rather than the first index $2$, which stems from the derivation above. Also, this index shift takes place in \cite[Section 8.5]{Arnold1998RDS}, so we decided to keep it.

    How do we concretely deal with the hierarchical system of affine RDEs? Initially, we solve the first stage and denote the resulting cocycle by $\phi_1$. Then, we insert $Y^1_t=\phi_1(t,\omega)y_1$ into the second stage, which leads to a non-autonomous affine RDE with solution flow $\phi_2(\cdot,\omega,y_1)$.  Continuing this procedure, the $n$-th stage RDE reads
    \begin{align*}
        \mathrm{d}Y^n_t&=\big(A^n_0Y_t^n+P^n_0(\phi_1(t,\omega)y_1,\dots,\phi_{n-1}(t,\omega,y_1,\dots,y_{n-2})y_{n-1}) \big)\,\mathrm{d}t\\
        &\quad+\big(A^nY_t^n+P^n(\phi_1(t,\omega)y_1,\dots,\phi_{n-1}(t,\omega,y_1,\dots,y_{n-2})y_{n-1}) \big)\,\mathrm{d}\W_t(\omega).
    \end{align*}
    Realize that the first $n$ RDEs considered as one generate a cocycle given by $\phi^{(n)}(\cdot,\omega)(y_1,\dots,y_n)=(\phi_1(\cdot,\omega)y_1,\dots,\phi_{n}(\cdot,\omega,y_1,\dots,y_{n-1})y_{n})$, with $(y_1,\dots,y_n)\in\R^{D_1}\times\dots\times\R^{D_n}$.

    \subsection{Studying the Hierarchical System (\ref{rde                  hierarchical system basic})}
        Remember that our strategy to treat the problem is to borrow the entire train of thought followed in \cite[Section 8.5]{Arnold1998RDS} and transfer it to the current RDE framework.
        \subsubsection{Bounded Polynomial Moments of Solutions to (\ref{rde hierarchical system basic})}

        Let us describe the conditions $\conditionRDE$ that we take from \cite[Lemma 8.5.3]{Arnold1998RDS} and that govern this subsection: We say that a parametrized family of random controlled rough paths $\U=(\U(\omega,y_1))_{\omega\in\Omega,y_1\in\V_1}\subseteq\bigcup_{\omega\in\Omega}\controlledRoughPath{\W(\omega)}{[0,1]}{\V_2}$ satisfies $\conditionRDE$ if $\U$ enjoys the following properties:
        \begin{enumerate}
            \item For each $\omega$ and every parameter $y_1\in\V_1$, we have $\U(\omega,y_1)\in\controlledRoughPath{\W(\omega)}{[0,1]}{\V_2}$.
            \item The entire process \begin{align*}[0,1]\times\Omega\to\V_2\times\linearMaps{\R^M}{\V_2},\;(t,\omega)\mapsto \left(U_t(\omega,y_1),U_t'(\omega,y_1)\right)\end{align*} is $(\F^t_{-\infty})_{t\in[0,1]}$-adapted for every $y_1\in\V_1$.
            \item For any $p\in[2,\infty)$ and any compact set $K\subseteq \V_1$ with $0\in K$, there exist constants $C=C(p),C(p,K)\in\R$ such that \begin{align*}\expectationPathwise{\sup_{t\in[0,1]}\abs{U_t(\omega,0)}^p+\sup_{t\in[0,1]}\abs{U'_t(\omega,0)}^p+ \norm{\U(\omega,0)}{\mathscr{D}_{\W(\omega)}^{2\alpha}}^p} \leq C\end{align*}
                 and, for all $y_1,\Tilde{y}_1\in K$, it holds \begin{align*}
                    &\int_\Omega\bigg(\sup_{t\in[0,1]}\abs{U_t(\omega,y_1)-U_t(\omega,\Tilde{y}_1)}^p+\sup_{t\in [0,1]}\abs{U'_t(\omega,y_1)-U_t'(\omega,\Tilde{y}_1)}^p\\
                    &\quad+ \norm{\U(\omega,y_1)-\U(\omega,\Tilde{y}_1)}{\mathscr{D}_{\W(\omega)}^{2\alpha}}^p\bigg)\,\mathrm{d}P(\omega) \leq C \abs{y_1-\Tilde{y}_1}^{p/2}.
                \end{align*}
            \end{enumerate}
        Slightly inaccurately, we often say that $\U(\omega)$ satisfies $\conditionRDE$, but we always mean this notion.
        Here and in the following, we employ the explicit integral notation for the expectation in order to stress the pathwise nature of the rough paths. Let us quickly write down an immediate consequence of this new notion.
        \begin{lemma}\label{rde lem compactness lemma}
            Let $\U(\omega):=(\U(\omega,y_1))_{y_1\in\V_1}\subseteq\controlledRoughPath{\W(\omega)}{[0,1]}{\V_2}$ meet $\conditionRDE$. It holds \begin{align*}
               \expectationPathwise{\sup_{t\in[0,1]}\abs{U_t(\omega,y_1)}^p+\sup_{t\in[0,1]}\abs{U'_t(\omega,y_1)}^p+ \norm{\U(\omega,y_1)}{\mathscr{D}_{\W(\omega)}^{2\alpha}}^p}\leq C
           \end{align*}
           for all $y_1\in\V_1$, where $C=C(p,y_1)\in\R$. Moreover, if $y_1\in K$ for a compact $K\subseteq\V_1$ with $0\in K$, then $C=C(p,K)$.
        \end{lemma}
        \begin{proof}
           The claim follows immediately from the definition of $\conditionRDE$ and elementary inequalities.
        \end{proof}

        This entire section is devoted to demonstrating that the operations that take place in our procedure pass on the requirements defined by $\conditionRDE$. It turns out that it is absolutely fundamental for us to detour to the theory of tails of enhanced Gaussian processes and RDEs driven by such. From this field, we borrow heavy machinery studied in \cite{CLL},  \cite{integrabilitynonlinearroughdifferential}, and \cite{GHANIVARZANEH2025110676} and \cite{mildGronwall2025}. Deducing these results is a story in itself, so we just state the main consequences that are relevant to us. In fact, these statements will play a role in almost every of the upcoming estimations.

    \begin{proposition}[Finite Moments of Solutions to RDEs Driven by Gaussian Enhancements]\label{rde prop consequences of heavy machinery}
        Let $(\W(\omega))_\omega\subseteq\driverRoughPath{[0,1]}{\R^{M}}$ be a weakly geometric enhancement of a centered Gaussian path. Then, we have $\int_\Omega\norm{W(\omega)}{\alpha}^q\,\mathrm{d}P(\omega)<\infty$ and $\int_\Omega\norm{\W(\omega)}{\mathscr{C}^\alpha}^q\,\mathrm{d}P(\omega)<\infty$ for all $q\in(0,\infty)$.

        Let now in addition $\,\mathrm{d}J_t=B_0J_t\,\mathrm{d}t+B J_t\,\mathrm{d}\W_t(\omega)$ be a matrix-valued linear RDE in $\linearMaps{\R^D}{\R^D}$, given $B_0\in\linearMaps{\R^D}{\R^D}$ and $B\in\linearMaps{\R^D}{\linearMaps{\R^M}{\R^D}}$. Denote by $\Phi(\cdot,\omega)$ the linear solution flow generated by this RDE. Then, any solution $\Upphi(\cdot,\omega,\Phi_0)$ starting in $\Phi_0\in\linearMaps{\R^D}{\R^D}$ fulfills $\int_\Omega\norm{\Upphi(\cdot,\omega,\Phi_0)}{\mathscr{D}^{2\alpha}_{\W(\omega)}}^q\,\mathrm{d}P(\omega) <\infty$ for all $q\in[1,\infty)$.
        Similarly, any solution $\boldsymbol{\Uppsi}(\cdot,\omega,\Psi_0)$ to the matrix-valued linear RDE $\,\mathrm{d}J_t=-J_tB_0\,\mathrm{d}t-J_tB\,\mathrm{d}\W_t(\omega)$ in $\linearMaps{\R^{D}}{\R^D}$ starting in $\Psi_0\in\linearMaps{\R^D}{\R^D}$ also satisfies $\int_\Omega\norm{\boldsymbol{\Uppsi}(\cdot,\omega,\Psi_0)}{\mathscr{D}^{2\alpha}_{\W(\omega)}}^q\,\mathrm{d}P(\omega) <\infty$ for all $q\in[1,\infty)$.
    \end{proposition}
    \begin{proof}
        The first statement is a consequence of \cite[Proposition 4]{integrabilitynonlinearroughdifferential}, and, then, we deduce the second statement from the first one combined with \cite[Theorem 7.44]{Friz_Victoir_2010}. Last, \cite[Theorem 2.13]{GHANIVARZANEH2025110676} as well as \cite[Theorem 3.10]{mildGronwall2025} produce the third and fourth results.
    \end{proof}

    Armed with these tools, we proceed with our main agenda. As a convention, whenever we intend to demonstrate inheritance of $\conditionRDE$ under some operation, we solely address point (ii) and the second condition of point (iii) of the definition above. This is simply for the sake of brevity, because (i) is always obvious and the computations for the first condition of (iii) are usually quite similar to, and/or significantly simpler than, the calculations we perform to establish the second condition of (iii).
       \begin{lemma}\label{rde lem Lemma 8.5.3}
        Let $\U(\omega):=(\U(\omega,y_1))_{y_1\in\R^{D_1}}\subseteq\controlledRoughPath{\W(\omega)}{[0,1]}{\linearMaps{\R^M}{\R^{D_2}}}$ and $\Z(\omega):=(\Z(\omega,y_1))_{y_1\in\R^{D_1}}\subseteq\controlledRoughPath{\W(\omega)}{[0,1]}{\R^{D_2}}$ satisfy $\conditionRDE$.
        \begin{enumerate}
            \item Then, $\I(\omega):=(\I(\omega,y_1))_{y_1\in\R^{D_1}}:=\left(\int_0^\cdot \U_s(\omega,y_1)\,\mathrm{d}\W_s(\omega)\right)_{y_1\in\R^{D_1}}\subseteq\controlledRoughPath{\W(\omega)}{[0,1]}{\R^{D_2}}$ fulfills $\conditionRDE$.
            \item Also, $\I^0(\omega):=(\I^0(\omega,y_1))_{y_1\in\R^{D_1}}:=\left(\left(\int_0^\cdot Z_s(\omega,y_1)\,\mathrm{d}s,0\right)\right)_{y_1\in\R^{D_1}}\subseteq\controlledRoughPath{\W(\omega)}{[0,1]}{\R^{D_2}}$ fulfills $\conditionRDE$.
        \end{enumerate}
       \end{lemma}
       \begin{proof}
           Let $p\in[2,\infty)$, and fix some compact $K\subseteq\R^{D_1}$ with $0\in K$ and $y_1,\tildeY_1\in K$.

           $\underline{\textit{Regarding (i)}}$.
           We readily validate that the controlled integral of an adapted process is adapted itself by reviewing its definition as a limit of compensated Riemann sums.
           Then, utilizing Theorem \ref{theorem gubinelli estimates}, using elementary inequalities, and considering trivial estimations such as $ \abs{U_0(\omega,y_1)-U_0(\omega,\tildeY_1)}\leq \sup_{t\in[0,1]} \abs{U_t(\omega,y_1)-U_t(\omega,\tildeY_1)}$ first yields
           \begin{align*}
               \expectationPathwise{\sup_{t\in[0,1]} \abs{I_t(\omega,y_1)-I_t(\omega,\tildeY_1)}^p}&\leq C\expectationPathwise{\sup_{t\in[0,1]} \abs{U_t(\omega,y_1)-U_t(\omega,\tildeY_1)}^p\normCalpha{\W(\omega)}^p}\\
               &\quad+C\expectationPathwise{\sup_{t\in[0,1]}\abs{U_t'(\omega,y_1)-U_t'(\omega,\tildeY_1)}^p\normCalpha{\W(\omega)}^p}\\
               &\quad +C\expectationPathwise{\normDalphaOmega{\U(\omega,y_1)-\U(\omega,\tildeY_1)}^p\normCalpha{\W(\omega)}^p}.
           \end{align*}
        Next, we apply Hölder's inequality and absorb the quantity $\expectationPathwise{\normCalpha{\W(\omega)}^{2p}}<\infty$ in the constant, which is possible due to Proposition \ref{rde prop consequences of heavy machinery}. This gives
        \begin{align*}
            \expectationPathwise{\sup_{t\in[0,1]} \abs{I_t(\omega,y_1)-I_t(\omega,\tildeY_1)}^p}&\leq C\left(\expectationPathwise{\sup_{t\in[0,1]} \abs{U_t(\omega,y_1)-U_t(\omega,\tildeY_1)}^{2p}}\right)^{1/2}\\
               &\quad+C\left(\expectationPathwise{\sup_{t\in[0,1]}\abs{U_t'(\omega,y_1)-U_t'(\omega,\tildeY_1)}^{2p}}\right)^{1/2}\\
               &\quad +C\left(\expectationPathwise{\normDalphaOmega{\U(\omega,y_1)-\U(\omega,\tildeY_1)}^{2p}}\right)^{1/2}.
        \end{align*}
        Since $\U$ meets $\conditionRDE$, we end up with $\expectationPathwise{\sup_{t\in[0,1]} \abs{I_t(\omega,y_1)-I_t(\omega,\tildeY_1)}^p}\leq C\abs{y_1-\tildeY_1}^{p/2}$.

        Second, we observe
        \begin{align*}
            \expectationPathwise{\sup_{t\in[0,1]} \abs{I_t'(\omega,y_1)-I_t'(\omega,\tildeY_1)}^p}
            = \expectationPathwise{\sup_{t\in[0,1]} \abs{U_t(\omega,y_1)-U_t(\omega,\tildeY_1)}^p}\leq C\abs{y_1-\tildeY_1}^{p/2}.
        \end{align*}

        Third, by means of Theorem \ref{theorem gubinelli estimates} and similar considerations as before, we deduce
        \begin{align*}
            &\quad\expectationPathwise{\normDalphaOmega{\I(\omega,y_1)-\I(\omega,\tildeY_1)}^p}\\
            &\leq C\expectationPathwise{\left(\abs{U_0'(\omega,y_1)-U_0'(\omega,\tildeY_1)}+\normDalphaOmega{\U(\omega,y_1)-\U(\omega,\tildeY_1)}\right)^p\left(1+\normCalpha{\W(\omega)}\right)^p}\\
            &\leq C\abs{y_1-\tildeY_1}^{p/2},
        \end{align*}
        which completes the proof of claim \textit{(i)}.\\

        $\underline{\textit{Regarding (ii)}}$.
        Adaptedness is a direct consequence of the definition of the Riemann integral as the limit of Riemann sums. Then, note that, since $Z(\omega,y_1)$ is in particular continuous, the mapping $t\mapsto\int_0^t Z_s(\omega,y_1)\,\mathrm{d}s$ is actually continuously differentiable, so that, for all $s,t\in[0,1]$ with $s<t$, we have
        \begin{align*}
            I^0_t(\omega,y_1)-I^0_s(\omega,y_1)=0W_{s,t}(\omega)+\left(I^0_t(\omega,y_1)-I^0_s(\omega,y_1)\right),
        \end{align*}
        for $I^0_t(\omega,y_1)-I^0_s(\omega,y_1)=\mathrm{O}(\abs{t-s}^{2\alpha})$.
        This is why the Gubinelli derivative vanishes. Now, we transition to the estimations.

        It is easy to see that
        \begin{align*}
            \sup_{t\in[0,1]}\abs{I^0_t(\omega,y_1)-I^0_t(\omega,\tildeY_1)}
        \leq \sup_{t\in[0,1]}\abs{Z_t(\omega,y_1)-Z_t(\omega,\tildeY_1)},
        \end{align*}
        which we advance to
        \begin{align*}
            \expectationPathwise{\sup_{t\in[0,1]}\abs{I^0_t(\omega,y_1)-I^0_t(\omega,\tildeY_1)}^p}\leq C\abs{y_1-\tildeY_1}^{p/2}.
        \end{align*}
        Next, recall that the Gubinelli derivative does not need to be considered. Last, we compute
        \begin{align*}
            \normDalphaOmega{\I^0(\omega,y_1)-\I^0(\omega,\tildeY_1)}
            &=\norm{I^0(\omega,y_1)-I^0(\omega,\tildeY_1)}{2\alpha}\\
            &=\sup_{(s,t)\in\Delta}\frac{\abs{\int_s^t Z_r(\omega,y_1)-Z_r(\omega,\tildeY_1)\,\mathrm{d}r}}{\abs{t-s}^{2\alpha}}\\
            &\leq\sup_{t\in[0,1]}\abs{Z_t(\omega,y_1)-Z_t(\omega,\tildeY_1)},
        \end{align*}
        where $\Delta:=\{(s,t)\in[0,1]^2:s<t\}$.
        Putting this together entails
        \begin{align*}
            \expectationPathwise{\normDalphaOmega{\I^0(\omega,y_1)-\I^0(\omega,\tildeY_1)}^p}\leq C\abs{y_1-\tildeY_1}^{p/2},
        \end{align*}
        which closes the proof.
       \end{proof}

    \begin{lemma}\label{rde lem Lemma 8.5.5}
        Let $\U(\omega):=(\U(\omega,y_1))_{y_1\in\R^{D_1}}\subseteq\controlledRoughPath{\W(\omega)}{[0,1]}{\linearMaps{\R^M}{\R^{D_2}}}$ and $\Z(\omega):=(\Z(\omega,y_1))_{y_1\in\R^{D_1}}\subseteq\controlledRoughPath{\W(\omega)}{[0,1]}{\R^{D_2}}$ satisfy $\conditionRDE$, and let $B_0\in\linearMaps{\R^{D_2}}{\R^{D_2}}$ and $B\in\linearMaps{\R^{D_2}}{\linearMaps{\R^M}{\R^{D_2}}}$. Consider the non-autonomous affine RDE
        \begin{align*}
        \mathrm{d}Y_t=(B_0Y_t+Z_t(\omega,y_1))\,\mathrm{d}t+(BY_t+U_t(\omega,y_1))\,\mathrm{d}\W_t(\omega)
       \end{align*}
        in $\R^{D_2}$ with initial condition $Y_0=y_2\in\R^{D_2}$, whose solution flow we denote by $\phi(\cdot,\omega,y_1)$.
        Moreover, we write $\Phi$ for the linear cocycle generated by the linearized equation $\,\mathrm{d}Y_t=B_0Y_t\,\mathrm{d}t+B Y_t\,\mathrm{d}\W_t(\omega)$.  Then, $(\upvarphi(\cdot,\omega,y_1)y_2)_{y_1\in\R^{D_1},y_2\in\R^{D_2}}$ satisfies $\conditionRDE$.
    \end{lemma}

    Due to the lack of "good" and "elementary" BDG-type plus Gronwall-type inequalities we cannot conduct the proof in full analogy with the SDE case: A naive method involving exponential term estimates as in \cite[Proposition 8.13]{FrizHairer2020} is insufficient because, in general, we have no hope of achieving finite exponential moments such as
    \begin{align*}
        \xcancel{\expectationPathwise{\mathrm{e}^{c\normCalpha{\W(\omega)}^{1/\alpha}}} <\infty}
    \end{align*}
    for a generic (possibly large) constant $c\in\R$ plus an exponent $1/\alpha\in[2,3)$. For reference, classical Fernique-type results in the style of \cite[Theorem 4.1]{fernique} state that $\expectationPathwise{\mathrm{e}^{c\abs{W^*(\omega)}^2}}=\infty$ if $c$ is too large and $W^*$ is a generic centered Gaussian random variable. See also \cite[Theorem 15.33]{Friz_Victoir_2010} for a similar Fernique estimate for centered Gaussian rough paths. More advanced, \cite[Lemma 4.2]{mildGronwall2025} provides a mild rough Gronwall inequality that actually circumvents the issues of exponential moments by reducing them to polynomial moments. In fact, the mild Gronwall inequality allows us to conclude Lemma \ref{rde lem Lemma 8.5.5} very quickly, which we briefly illustrate in Remark \ref{rde rem difficult lemma with mild gronwall}. However, we focus on a more elementary presentation which relies solely on Proposition \ref{rde prop consequences of heavy machinery}, from which we obtain finite polynomial moments of all orders. These provide the sufficient key to bypassing the SDE-type argumentation and sticking with standard Gubinelli-type estimates. For us, this procedure has the advantage that, on the way, we encounter the intermediate results Lemma \ref{rde lem product rule preserves (C-RDE)} and Lemma \ref{rde lem bounding expectation of I= int Phi U dW} which are essential for the rest of this section and, therefore, would have been to be established in retrospect, anyway.

    \begin{lemma}\label{rde lem product rule preserves (C-RDE)}
        Let $\Upphi(\cdot,\omega)\in\controlledRoughPath{\W(\omega)}{[0,1]}{\linearMaps{\R^{D_2}}{\R^{D_2}}}$ be the controlled rough path that comes from the linear cocycle $\Phi$ generated by the linear RDE $\,\mathrm{d}Y_t=B_0Y_t\,\mathrm{d}t+BY_t\,\mathrm{d}\W_t(\omega)$, where $B_0\in \linearMaps{\R^{D_2}}{\R^{D_2}}$ and $B\in\linearMaps{\R^{D_2}}{\linearMaps{\R^M}{\R^{D_2}}}$. Moreover, let $\U(\omega):=(\U(\omega,y_1))_{y_1\in\R^{D_1}}\subseteq\controlledRoughPath{\W(\omega)}{[0,1]}{\linearMaps{\R^M}{\R^{D_2}}}$ and $\Z(\omega):=(\Z(\omega,y_1))_{y_1\in\R^{D_1}}\subseteq\controlledRoughPath{\W(\omega)}{[0,1]}{\R^{D_2}}$ satisfy $\conditionRDE$.
         \begin{enumerate}
             \item Then, also $(\Upphi(\cdot,\omega)^{-1}\U(\omega,y_1))_{y_1\in\R^{D_1}}\subseteq\controlledRoughPath{\W(\omega)}{[0,1]}{\linearMaps{\R^M}{\R^{D_2}}}$ meets $\conditionRDE$.
             \item Furthermore, $(\Upphi(\cdot,\omega)^{-1}\Z(\omega,y_1))_{y_1\in\R^{D_1}}\subseteq\controlledRoughPath{\W(\omega)}{[0,1]}{\R^{D_2}}$ fulfills $\conditionRDE$.
         \end{enumerate}
    \end{lemma}
    \begin{proof}
    To start off, notice that, as Lemma \ref{rde lem phi inverse is also controlled rough path} provides $\Upphi(\cdot,\omega)^{-1}\in\controlledRoughPath{\W(\omega)}{[0,1]}{\linearMaps{\R^{D_2}}{\R^{D_2}}}$,
    Lemma \ref{rde lem product rule} guarantees $\Upphi(\cdot,\omega)^{-1}\U(\omega,y_1)\in\controlledRoughPath{\W(\omega)}{[0,1]}{\linearMaps{\R^M}{\R^{D_2}}}$ and $\Upphi(\cdot,\omega)^{-1}\Z(\omega,y_1)\in\controlledRoughPath{\W(\omega)}{[0,1]}{\R^{D_2}}$. To check adaptedness, notice that adaptedness of the solution to an RDE with adapted weakly geometric driver and adapted initial condition is a standard problem resolved by means of continuous dependence of the solution upon the adapted weakly geometric driver, as presented in \cite[Theorem 8.5]{FrizHairer2020} and \cite[Theorem 9.1]{FrizHairer2020} in the case of Brownian motion. This entails that
    $(t,\omega)\mapsto\Upphi(\cdot,\omega)^{-1}$ is an adapted process. As a consequence, the two processes in question inherit adaptedness because matrix-multiplication is continuous and, hence, preserves adaptedness of two adapted processes.

    Now, let $p\in[2,\infty)$ and $y_1,\tildeY_1\in K$ for some arbitrary compact $K\subseteq\R^{D_1}$ with $0\in K$.\\

    $\underline{\textit{Regarding (i)}}$. First, an application of Hölder's inequality yields
    \begin{align*}
        &\quad\expectationPathwise{\sup_{t\in[0,1]}\abs{\Phi(t,\omega)^{-1}(U_t(\omega,y_1)-U_t(\omega,\tildeY_1))}^p}\\
        &\leq C\left(\expectationPathwise{\sup_{t\in[0,1]}\abs{\Phi(t,\omega)^{-1}}^{2p}}\right)^{1/2} \left(\expectationPathwise{\sup_{t\in[0,1]}\abs{U_t(\omega,y_1)-U_t(\omega,\tildeY_1)}^{2p}}\right)^{1/2}.
    \end{align*}
    On the one hand, inserting $(\Phi^{-1})'(\cdot,\omega)=-\Phi(\cdot,\omega)^{-1}B$ and $\Phi(0,\omega)^{-1}=\id$, the computation performed in \cite[(4.20)]{FrizHairer2020}, and $\abs{B}<\infty$, we deduce
     \begin{align*}
            \quad\norm{\Phi(\cdot,\omega)^{-1}}{\sup}
            &\leq \abs{\Phi(0,\omega)^{-1}}+\norm{\Phi(\cdot,\omega)^{-1}}{\alpha}\\
            &\leq C+(1+\norm{W(\omega)}{\alpha})\left(C+\normDalphaOmega{\Upphi(\cdot,\omega)^{-1}}\right),
        \end{align*}
    which by means of Hölder's inequality, elementary inequalities and Proposition \ref{rde prop consequences of heavy machinery} results in finiteness of the first factor. On the other hand, $\expectationPathwise{\sup_{t\in[0,1]}\abs{U_t(\omega,y_1)-U_t(\omega,\tildeY_1)}^{2p}}\leq C\abs{y_1-\tildeY_1}^p$ holds true by assumption. Hence,
    \begin{align*}
        \expectationPathwise{\sup_{t\in[0,1]}\abs{\Phi(t,\omega)^{-1}(U_t(\omega,y_1)-U_t(\omega,\tildeY_1))}^p}\leq C\abs{y_1-\tildeY_1}^{p/2}.
    \end{align*}

    Second, in a similar manner, employing the product rule from Lemma \ref{rde lem product rule} and some additional elementary inequalities, we write
    \begin{align*}
        &\quad\expectationPathwise{\sup_{t\in[0,1]}\abs{\big(\Phi(\cdot,\omega)^{-1}(U(\omega,y_1)-U(\omega,\tildeY_1))\big)'_t}^p}\\
        &\leq C\left(\expectationPathwise{\sup_{t\in[0,1]}\abs{(\Phi^{-1})'(t,\omega)}^{2p}}\right)^{1/2} \left(\expectationPathwise{\sup_{t\in[0,1]}\abs{U_t(\omega,y_1)-U_t(\omega,\tildeY_1)}^{2p}}\right)^{1/2}\\
        &\quad+C \left(\expectationPathwise{\sup_{t\in[0,1]}\abs{\Phi(t,\omega)^{-1}}^{2p}}\right)^{1/2} \left(\expectationPathwise{\sup_{t\in[0,1]}\abs{U_t'(\omega,y_1)-U_t'(\omega,\tildeY_1)}^{2p}}\right)^{1/2}.
    \end{align*}
    From the above calculation, we know how to control the second summand by $C\abs{y_1-\tildeY_1}^{p/2}$. Also, recalling $(\Phi^{-1})'(\cdot,\omega)=-\Phi(\cdot,\omega)^{-1}B$ and borrowing an estimate from above entails
    \begin{align*}
        \quad\norm{(\Phi^{-1})'(\cdot,\omega)}{\sup}
        &\leq \norm{-\Phi^{-1}(\cdot,\omega)}{\sup}\abs{B}\\
        &\leq C+C(1+\norm{W(\omega)}{\alpha})\left(C+\normDalphaOmega{\Upphi(\cdot,\omega)^{-1}}\right).
    \end{align*}
    Consequently, it is, according to standard estimations and Proposition \ref{rde prop consequences of heavy machinery}, easy to see that the first summand is $\leq C\abs{y_1-\tildeY_1}^{p/2}$, too.

    Third, we have to handle $\expectationPathwise{\normDalphaOmega{\Upphi(\cdot,\omega)^{-1}(\U(\omega,y_1)-\U(\omega,\tildeY_1))}^p}$. We obtain
    \begin{align*}
        &\quad \expectationPathwise{\normDalphaOmega{\Upphi(\cdot,\omega)^{-1}(\U(\omega,y_1)-\U(\omega,\tildeY_1))}^p}\\
        &\leq C\left(\expectationPathwise{\norm{(\Phi^{-1})'(\cdot,\omega)}{\sup}^{3p}}\right)^{1/3} \left(\expectationPathwise{\norm{U(\omega,y_1)-U(\omega,\tildeY_1)}{\alpha}^{3p}}\right)^{1/3}\\
        &\quad\quad\cdot\left(\expectationPathwise{(1+\norm{W(\omega)}{\alpha})^{3p}}\right)^{1/3}\\
            &\quad+C \left(\expectationPathwise{\norm{U(\omega,y_1)-U(\omega,\tildeY_1)}{\sup}^{2p}}\right)^{1/2}\left(\expectationPathwise{\normDalphaOmega{\Upphi(\cdot,\omega)^{-1}}^{2p}}\right)^{1/2}\\
            &\quad+ C\left(\expectationPathwise{\norm{\Phi(\cdot,\omega)^{-1}}{\sup}^{2p}}\right)^{1/2}\left(\expectationPathwise{\normDalphaOmega{\U(\omega,y_1)-\U(\omega,\tildeY_1)}^{2p}}\right)^{1/2}\\
            &\quad+ C\left(\expectationPathwise{\norm{\Phi(\cdot,\omega)^{-1}}{\alpha}^{2p}}\right)^{1/2} \left(\expectationPathwise{\norm{U'(\omega,y_1)-U'(\omega,\tildeY_1)}{\sup}^{2p}}\right)^{1/2}
    \end{align*}
    by means of the product rule estimate derived in Lemma \ref{rde lem product rule} and standard elementary inequalities combined with Hölder's inequality.
        Controlling all these terms works as above and produces \begin{align*}
            \expectationPathwise{\normDalphaOmega{\Upphi(\cdot,\omega)^{-1}(\U(\omega,y_1)-\U(\omega,\tildeY_1))}^p}\leq C\abs{y_1-\tildeY_1}^{p/2}.
        \end{align*}

        $\underline{\textit{Regarding (ii)}}$. This is just a special case of \textit{(i)} with $M=1$.
    \end{proof}

    The following intermediate result is now wide open to be picked up:
    \begin{lemma}\label{rde lem bounding expectation of I= int Phi U dW}        Let $\Upphi(\cdot,\omega)\in\controlledRoughPath{\W(\omega)}{[0,1]}{\linearMaps{\R^{D_2}}{\R^{D_2}}}$ be the controlled rough path that comes from the linear cocycle $\Phi$ generated by the linear RDE $\,\mathrm{d}Y_t=B_0Y_t\,\mathrm{d}t+BY_t\,\mathrm{d}\W_t(\omega)$. Furthermore, let $\U(\omega):=(\U(\omega,y_1))_{y_1\in\R^{D_1}}\subseteq\controlledRoughPath{\W(\omega)}{[0,1]}{\linearMaps{\R^M}{\R^{D_2}}}$ and $\Z(\omega):=(\Z(\omega,y_1))_{y_1\in\R^{D_1}}\subseteq\controlledRoughPath{\W(\omega)}{[0,1]}{\R^{D_2}}$ satisfy $\conditionRDE$.
        \begin{enumerate}
            \item Then, $\left(\int_0^\cdot\Upphi(s,\omega)^{-1}\U_s(\omega,y_1)\,\mathrm{d}\W_s(\omega)\right)_{y_1\in\R^{D_1}}\subseteq\controlledRoughPath{\W(\omega)}{[0,1]}{\R^ {D_2}}$ meets $\conditionRDE$.
            \item Also, $\left(\left(\int_0^\cdot\Phi(s,\omega)^{-1}Z_s(\omega,y_1)\,\mathrm{d}s,0\right)\right)_{y_1\in\R^{D_1}}\subseteq\controlledRoughPath{\W(\omega)}{[0,1]}{\R^ {D_2}}$ fulfills $\conditionRDE$.
        \end{enumerate}
    \end{lemma}
    \begin{proof}
        Pair Lemma \ref{rde lem product rule preserves (C-RDE)} and Lemma \ref{rde lem Lemma 8.5.3}.
    \end{proof}

    We have finally gathered all the necessary ingredients to be able to tackle the proof of Lemma \ref{rde lem Lemma 8.5.5}.
    \begin{proof}[Proof of Lemma \ref{rde lem Lemma 8.5.5}.]
        Let us first verify adaptedness of the entire solution process
        \begin{align*}
            (t,\omega)\mapsto\left(\varphi(t,\omega,y_1)y_2,\varphi'(t,\omega,y_1)y_2\right).
        \end{align*}
        However, this task is immediately resolved by combining Lemma \ref{rde lem bounding expectation of I= int Phi U dW} and Duhamel's formula, Corollary \ref{rde corollary duhamels formula}.
        Let now $p\in[2,\infty)$, and let $K\subseteq\R^{D_1}\times\R^{D_2}$ be compact with $0\in K$. Let in addition $(y_1,y_2),(\Tilde{y}_1,\Tilde{y}_2)\in K$ be arbitrary. We must bound three instances:\begin{enumerate}
            \item \begin{align*}\expectationPathwise{\sup_{t\in[0,1]}\abs{\phi(t,\omega,y_1)y_2-\phi(t,\omega,\Tilde{y}_1)\Tilde{y}_2}^p}\leq C\abs{(y_1,y_2)-(\Tilde{y}_1,\Tilde{y}_2)}^{p/2},\end{align*}
            \item \begin{align*}\expectationPathwise{\sup_{t\in[0,1]}\abs{\phi'(t,\omega,y_1)y_2-\phi'(t,\omega,\Tilde{y}_1)\Tilde{y}_2}^p}\leq C\abs{(y_1,y_2)-(\Tilde{y}_1,\Tilde{y}_2)}^{p/2},\end{align*}
            \item \begin{align*}\expectationPathwise{\norm{\upvarphi(\cdot,\omega,y_1)y_2-\upvarphi(\cdot,\omega,\Tilde{y}_1)\Tilde{y}_2}{\mathscr{D}^{2\alpha}_{\W(\omega)}}^p}\leq  C\abs{(y_1,y_2)-(\Tilde{y}_1,\Tilde{y}_2)}^{p/2}.\end{align*}
        \end{enumerate}
        We begin with (iii) because it turns out that proving (i) and (ii) relies on statement (iii) being true.\\

        \underline{\textit{Regarding} (iii)}. To start off, we do some pathwise considerations. We represent $\upvarphi(\cdot,\omega,y_1)y_2$ by means of Duhamel's formula: Recall from Corollary \ref{rde corollary duhamels formula} that
        \begin{align*}
            \upvarphi(t,\omega,y_1)y_2 &=\Big(\Phi(t,\omega)y_2+\Phi(t,\omega)I^0_t(\omega,y_1)+\Phi(t,\omega)I_t(\omega,y_1),\\
            &\quad\quad
             B\left(\Phi(t,\omega)y_2+\Phi(t,\omega)I^0_t(\omega,y_1)+\Phi(t,\omega)I_t(\omega,y_1)\right)+U_t(\omega,y_1)\Big)
        \end{align*}
        for $t\in[0,1]$, where we set $\I^0(\omega,y_1):=(\int_0^\cdot\Phi(s,\omega)^{-1}Z_s(\omega,y_1)\,\mathrm{d}s,0)$ and   $\I(\omega,y_1):=\int_0^\cdot\Upphi(s,\omega)^{-1}\U_s(\omega,y_1)\,\mathrm{d}\W_s(\omega)$.
        Since the occurring semi-norms are sub-additive, it is easy to see that $\normDalphaOmega{\cdot}$ is sub-additive, too. In our case, that means that
        \begin{align*}
            \normDalphaOmega{\upvarphi(\cdot,\omega,y_1)y_2-\upvarphi(\cdot,\omega,\Tilde{y}_1)\Tilde{y}_2}
             &\leq \normDalphaOmega{\big(\Phi(\cdot,\omega)(y_2-\tildeY_2),B\Phi(\cdot,\omega)(y_2-\tildeY_2)\big)}\\
             &\quad+\normDalphaOmega{\Upphi(\cdot,\omega)\left(\I^0(\omega,y_1)-\I^0(\omega,\tildeY_1)\right)}\\
             &\quad+\normDalphaOmega{\Upphi(\cdot,\omega)\left(\I(\omega,y_1)-\I(\omega,\tildeY_1)\right)}.
        \end{align*}
        Using the product rule estimate from Lemma \ref{rde lem product rule} twice, inserting $(I^0)'(\omega,y_1)=(I^0)'(\omega,\tildeY_1)= 0$, raising to the $p$-th power, taking expectation, and employing elementary inequalities as well as Hölder's inequality, we arrive at
        \begin{align*}
            &\quad\expectationPathwise{\norm{\upvarphi(\cdot,\omega,y_1)y_2-\upvarphi(\cdot,\omega,\Tilde{y}_1)\Tilde{y}_2}{\mathscr{D}^{2\alpha}_{\W(\omega)}}^p}\\
            &\leq C \expectationPathwise{\normDalphaOmega{\big(\Phi(\cdot,\omega)(y_2-\Tilde{y}_2),B\Phi(\cdot,\omega)(y_2-\Tilde{y}_2)\big)}^p}\\
            &\quad + C\left(\expectationPathwise{\norm{\Phi'(\cdot,\omega)}{\sup}^{3p}}\right)^{1/3}
            \left(\expectationPathwise{\norm{I^0(\omega,y_1)-I^0(\omega,\Tilde{y}_1)}{\alpha}^{3p}}\right)^{1/3}\\
            &\quad\quad\cdot\left(\expectationPathwise{(1+\norm{W(\omega)}{\alpha})^{3p }}\right)^{1/3}\\
            &\quad+C\left(\expectationPathwise{\norm{I^0(\omega,y_1)-I^0(\omega,\Tilde{y}_1)}{\sup}^{2p}}\right)^{1/2} \left(\expectationPathwise{\normDalphaOmega{\Upphi(\cdot,\omega)}^{2p}}\right)^{1/2}\\
            &\quad+C\left(\expectationPathwise{\norm{\Phi(\cdot,\omega)}{\sup}^{2p}}\right)^{1/2} \left(\expectationPathwise{\normDalphaOmega{\I^0(\omega,y_1)-\I^0(\omega,\Tilde{y}_1)}^{2p}}\right)^{1/2} \\
            &\quad + C\left(\expectationPathwise{\norm{\Phi'(\cdot,\omega)}{\sup}^{3p}}\right)^{1/3}
            \left(\expectationPathwise{\norm{I(\omega,y_1)-I(\omega,\Tilde{y}_1)}{\alpha}^{3p}}\right)^{1/3}\\
            &\quad\quad\cdot\left(\expectationPathwise{(1+\norm{W(\omega)}{\alpha})^{3p }}\right)^{1/3}\\
            &\quad+C\left(\expectationPathwise{\norm{I(\omega,y_1)-I(\omega,\Tilde{y}_1)}{\sup}^{2p}}\right)^{1/2} \left(\expectationPathwise{\normDalphaOmega{\Upphi(\cdot,\omega)}^{2p}}\right)^{1/2}\\
            &\quad+C\left(\expectationPathwise{\norm{\Phi(\cdot,\omega)}{\sup}^{2p}}\right)^{1/2} \left(\expectationPathwise{\normDalphaOmega{\I(\omega,y_1)-\I(\omega,\Tilde{y}_1)}^{2p}}\right)^{1/2} \\
            &\quad+C\left(\expectationPathwise{\norm{\Phi(\cdot,\omega)}{\alpha}^{2p}}\right)^{1/2}\left(\expectationPathwise{\norm{I'(\omega,y_1)-I'(\omega,\Tilde{y}_1)}{\sup}^{2p}}\right)^{1/2}.
        \end{align*}
        Reviewing the terms, it seems reasonable that we wish to find that the first summand is $\leq C\abs{y_2-\Tilde{y}_2}^{p/2}$ and that all other summands are $\leq C\abs{y_1-\Tilde{y}_1}^{p/2}$.

         But, the observation
         \begin{align*}
            \normDalphaOmega{\big(\Phi(\cdot,\omega)(y_2-\Tilde{y}_2),B\Phi(\cdot,\omega)(y_2-\Tilde{y}_2)\big)}\leq C\normDalphaOmega{\Upphi(\cdot,\omega)}\abs{y_2-\Tilde{y}_2}
        \end{align*}
        combined with standard procedures entails that the first summand is $\leq C\abs{y_2-\tildeY_2}^p$. For the remaining summands, the required steps to arrive at $\leq C\abs{y_1-\tildeY_1}^{p/2}$ are by now routine, thanks to Lemma \ref{rde lem bounding expectation of I= int Phi U dW}. In total, we have
        \begin{align*}
            \quad\expectationPathwise{\norm{\upvarphi(\cdot,\omega,y_1)y_2-\upvarphi(\cdot,\omega,\Tilde{y}_1)\Tilde{y}_2}{\mathscr{D}^{2\alpha}_{\W(\omega)}}^p}
            \leq C(\abs{y_2-\Tilde{y}_2}^p+\abs{y_1-\Tilde{y}_1}^{p/2}).
        \end{align*}
        It remains to recall that $(y_1,y_2),(\Tilde{y}_1,\Tilde{y}_2)\in K$, and that $K$ is compact, which implies $\abs{y_2-\Tilde{y}_2}^p\leq C\abs{y_2-\Tilde{y}_2}^{p/2}$. Merging the two partial Euclidean distances establishes statement (iii).\\

        \underline{\textit{Regarding} (i)}. We insert that $\phi$ is the solution flow, apply standard methods and $\abs{B_0},\abs{B}<\infty$ and utilize Theorem \ref{theorem gubinelli estimates}  to obtain, for all $t\in[0,1]$, the generous estimate
        \begin{align*}
            &\quad\abs{\phi(t,\omega,y_1)y_2-\phi(t,\omega,\Tilde{y}_1)\Tilde{y}_2}\\
            &=\bigg|y_2-\tildeY_2+\int_0^t B_0\left(\phi(s,\omega,y_1)y_2-\phi(s,\omega,\Tilde{y}_1)\Tilde{y}_2\right) + \left(Z_s(\omega,y_1)-Z_s(\omega,\Tilde{y}_1)\right)\mathrm{d}s\\
            &\quad\quad+\int_0^t B\left(\phi(s,\omega,y_1)y_2-\phi(s,\omega,\Tilde{y}_1)\Tilde{y}_2\right) + \left(U_s(\omega,y_1)-U_s(\omega,\Tilde{y}_1)\right)\mathrm{d}\W_s(\omega)\bigg|\\
            &\leq\abs{y_2-\tildeY_2}+\sup_{t\in[0,1]}\abs{Z_t(\omega,y_1)-Z_t(\omega,\tildeY_1)} +C\int_0^t \abs{\phi(s,\omega,y_1)y_2-\phi(s,\omega,\Tilde{y}_1)\Tilde{y}_2}\,\mathrm{d}s\\
            &\quad+\abs{\big(B(y_2-\Tilde{y}_2)+(U_0(\omega,y_1)-U_0(\omega,\Tilde{y}_1))\big)W_{0,t}(\omega)}\\
            &\quad+\abs{\big((B(\phi(\cdot,\omega,y_1)y_2-\phi(\cdot,\omega,\Tilde{y}_1)\Tilde{y}_2))'_0+(U_0'(\omega,y_1)-U_0'(\omega,\Tilde{y}_1))\big)\mathbb{W}_{0,t}(\omega)}\\
            &\quad +C\normCalpha{\W(\omega)}\norm{B(\upvarphi(\cdot,\omega,y_1)y_2-\upvarphi(\cdot,\omega,\Tilde{y}_1)\Tilde{y}_2)+\left(\U(\omega,y_1)-\U(\omega,\Tilde{y}_1)\right)}{\mathscr{D}^{2\alpha}_{\W(\omega)}}\\
            &\leq C\abs{y_2-\tildeY_2}(1+\normCalpha{\W(\omega)})+\sup_{t\in[0,1]}\abs{Z_t(\omega,y_1)-Z_t(\omega,\tildeY_1)} \\
            &\quad+C\sup_{t\in[0,1]}\abs{U_t(\omega,y_1)-U_t(\omega,\Tilde{y}_1)}\normCalpha{\W(\omega)}+\sup_{t\in[0,1]}\abs{U_t'(\omega,y_1)-U_t'(\omega,\Tilde{y}_1)}\normCalpha{\W(\omega)}\\
            &\quad +C\normCalpha{\W(\omega)}\left(\norm{\upvarphi(\cdot,\omega,y_1)y_2-\upvarphi(\cdot,\omega,\Tilde{y}_1)\Tilde{y}_2}{\mathscr{D}^{2\alpha}_{\W(\omega)}}+\norm{\U(\omega,y_1)-\U(\omega,\Tilde{y}_1)}{\mathscr{D}^{2\alpha}_{\W(\omega)}}\right)\\
            &\quad+\int_0^t C\abs{\phi(s,\omega,y_1)y_2-\phi(s,\omega,\Tilde{y}_1)\Tilde{y}_2}\,\mathrm{d}s
        \end{align*}
        on a pathwise level. We realize that, in the last line, only the very last integral term still depends upon $t$, whereas the rest is constant in $t$, so that we rightfully consider it as increasing. Employing a classical Gronwall inequality like \cite[Lemma 2.7]{teschl2012ordinary}, we end up with
        \begin{align*}
            &\quad\abs{\phi(t,\omega,y_1)y_2-\phi(t,\omega,\Tilde{y}_1)\Tilde{y}_2}\\
             &\leq \mathrm{e}^{Ct}\bigg(C\abs{y_2-\tildeY_2}(1+\normCalpha{\W(\omega)})+\sup_{t\in[0,1]}\abs{Z_t(\omega,y_1)-Z_t(\omega,\tildeY_1)} \\
            &\quad+C\sup_{t\in[0,1]}\abs{U_t(\omega,y_1)-U_t(\omega,\Tilde{y}_1)}\normCalpha{\W(\omega)}+\sup_{t\in[0,1]}\abs{U_t'(\omega,y_1)-U_t'(\omega,\Tilde{y}_1)}\normCalpha{\W(\omega)}\\
            &\quad +C\normCalpha{\W(\omega)}\left(\norm{\upvarphi(\cdot,\omega,y_1)y_2-\upvarphi(\cdot,\omega,\Tilde{y}_1)\Tilde{y}_2}{\mathscr{D}^{2\alpha}_{\W(\omega)}}+\norm{\U(\omega,y_1)-\U(\omega,\Tilde{y}_1)}{\mathscr{D}^{2\alpha}_{\W(\omega)}}\right)\bigg)
        \end{align*}
        for all $t\in[0,1]$.
        With the same steps as before and $\mathrm{e}^{Ct}\leq \mathrm{e}^C\leq C$, this leads to
        \begin{align*}
            \expectationPathwise{\sup_{t\in[0,1]}\abs{\phi(t,\omega,y_1)y_2-\phi(t,\omega,\Tilde{y}_1)\Tilde{y}_2}^p}\leq C\abs{(y_1,y_2)-(\Tilde{y}_1,\Tilde{y}_2)}^{p/2},
        \end{align*}
        where we also exploited (iii).
        This bound is exactly what we intended to show in statement (i).\\

        \underline{\textit{Regarding} (ii)}. This is the case that is the easiest to deal with. In fact, with (i) established, we see that
        \begin{align*}
            &\quad\expectationPathwise{\sup_{t\in[0,1]}\abs{\phi'(t,\omega,y_1)y_2-\phi'(t,\omega,\Tilde{y}_1)\Tilde{y}_2}^p}\\
            &=\expectationPathwise{\sup_{t\in[0,1]}\abs{B\left(\phi(t,\omega,y_1)y_2-\phi(t,\omega,\Tilde{y}_1)\Tilde{y}_2\right) + \left(U_t(\omega,y_1)-U_t(\omega,\Tilde{y}_1)\right)}^p}\\
            &\leq C\expectationPathwise{\sup_{t\in[0,1]}\abs{\phi(t,\omega,y_1)y_2-\phi(t,\omega,\Tilde{y}_1)\Tilde{y}_2}^p}\\
            &\quad +C\expectationPathwise{\sup_{t\in[0,1]}\abs{U_t(\omega,y_1)-U_t(\omega,\Tilde{y}_1)}^p}\leq C\abs{(y_1,y_2)-(\Tilde{y}_1,\Tilde{y}_2)}^{p/2}
        \end{align*}
        already produces the desired inequality, and, thus, in total, concludes the proof.
    \end{proof}

    Let us shortly deviate from our main program to illustrate the shortcut that utilizes the mild rough Gronwall inequality.
    \begin{remark}\label{rde rem difficult lemma with mild gronwall} Here, we briefly explain that, under the conditions of the previous proof, the mild rough Gronwall inequality stated in \cite[Lemma 4.2]{mildGronwall2025} simplifies the argumentation by providing
    \begin{align*}\expectationPathwise{\norm{\upvarphi(\cdot,\omega,y_1)y_2-\upvarphi(\cdot,\omega,\Tilde{y}_1)\Tilde{y}_2}{\mathscr{D}^{2\alpha}_{\W(\omega)}}^p}\leq  C\abs{(y_1,y_2)-(\Tilde{y}_1,\Tilde{y}_2)}^{p/2}\end{align*}
    almost instantly. In this paragraph, we adopt the notation from \cite{mildGronwall2025}
    to prevent confusion. It is clear that $C_F$ and $C_G$ may be regarded $\geq 1$. The admissible choice $\kappa^\nu:=1/(2C\Phi_3)\in(0,1)$ in the mild rough Gronwall inequality leads, after plugging it into $C_1$ and $C_2$, and doing basic estimates, to
    \begin{align*}
        &\quad\norm{\upvarphi(\cdot,\omega,y_1)y_2-\upvarphi(\cdot,\omega,\Tilde{y}_1)\Tilde{y}_2}{\mathscr{D}^{2\alpha}_{\W(\omega)}} \\
        &\leq \mathrm{Const}(1+\abs{y_2-\tildeY_2}+\abs{B(y_2-\tildeY_2)+U_0(\omega,y_1)-U_0(\omega,\tildeY_1)})f(\normCalpha{\W(\omega)}),
    \end{align*}
        where $f:\R_{\geq 0}\to\R_{\geq 0}$ denotes an increasing polynomial. Standard estimates as before then entail the desired inequality
        \begin{align*}
            \expectationPathwise{\norm{\upvarphi(\cdot,\omega,y_1)y_2-\upvarphi(\cdot,\omega,\Tilde{y}_1)\Tilde{y}_2}{\mathscr{D}^{2\alpha}_{\W(\omega)}}^p}\leq  C\abs{(y_1,y_2)-(\Tilde{y}_1,\Tilde{y}_2)}^{p/2}.
        \end{align*}
    \end{remark}

\begin{lemma}\label{rde lem Lemma 8.5.6}
    Let $\Z(\omega):=(\Z(\omega,y_1))_{y_1\in\R^{D_1}}\subseteq\controlledRoughPath{\W(\omega)}{[0,1]}{\R^{D_2}}$ meet $\conditionRDE$. If $f:\R^{D_2}\rightarrow \R^{D_3}$ is a polynomial, then $(f(\Z(\omega,y_1)))_{y_1\in\R^{D_1}}\subseteq\controlledRoughPath{\W(\omega)}{[0,1]}{\R^{D_3}}$ also satisfies $\conditionRDE$.
\end{lemma}
\begin{proof}
    The proof consists of by-now-standard estimations supplemented by general calculus tricks like the identity $x^\sigma-y^\sigma=\sum_{d=1}^D(x_d^{\sigma_d}-y_d^{\sigma_d})\prod_{k=1}^{d-1}x_k^{\sigma_k}\prod_{l=d+1}^Dy_l^{\sigma_l}$ that holds for $x,y\in\R^D$ and $\sigma\in\N_0^D$. As the concrete computations are not entirely repetitive, we also supply a detailed proof in Appendix \ref{rde appendix proof of lemma with polynomial composition}.
\end{proof}
\begin{remark}\label{rde rem after Lem 8.5.6}
    Of course, by virtue of $\linearMaps{\R^M}{\R^{D_4}}\simeq\R^{MD_4}$, the statement of Lemma \ref{rde lem Lemma 8.5.6} remains true when studying a matrix-valued polynomial $\R^{D_2}\to \linearMaps{\R^M}{\R^{D_4}}$.
\end{remark}
At this point, we have finished the hard work of this section. Still, we need to harvest the implications that are of real importance to us.
\begin{lemma}\label{rde lem Lemma 8.5.8}
    For $n\in\N$, denote by $\upvarphi_n(\cdot,\omega,y_1,\dots,y_{n-1})y_n$ the solution to the $n$-th equation of (\ref{rde hierarchical system basic}), that is, to
    \begin{align*}
        \mathrm{d}Y_t&=\left(A^n_0Y_t+P_0^n(\phi_1(t,\omega)y_1,\dots,\phi_{n-1}(t,\omega,y_1,\dots,y_{n-2})y_{n-1})\right)\mathrm{d}t\\
        &\quad+\left(A^nY_t+P^n(\phi_1(t,\omega)y_1,\dots,\phi_{n-1}(t,\omega,y_1,\dots,y_{n-2})y_{n-1})\right)\mathrm{d}\W_t(\omega),\; Y_0=y_n\in\R^{D_n}.
    \end{align*}
    For each $n\in\N$, the family  $(\upvarphi_n(\cdot,\omega,y_1,\dots,y_{n-1})y_n)_{y_1\in\R^{D_1},\dots,y_n\in\R^{D_n}}$ satisfies $\conditionRDE$.
\end{lemma}
\begin{proof}
    We perform an inductive argument.

    \underline{$n=1$}. The initial step follows directly from Lemma \ref{rde lem Lemma 8.5.5} by considering the particularly simple, deterministic, constant controlled rough paths $ t\mapsto\Z_t(\omega,y_1):= (P^1_0,0)\in\controlledRoughPath{\W(\omega)}{[0,1]}{\R^{D_1}}$ and $ t\mapsto\U_t(\omega,y_1):= (P^1,0)\in\controlledRoughPath{\W(\omega)}{[0,1]}{\linearMaps{\R^M}{\R^{D_1}}}$, that are additionally independent of the parameter $y_1\in\R^{D_1}$.

    \underline{$n-1\to n$}. By the induction hypothesis, we assume that all of the first $n-1$ families $(\upvarphi_1(\cdot,\omega)y_1)_{y_1\in\R^{D_1}},\dots,(\upvarphi_{n-1}(\cdot,\omega,y_1,\dots,y_{n-2})y_{n-1})_{y_1\in\R^{D_1},\dots,y_{n-1}\in\R^{D_{n-1}}}$ satisfy conditions $\conditionRDE$. According to Lemma \ref{rde lem Lemma 8.5.6} and Remark \ref{rde rem after Lem 8.5.6}, also the two polynomially composed quantities $(P^n_0(\upvarphi_1(\cdot,\omega)y_1,\dots,\upvarphi_{n-1}(\cdot,\omega,y_1,\dots,y_{n-2})y_{n-1}))_{y_1\in\R^{D_1},\dots,y_{n-1}\in\R^{D_{n-1}}}$ and $(P^n(\upvarphi_1(\cdot,\omega)y_1,\dots,\upvarphi_{n-1}(\cdot,\omega,y_1,\dots,y_{n-2})y_{n-1}))_{y_1\in\R^{D_1},\dots,y_{n-1}\in\R^{D_{n-1}}}$ fulfill $\conditionRDE$. Employing again Lemma \ref{rde lem Lemma 8.5.5} yields the assertion.
\end{proof}

\begin{lemma}\label{rde lem Lemma 8.5.9}
    For $n\in\N$, denote by $\Phi_n$ the linear cocycle generated by the linearized version of the $n$-th equation of (\ref{rde hierarchical system basic}), that is, $\Phi_n$ is generated by
    $
        \mathrm{d}Y_t=A_0^nY_t\,\mathrm{d}t+A^nY_t\,\mathrm{d}\W_t(\omega)
    $, with $\Upphi_n(\cdot,\omega)\in\controlledRoughPath{\W(\omega)}{[0,1]}{\linearMaps{\R^{D_n}}{\R^{D_n}}}$ being the corresponding solution to the linear matrix-valued RDE $\,\mathrm{d}J_t=A_0^nJ_t\,\mathrm{d}t+A^nJ_t\,\mathrm{d}\W_t(\omega)$ starting at $\id$.
    \begin{enumerate}
    \item For parameters $y\in\R^{D_1}$ and $B\in\linearMaps{\R^M}{\R^{D_1}}$, and the deterministic, constant controlled rough paths $t\mapsto (y,0)\in\controlledRoughPath{\W(\omega)}{[0,1]}{\R^{D_1}}$ and  $t\mapsto(B,0)\in\controlledRoughPath{\W(\omega)}{[0,1]}{\linearMaps{\R^M}{\R^{D_1}}}$, define
    \begin{align*}
        \X^1(\omega,y,B)
        &:=\left(\int_0^\cdot \Phi_1(s,\omega)^{-1}y\,\mathrm{d}s+\int_0^\cdot \Phi_1(s,\omega)^{-1}B\,\mathrm{d}\W_s(\omega), \Phi_1(\cdot,\omega)^{-1}B\right)\in\controlledRoughPath{\W(\omega)}{[0,1]}{\R^{D_1}}.
    \end{align*}
    Then, $(\X^1(\omega,y,B))_{y\in\R^{D_1},B\in\linearMaps{\R^M}{\R^{D_1}}}$ meets $\conditionRDE$.
    \item For $N\in\N_{\geq 2}$ and a given set of initial conditions $y_1\in\R^{D_1},\dots,y_{N-1}\in\R^{D_{N-1}}$, denote by $\upvarphi_n(\cdot,\omega,y_1,\dots,y_{n-1})y_n$ the solution to the $n$-th equation of (\ref{rde hierarchical system basic}), $n=1,\dots,N-1$. Now, set
    \begin{align*}
        &\quad\X^N(\omega,y_1,\dots,y_{N-1})\\
            &:= \bigg(\int_0^\cdot\Phi_N(s,\omega)^{-1}P^N_0(\varphi_1(s,\omega)y_1,\dots,\varphi_{N-1}(s,\omega,y_1,\dots,y_{N-2})y_{N-1})\,\mathrm{d}s\\
            &\quad\quad\quad+\int_0^\cdot \Phi_N(s,\omega)^{-1}P^N(\varphi_1(s,\omega)y_1,\dots,\varphi_{N-1}(s,\omega,y_1,\dots,y_{N-2})y_{N-1})\,\mathrm{d}\W_s(\omega),\\
            &\quad\quad \Phi_N(\cdot,\omega)^{-1}P^N(\varphi_1(\cdot,\omega)y_1,\dots,\varphi_{N-1}(\cdot,\omega,y_1,\dots,y_{N-2})y_{N-1})\bigg)\in\controlledRoughPath{\W(\omega)}{[0,1]}{\R^{D_N}}.
    \end{align*}
    Then, for each $N\in\N_{\geq 2}$, the family $(\X^N(\omega,y_1,\dots,y_{N-1}))_{y_1\in\R^{D_1},\dots,y_{N-1}\in\R^{D_{N-1}}}$ fulfills $\conditionRDE$.
    \end{enumerate}
\end{lemma}
\begin{proof}
    \underline{\textit{Regarding (i)}}. Viewing $(y,B)\in\R^{D_1}\times\linearMaps{\R^M}{\R^{D_1}}$ as a long vector $\in\R^{D_1+MD_1}$, the case $N=1$ is a consequence of Lemma \ref{rde lem bounding expectation of I= int Phi U dW} by considering only deterministic and constant $t\mapsto \Z_t(\omega,y,B):=(y,0)\in\controlledRoughPath{\W(\omega)}{[0,1]}{\R^{D_1}}$ and $t\mapsto\U_t(\omega,y,B):=(B,0)\in\controlledRoughPath{\W(\omega)}{[0,1]}{\linearMaps{\R^M}{\R^{D_1}}}$.
    Of course, these controlled rough paths satisfy $\conditionRDE$.\\

    \underline{\textit{Regarding (ii)}}. If otherwise $N\in\N_{\geq 2}$, combining Lemma \ref{rde lem Lemma 8.5.8}, Lemma \ref{rde lem Lemma 8.5.6} and Lemma \ref{rde lem bounding expectation of I= int Phi U dW} entails the claim.
\end{proof}
The fact that all of the above $\X^n$ meet $\conditionRDE$, $n\in\N$, produces strong moment statements which are concretized as follows.
\begin{proposition}\label{rde prop Proposition 8.5.10} The following moment bounds hold true under the specified requirements depending on $N\in \N$:
\begin{enumerate}
    \item Let $\X^1(\omega,y,B)\in\controlledRoughPath{\W(\omega)}{[0,1]}{\R^{D_1}}$ be defined as in Lemma \ref{rde lem Lemma 8.5.9}.
    Then, for any $p\in(D_1+MD_1,\infty)$ and any compact $K\subseteq\R^{D_1}\times\linearMaps{\R^M}{\R^{D_1}}$ with $0\in K$, there exist some $C=C(p)\in\R$ and $q=q(p)\in[1,\infty)$ that satisfy
    \begin{align*}
        &\expectationPathwise{\sup_{(y,B)\in K}\sup_{t\in[0,1]} \abs{X^1_t(\omega,y,B)}^p
        }\leq C(\diam{K})^q.
    \end{align*}

    \item Let $N\in\N_{\geq2}$, and
    let $\X^N(\omega,y_1,\dots,y_{N-1})\in\controlledRoughPath{\W(\omega)}{[0,1]}{\R^{D_N}}$ be defined as in Lemma \ref{rde lem Lemma 8.5.9}. Then, for any $p\in(\sum_{n=1}^{N-1} D_{n},\infty)$ and any compact $K\subseteq \R^{D_1}\times\dots\times\R^{D_{N-1}}$ with $0\in K$, we find some constants $C=C(N,p)\in\R$ and $q=q(p)\in[1,\infty)$ such that
    \begin{align*}
        \expectationPathwise{&\sup_{(y_1,\dots,y_{N-1})\in K}\sup_{t\in[0,1]} \abs{X^N_t(\omega,y_1,\dots,y_{N-1})}^p
        } \leq C(\diam{K})^q.
    \end{align*}
    \end{enumerate}
\end{proposition}
\begin{proof}
    We intend to recycle \cite[Lemma 8.5.3]{Arnold1998RDS}. But, we point out that, as stated, the quoted result is flawed: The additional assumption that the compact sets in question contain the origin repairs the statement and the subsequent argumentation in \cite[Section 8.5]{Arnold1998RDS} goes through unchanged. We already implemented this necessary modification in our work, and we still refer to \cite[Lemma 8.5.3]{Arnold1998RDS} despite meaning its altered version.

    \underline{\textit{Regarding (ii)}}. Indeed, if we consider $N\in\N_{\geq2}$, and $(t,\omega)\mapsto X^N_t(\omega,y_1,\dots,y_{N-1})$
    as $\R^{D_N}$-valued
    adapted stochastic process,
    the process meets the requirements of \cite[Lemma 8.5.3]{Arnold1998RDS} due to Lemma \ref{rde lem Lemma 8.5.9}. Hence, applying \cite[Lemma 8.5.3]{Arnold1998RDS}
    yields the assertion. \\

    \underline{\textit{Regarding (i)}}. The same reasoning works for $N=1$, which completes the proof.
\end{proof}

\begin{remark}\label{rde rem after prop 8.5.10}
    We do not want to overlook the important detail that, up to this moment, we have only taken into account processes with time domain $[0,1]\subseteq\R_{\geq0}$. It remains an analysis of the "backward integral" situation described in Lemma \ref{rde Lem rdes cocycles but mild formulation} in order to also cover the negative axis $\R_{<0}$. To start off, we extend the notion of conditions $\conditionRDE$:
     We say that a parametrized family of random controlled rough paths $\U^-=(\U^-(\omega,y_1))_{\omega\in\Omega,y_1\in\V_1}\subseteq\bigcup_{\omega\in\Omega}\controlledRoughPath{\W(\omega)}{[-1,0]}{\V_2}$ satisfies $\conditionRDE$ in the $t\in\R_{<0}$-case if $\U^-$ enjoys the following properties:
        \begin{enumerate}
            \item For each $\omega$ and every parameter $y_1\in\V_1$, we have $\U^-(\omega,y_1)\in\controlledRoughPath{\W(\omega)}{[-1,0]}{\V_2}$.
            \item The entire process \begin{align*}[-1,0]\times\Omega\to\V_2\times\linearMaps{\R^M}{\V_2},\;(t,\omega)\mapsto \left(U^-_t(\omega,y_1),(U^-)_t'(\omega,y_1)\right)\end{align*} is $(\F^\infty_{t})_{t\in[-1,0]}$-backward adapted for every $y_1\in\V_1$.
            \item For any $p\in[2,\infty)$ and any compact set $K\subseteq \V_1$ with $0\in K$, there exist constants $C=C(p),C(p,K)\in\R$ such that \begin{align*}\expectationPathwise{\sup_{t\in[-1,0]}\abs{U^-_t(\omega,0)}^p+\sup_{t\in[-1,0]}\abs{(U^-)'_t(\omega,0)}^p+ \norm{\U^-(\omega,0)}{\mathscr{D}_{\W(\omega)}^{2\alpha}}^p} \leq C\end{align*}
                 and, for all $y_1,\Tilde{y}_1\in K$, it holds \begin{align*}
                    &\int_\Omega\bigg(\sup_{t\in[-1,0]}\abs{U^-_t(\omega,y_1)-U^-_t(\omega,\Tilde{y}_1)}^p+\sup_{t\in [-1,0]}\abs{(U^-)'_t(\omega,y_1)-(U^-)_t'(\omega,\Tilde{y}_1)}^p\\
                    &\quad+ \norm{\U^-(\omega,y_1)-\U^-(\omega,\Tilde{y}_1)}{\mathscr{D}_{\W(\omega)}^{2\alpha}}^p\bigg)\,\mathrm{d}P(\omega) \leq C \abs{y_1-\Tilde{y}_1}^{p/2}.
                \end{align*}
            \end{enumerate}
        As before, we also speak of $\U^-(\omega)$ satisfying $\conditionRDE$ in the $t\in\R_{<0}$-case.
        Of course, all results of the current section generalize almost verbatim to this case.

        We must point out that, although we use usual forward integration also for backward time, as presented in Lemma \ref{rde Lem rdes cocycles but mild formulation}, rewriting
        \begin{align*}
            \int_t^0 Y_r\,\mathrm{d}\W_r(\omega)
            &=\lim_{\abs{\mathrm{P}}\to0} \sum_{[u,v]\in\mathrm{P}} Y_uW_{u,v}(\omega)+Y_u'\mathbb{W}_{u,v}(\omega)\\
            &=\lim_{\abs{\mathrm{P}}\to0} \sum_{[u,v]\in\mathrm{P}} Y_vW_{u,v}+Y_v'(\mathbb{W}_{u,v}(\omega)-W_{u,v}(\omega)\otimes W_{u,v}(\omega))+\mathrm{o}(\abs{v-u})
        \end{align*}
        justifies that controlled forward integration does not destroy backward adaptedness.\\

        Finally, let us extend Lemma \ref{rde lem Lemma 8.5.9} and Proposition \ref{rde prop Proposition 8.5.10} paradigmatically: \textit{Still, we denote by $\Phi_n$ the linear cocycle generated by the linearized version of the $n$-th equation of (\ref{rde hierarchical system basic}) with $(\Phi_n(\cdot,\omega),-A^n\Phi_n(\cdot,\omega))\in\controlledRoughPath{\W(\omega)}{[-1,0]}{\linearMaps{\R^{D_n}}{\R^{D_n}}}$ being the corresponding solution to the linear matrix-valued RDE $\,\mathrm{d}J_t=-A_0^nJ_t\,\mathrm{d}t-A^nJ_t\,\mathrm{d}\W_t(\omega)$ with $J_0=\id$, where the extra minus sign is explained in Lemma \ref{rde Lem rdes cocycles but mild formulation}.
    \begin{enumerate}
    \item For parameters $y\in\R^{D_1}$ and $B\in\linearMaps{\R^M}{\R^{D_1}}$, and the deterministic and constant controlled rough paths $t\mapsto (y,0)\in\controlledRoughPath{\W(\omega)}{[-1,0]}{\R^{D_1}}$ and  $t\mapsto(B,0)\in\controlledRoughPath{\W(\omega)}{[-1,0]}{\linearMaps{\R^M}{\R^{D_1}}}$, put
    \begin{align*}
        t\mapsto\X^{-,1}_t(\omega,y,B)
        &:=\left(\int_{t}^0 \Phi_1(s,\omega)^{-1}y\,\mathrm{d}s+\int_t^0 \Phi_1(s,\omega)^{-1}B\,\mathrm{d}\W_s(\omega),\Phi(t,\omega)^{-1}B\right),
    \end{align*}
    so that $\X^{-,1}(\omega,y,B)\in\controlledRoughPath{\W(\omega)}{[-1,0]}{\R^{D_1}}$.
    Then, $(\X^{-,1}(\omega,y,B))_{y\in\R^{D_1},B\in\linearMaps{\R^M}{\R^{D_1}}}$ satisfies $\conditionRDE$ in the $t\in\R_{<0}$-case.
    \item For $N\in\N_{\geq 2}$ and a given set of initial conditions $y_1\in\R^{D_1},\dots,y_{N-1}\in\R^{D_{N-1}}$, we write $\upvarphi_n(\cdot,\omega,y_1,\dots,y_{n-1})y_n$ for the solution to the $n$-th equation of (\ref{rde hierarchical system basic}), $n=1,\dots,N-1$. Now, set
    \begin{align*}
        t\mapsto&\X^{-,N}_t(\omega,y_1,\dots,y_{N-1})\\
            &:= \bigg(\int_t^0\Phi_N(s,\omega)^{-1}P^N_0(\varphi_1(s,\omega)y_1,\dots,\varphi_{N-1}(s,\omega,y_1,\dots,y_{N-2})y_{N-1})\,\mathrm{d}s\\
            &\quad\quad\quad+\int_t^0 \Phi_N(s,\omega)^{-1}P^N(\varphi_1(s,\omega)y_1,\dots,\varphi_{N-1}(s,\omega,y_1,\dots,y_{N-2})y_{N-1})\,\mathrm{d}\W_s(\omega),\\
            &\quad\quad \Phi_N(t,\omega)^{-1}P^N(\varphi_1(t,\omega)y_1,\dots,\varphi_{N-1}(t,\omega,y_1,\dots,y_{N-2})y_{N-1})\bigg),
    \end{align*}
    such that $\X^{-,N}(\omega,y_1,\dots,y_{N-1})\in\controlledRoughPath{\W(\omega)}{[-1,0]}{\R^{D_N}}$.
    Then, for each $N\in\N_{\geq 2}$, the family $(\X^{-,N}(\omega,y_1,\dots,y_{N-1}))_{y_1\in\R^{D_1},\dots,y_{N-1}\in\R^{D_{N-1}}}$ fulfills $\conditionRDE$ in the $t\in\R_{<0}$-case.
    \end{enumerate}}

    \textit{As a consequence of this, the following moment bounds hold true under the specified requirements depending on $N\in \N$:}
\begin{enumerate}[label=\textit{(\roman*)}]
    \item
     \textit{For any $p\in(D_1+MD_1,\infty)$ and any compact $K\subseteq\R^{D_1}\times\linearMaps{\R^M}{\R^{D_1}}$ with $0\in K$, there exist some $C=C(p)\in\R$ and $q=q(p)\in[1,\infty)$ that satisfy}
    \begin{align*}
        &\expectationPathwise{\sup_{(y,B)\in K}\sup_{t\in[-1,0]} \abs{X^{-,1}_t(\omega,y,B)}^p
        }\leq C(\diam{K})^q.
    \end{align*}

    \item \textit{Let $N\in\N_{\geq2}$. Then, for any $p\in(\sum_{n=1}^{N-1} D_{n},\infty)$ and any compact $K\subseteq \R^{D_1}\times\dots\times\R^{D_{N-1}}$ with $0\in K$, we find some constants $C=C(N,p)\in\R$ and $q=q(p)\in[1,\infty)$ such that}
    \begin{align*}
        \expectationPathwise{&\sup_{(y_1,\dots,y_{N-1})\in K}\sup_{t\in[-1,0]} \abs{X^{-,N}_t(\omega,y_1,\dots,y_{N-1})}^p
        } \leq C(\diam{K})^q.
    \end{align*}
    \end{enumerate}
\end{remark}
        \subsubsection{Inheritance of Temperedness Revisited}
            The fact that all three components - the path, the Gubinelli derivative, and the semi-norm\footnote{This is, of course, different from the set of conditions used in \cite[Lemma 8.5.3]{Arnold1998RDS}. In fact, we only have the "path component" in the SDE case.} - were included in the definition of $\conditionRDE$ was motivated from a technical viewpoint. Due to the absence of martingale inequalities as in the SDE case, our estimations had to work with, for instance,  Theorem \ref{theorem gubinelli estimates}, where all three of these objects appear. In particular, in order to suitably control the path component itself like in Proposition \ref{rde prop Proposition 8.5.10}, we needed knowledge of all three components. This is now no longer the case because the invariant measures corresponding to (\ref{rde hierarchical system basic}) will be constructed solely by means of the path component of the random variables that give rise to the stationary solutions to the stages of (\ref{rde hierarchical system basic}).

Especially, temperedness properties of these stationary solutions will be analyzed on the path level (and they are, then, easy to extend to the Gubinelli derivative level, too). Therefore, we do not (have to) extend the results from the respective SDE paragraph \cite[Section 8.5.2, Step 2]{Arnold1998RDS} to the current rough framework. Instead, we are safe to recycle the statements which suffice for our purposes.
        \subsubsection{Invariant Measures of the Hierarchical System (\ref{rde hierarchical system basic})}
             \begin{definition}
        We call a linear cocycle $\Phi$ hyperbolic or non-resonant if $0\not\in\spectrum{\Phi}$.
    \end{definition}

    Recall that, for every $n\in\N$, we denote by $\Phi_n:\R\times\Omega\times\R^{D_n}\to\R^{D_n}$ the linear cocycle generated by $\,\mathrm{d}Y_t=A^n_0Y_t\,\mathrm{d}t+A^nY_t\,\mathrm{d}\W_t(\omega)$, which is the linearization of the $n$-th equation of the hierarchical system (\ref{rde hierarchical system basic}). According to Lemma \ref{rde lem linear RDE cocycle atuomatically satisfy IC of MET}, the MET automatically applies to all $\Phi_n$, resulting in their Lyapunov spectra $\spectrum{ \Phi_n}$. The core assumption for our purposes is that all linear cocycles $\Phi_n$ are hyperbolic. We write $\projection_n^s(\omega):\R^{D_n}\to E^s_n(\omega)$ for the projection with image $E^s_n(\omega)$ and kernel $E^u_n(\omega)$, where $\R^{D_n}=E_n^s(\omega)\oplus E_n^u(\omega)$ is the decomposition into stable and unstable spaces of $\Phi_n$ induced by the Oseledets splitting. And, we use $\projection^u_n(\omega):\R^{D_n}\to E^u_n(\omega)$ in the opposite case.

\begin{proposition}\label{rde prop Thm 8.5.15}
    We analyze the first stage of (\ref{rde hierarchical system basic}), that is, we deal with \begin{align*}\,\mathrm{d}Y_t^1=(A^1_0Y_t^1+P^1_0)\,\mathrm{d}t+(A^1Y_t^1+P^1)\,\mathrm{d}\W_t(\omega)\end{align*} in $\R^{D_1}$. Assume that $\Phi_1$ is hyperbolic, where $\Phi_1$ is the cocycle generated by the linearized equation $\,\mathrm{d}Y_t=A^1_0Y_t\,\mathrm{d}t+A^1Y_t\,\mathrm{d}\W_t(\omega)$. Then, the following statements hold true:
    \begin{enumerate}
        \item We define first
        \begin{align*}
            X_1^-(\omega):=\int_{-1}^0\Phi_1(s,\omega)^{-1}P^1_0\,\mathrm{d}s+\int_{-1}^0\Phi_1(s,\omega)^{-1}P^1\,\mathrm{d}\W_s(\omega)
        \end{align*}
        and
        \begin{align*}
        X_1^+(\omega):=\int_{0}^1\Phi_1(s,\omega)^{-1}P^1_0\,\mathrm{d}s+\int_{0}^1\Phi_1(s,\omega)^{-1}P^1\,\mathrm{d}\W_s(\omega),
        \end{align*}
        starting from which we set
        \begin{align*}
            \xi_1^s(\omega):=\sum_{k=0}^\infty\projection^s_1(\omega)\left(\Phi_1(-k,\omega)^{-1}X^-_1(\theta_{-k}\omega)\right)
        \end{align*}
        and also
        \begin{align*}
            \xi_1^u(\omega):=-\sum_{k=0}^\infty\projection^u_1(\omega)\left(\Phi_1(k,\omega)^{-1}X_1^+(\theta_k\omega)\right).
        \end{align*}
        The last two sums converge, for every $\omega$, absolutely and geometrically.
        \item In addition to that, $\xi_1^s(\omega)$, $\xi_1^u(\omega)$ and $\xi_1(\omega):=\xi^s_1(\omega)\oplus\xi_1^u(\omega)$ meet the temperedness condition.
        \item The affine cocycle $\phi_1$ generated by the initial RDE admits a unique invariant measure given by the random Dirac measure $\omega\mapsto\delta_{\xi_1(\omega)}$, that is, it holds  $\phi_1(t,\omega)\xi_1(\omega)=\xi_1(\theta_t\omega)$ for all $t\in\R$ and all $\omega$.
        \item The Gubinelli derivative of the solution path $t\mapsto\xi_1(\theta_t\omega)$, which is given by $t\mapsto A^1\xi_1(\theta_t\omega)+P^1$, also fulfills the temperedness condition. In particular, so does the entire solution $t\mapsto(\xi_1(\theta_t\omega),A^1\xi_1(\theta_t\omega)+P^1)\in\controlledRoughPath{\W(\omega)}{\R}{\R^{D_1}}$.
    \end{enumerate}
\end{proposition}
\begin{proof}
    $\underline{\textit{Regarding (i) and (ii)}}$. We interpret $X_1^+$ and $X^-_1$ as $\R^{D_1}$-valued random fields
    \begin{align*}
            X_1^-(\omega,y,B):=\int_{-1}^0\Phi_1(s,\omega)^{-1}y\,\mathrm{d}s+\int_{-1}^0\Phi_1(s,\omega)^{-1}B\,\mathrm{d}\W_s(\omega)
            \end{align*}
    and
    \begin{align*}
    X_1^+(\omega,y,B):=\int_{0}^1\Phi_1(s,\omega)^{-1}y\,\mathrm{d}s+\int_{0}^1\Phi_1(s,\omega)^{-1}B\,\mathrm{d}\W_s(\omega),
        \end{align*}
    with parameters $y\in\R^{D_1}$ and $B\in\linearMaps{ \R^M}{\R^{D_1}}\simeq\R^{MD_1}$, such that $X_1^+(\omega)=X_1^+(\omega,P^1_0,P^1)$ and $X^-_1(\omega)=X_1^-(\omega,P^1_0,P^1)$.
    Continuity in $\R^{D_1}\times\linearMaps{\R^M}{\R^{D_1}}$ is readily demonstrated by means of routine Riemann integral estimates and Theorem \ref{theorem gubinelli estimates}.
    The second requirement of \cite[Proposition 8.5.12]{Arnold1998RDS} follows immediately from Proposition \ref{rde prop Proposition 8.5.10} and Remark \ref{rde rem after prop 8.5.10}. Since the deterministic object $(P^1_0,P^1)$ is trivially tempered, \cite[Proposition 8.5.12]{Arnold1998RDS} provides that  both $X_1^+(\omega)$ and $X_1^-(\omega)$ fulfill the temperedness condition. As a consequence, \cite[Lemma 8.5.14]{Arnold1998RDS} and \cite[Lemma 8.5.13]{Arnold1998RDS} establish the convergence assertions in \textit{(i)} and the temperedness claims in \textit{(ii)}.\\

    $\underline{\textit{Regarding (iii)}}$.
    The invariance is a lengthy but routine computational matter that we postpone to Appendix \ref{appendix rde proof of Thm 8.5.15}.\\

    $\underline{\textit{Regarding (iv)}}$. The three processes $t\mapsto A^1$, $t\mapsto \xi_1(\theta_t\omega)$ and $t\mapsto P^1$ meet the temperedness condition. Invoking \cite[Lemma 4.1.2]{Arnold1998RDS} yields that $t\mapsto A^1\xi_1(\theta_t\omega)+P^1$ also satisfies the temperedness condition.
\end{proof}

    Upon respecting the change in the additive part, this result generalizes almost verbatim to all stages of (\ref{rde hierarchical system basic}).
\begin{theorem}\label{rde thm Theorem 8.5.16}
    Assume that all stages of (\ref{rde hierarchical system basic}) have hyperbolic linear parts, that is, the cocycle $\Phi_N$ generated by the linearized equation $\,\mathrm{d}Y_t=A^N_0Y_t\,\mathrm{d}t+A^NY_t\,\mathrm{d}\W_t(\omega)$ is hyperbolic for all $N\in\N$. For increasing $N\in\N_{\geq 2}$, we recursively analyze the $N$-th stage of (\ref{rde hierarchical system basic}), that is, we treat \begin{align*}
        \mathrm{d}Y_t^N&=\big((A^N_0Y_t^N+P^N_0(\phi_1(t,\omega)y_1,\dots,\phi_{N-1}(t,\omega,y_1,\dots,y_{N-2})y_{N-1})\big)\,\mathrm{d}t\\
        &\quad+\big(A^NY_t^N+P^N(\phi_1(t,\omega)y_1,\dots,\phi_{N-1}(t,\omega,y_1,\dots,y_{N-2})y_{N-1})\big)\,\mathrm{d}\W_t(\omega)
\end{align*} in $\R^{D_N}$, provided all stages $n=1,\dots,N-1$ are already dealt with. Then, the following statements hold true:
    \begin{enumerate}
        \item Given $\xi_1(\omega)$ from Proposition \ref{rde prop Thm 8.5.15}, we recursively define
        \begin{align*}
            X_N^-(\omega)&:=\int_{-1}^0\Phi_N(s,\omega)^{-1}P^N_0(\phi_1(s,\omega)\xi_1(\omega),\dots,\phi_{N-1}(s,\omega,\xi_1(\omega),\dots,\xi_{N-2}(\omega))\xi_{N-1}(\omega))\,\mathrm{d}s\\
            &\quad+\int_{-1}^0\Phi_N(s,\omega)^{-1}P^N(\phi_1(s,\omega)\xi_1(\omega),\dots,\phi_{N-1}(s,\omega,\xi_1(\omega),\dots,\xi_{N-2}(\omega))\xi_{N-1}(\omega))\,\mathrm{d}\W_s(\omega)
        \end{align*}
        and
        \begin{align*}
            X_N^+(\omega)&:=\int_{0}^1\Phi_N(s,\omega)^{-1}P^N_0(\phi_1(s,\omega)\xi_1(\omega),\dots,\phi_{N-1}(s,\omega,\xi_1(\omega),\dots,\xi_{N-2}(\omega))\xi_{N-1}(\omega))\,\mathrm{d}s\\
            &\quad+\int_{0}^1\Phi_N(s,\omega)^{-1}P^N(\phi_1(s,\omega)\xi_1(\omega),\dots,\phi_{N-1}(s,\omega,\xi_1(\omega),\dots,\xi_{N-2}(\omega))\xi_{N-1}(\omega))\,\mathrm{d}\W_s(\omega),
        \end{align*}
        starting from which we set
        \begin{align*}
            \xi_N^s(\omega):=\sum_{k=0}^\infty\projection^s_N(\omega)\left(\Phi_N(-k,\omega)^{-1}X^-_N(\theta_{-k}\omega)\right)
        \end{align*}
        and
        \begin{align*}
            \xi_N^u(\omega):=-\sum_{k=0}^\infty\projection^u_N(\omega)\left(\Phi_N(k,\omega)^{-1}X_N^+(\theta_k\omega)\right).
        \end{align*}
        The last two sums converge, for every $\omega$, absolutely and geometrically.
        \item In addition to that, $\xi_N^s(\omega)$, $\xi_N^u(\omega)$, and $\xi_N(\omega):=\xi^s_N(\omega)\oplus\xi_N^u(\omega)$ meet the temperedness condition.
        \item The affine flow $\phi_N$ generated by the initial RDE admits a unique invariant measure given by the random Dirac measure $\omega\mapsto\delta_{\xi_N(\omega)}$, that is, it holds  $\phi_N(t,\omega)\xi_N(\omega)=\xi_N(\theta_t\omega)$ for all $t\in\R$ and all $\omega$.
        \item The Gubinelli derivative of the solution path $t\mapsto\xi_N(\theta_t\omega)$, which is given by \begin{align*}t\mapsto A^N\xi_N(\theta_t\omega)+P^N(\xi_1(\theta_t\omega),\dots,\xi_{N-1}(\theta_t\omega)),\end{align*} also fulfills the temperedness condition. In particular, so does the entire solution $t\mapsto(\xi_N(\theta_t\omega),A^N\xi_N(\theta_t\omega)+P^N(\xi_1(\theta_t\omega),\dots,\xi_{N-1}(\theta_t\omega))\in\controlledRoughPath{\W(\omega)}{\R}{\R^{D_N}}$.
    \end{enumerate}
\end{theorem}
\begin{proof}
    We proceed by induction; the base case $N=1$ is precisely Proposition \ref{rde prop Thm 8.5.15}.

    \underline{\textit{Regarding (i) and (ii)}}. Now, by the induction hypothesis, $(\xi_1(\omega),\dots,\xi_{N-1}(\omega))$ satisfies the temperedness condition, which by virtue of Proposition \ref{rde prop Proposition 8.5.10} and Remark \ref{rde rem after prop 8.5.10} and \cite[Proposition 8.5.12]{Arnold1998RDS} results in $X_N^+(\omega)$ and $X_N^-(\omega)$ satisfying the temperedness condition, when understanding $X^+_N(\omega)=X^+_N(\omega,\xi_1(\omega),\dots,\xi_{N-1}(\omega))$ and $X^-_N(\omega)=X^-_N(\omega,\xi_1(\omega),\dots,\xi_{N-1}(\omega))$ as random fields evaluated at $(\xi_1(\omega),\dots,\xi_{N-1}(\omega))$, similarly to the proof of Proposition \ref{rde prop Thm 8.5.15}. Assertions \textit{(i)} and \textit{(ii)} then follow from \cite[Lemma 8.5.14]{Arnold1998RDS} and \cite[Lemma 8.5.13]{Arnold1998RDS}.\\

    \underline{\textit{Regarding (iii)}}. To establish the invariance claim, consult the proof of the invariance property in Proposition \ref{rde prop Thm 8.5.15} and proceed analogously, while keeping in mind that the invariance relations $\phi_n(t,\omega,\xi_1(\omega),\dots,\xi_{n-1}(\omega))\xi_n(\omega)=\xi_n(\theta_t\omega)$, $n=1,\dots,N-1$, $t\in\R$, hold true by the induction hypothesis.\\

    \underline{\textit{Regarding (iv)}}. By \cite[Lemma 4.1.2]{Arnold1998RDS} and the induction hypothesis, the polynomial composition $t\mapsto P^N(\xi_1(\theta_t\omega),\dots,\xi_{N-1}(\theta_t\omega))$ satisfies the temperedness condition. Furthermore, $t\mapsto A^N$ and $t\mapsto \xi_N(\theta_t\omega)$ fulfill the temperedness condition, too. Applying \cite[Lemma 4.1.2]{Arnold1998RDS} again, we obtain that $t\mapsto A^N\xi_N(\theta_t\omega)+P^N(\xi_1(\theta_t\omega),\dots,\xi_{N-1}(\theta_t\omega))$ meets the temperedness condition.
\end{proof}

This completes the collection of ingredients necessary to finish the proof of the main result, Theorem \ref{main thm in introduction}, which we present next.

    \subsection{Returning to the Original Rough Normal Form Problem}

\begin{proof}[Proof of Theorem \ref{main thm in introduction}]
    Consulting Theorem \ref{rde thm Theorem 8.5.16},  we obtain, for all $n\in\N_{\geq 2}$, unique stationary Taylor coefficients $H^n(\omega)\in H_{n,D}(\R^D)$ as path components of controlled rough paths $t\mapsto\H^n(\theta_t\omega)\in\controlledRoughPath{\W(\omega)}{\R}{H_{n,D}(\R^D)}$ that solve (\ref{rde tag bold guess}). According to Theorem \ref{rde thm Theorem 8.5.16}, both components of $t\mapsto\H^n(\theta_t\omega)$ additionally satisfy the temperedness condition. Patching the $H^n$ together in the sense of Borel's Lemma \cite[Lemma 8.2.12]{Arnold1998RDS} produces the desired random coordinate transform $H$.
\end{proof}

\begin{remark}
    The cocycle generated by the first $n$ RDEs of (\ref{rde hierarchical system basic}) considered as one, which is defined as $\phi^{(n)}(\cdot,\omega)(y_1,\dots,y_n):=(\phi_1(\cdot,\omega)y_1,\dots,\phi_n(\cdot,y_1,\dots,y_{n-1})y_n)$ in $\R^{D_1}\times\dots\times \R^{D_n}$, evaluated at $(\xi_1(\omega),\dots,\xi_n(\omega))$ is the unique stationary solution to these RDEs, so that $\phi^{(n)}(t,\omega)(\xi_1(\omega),\dots,\xi_n(\omega))=(\xi_1(\theta_t\omega),\dots,\xi_n(\theta_t\omega))$ for all $t\in\R$ and all $\omega$.
\end{remark}

Let us close this work by considering a simple example to illustrate the result. To do so, we adapt and expand \cite[Example 8.5.18]{Arnold1998RDS} into more general centered Gaussian processes which are not necessarily Brownian motion.

\begin{example}
    Let $(\W(\omega))_\omega\subseteq\driverRoughPath{\R}{\R}$ be a scalar, two-sided, centered Gaussian process that is a weakly geometric rough path cocycle, and that, additionally, starts at $W_0=0$ surely and satisfies surely a Strong Law of Large Numbers. For example, such processes are given by (restricted) fractional Brownian motion with Hurst parameter $\in(1/3,1/2]$, which is covered by \cite[Theorem 1.1, (C)]{fBM_LLN}. Given $F_0\in C^\infty(\R,\R)\cap C_{\mathrm{b}}^3(\R,\R)$ and $F\in C^\infty(\R,\R)\cap C_{\mathrm{b}}^3(\R,\R)$ with $F_0(0)=F(0)=0$, we study the scalar RDE
    \begin{align*}
        \mathrm{d}Y_t=F_0(Y_t)\,\mathrm{d}t+F(Y_t)\,\mathrm{d}\W_t(\omega).
    \end{align*}
    Its linearization reads $\,\mathrm{d}Y_t=\FzeroPrimeZero Y_t\,\mathrm{d}t+\FPrimeZero Y_t\,\mathrm{d}\W_t(\omega)$ and possesses the explicit fundamental matrix $(t,\omega)\mapsto\Phi(t,\omega)=\mathrm{e}^{\FzeroPrimeZero t+\FPrimeZero W_t(\omega)}$.
    Consequently, $\spectrum{\Phi}=\{\FzeroPrimeZero\}$.

    Let us advance to the hierarchical system. It is straightforward to validate that the $n$-th equation of the hierarchical system reads
    \begin{align*}
        \mathrm{d}Y_t^n=\left((1-n)\FzeroPrimeZero Y_t^n+K^n_0(\theta_t\omega)\right)\mathrm{d}t+\left((1-n)\FPrimeZero Y_t^n+K^n(\theta_t\omega)\right)\mathrm{d}\W_t(\omega).
    \end{align*}
    Each linearized RDE $\,\mathrm{d}Y_t^n=(1-n)\FzeroPrimeZero Y_t^ndt+(1-n)\FPrimeZero Y_t^n\,\mathrm{d}\W_t(\omega)$ again admits an easy fundamental matrix, explicitly given by $(t,\omega)\mapsto\Phi_n(t,\omega)=\mathrm{e}^{(1-n)\FzeroPrimeZero t+(1-n)\FPrimeZero W_t(\omega)}$ with $\spectrum{\Phi_n}=\{(1-n)\FzeroPrimeZero\}$.

    We recall that $n\in\N_{\geq 2}$, so that resonance occurs for some $n$ if and only if resonance occurs for all $n$ if and only if $\FzeroPrimeZero=0$. If $\FzeroPrimeZero\neq0$, then we compute
    \begin{align*}
        H^n(\omega)=\begin{cases}
            -\sum_{k=0}^\infty\Big(\int_0^{1} \mathrm{e}^{(n-1)\FzeroPrimeZero (t+k)+(n-1)\FPrimeZero W_{t+k}(\omega)} K^n_0(\theta_{t+k}\omega)\,\mathrm{d}t\\
        \quad\quad+\int_0^{1} \mathrm{e}^{(n-1)\FzeroPrimeZero (t+k)+(n-1)\FPrimeZero W_{t+k}(\omega)} K^n(\theta_{t+k}\omega)\,\mathrm{d}\W_t(\theta_k\omega)\Big), & \FzeroPrimeZero<0\\
        \sum_{k=0}^\infty\Big(\int_{-1}^{0} \mathrm{e}^{(n-1)\FzeroPrimeZero (t-k)+(n-1)\FPrimeZero W_{t-k}(\omega)} K^n_0(\theta_{t-k}\omega)\,\mathrm{d}t\\
        \quad\quad+\int_{-1}^{0} \mathrm{e}^{(n-1)\FzeroPrimeZero (t-k)+(n-1)\FPrimeZero W_{t-k}(\omega)} K^n(\theta_{t-k}\omega)\,\mathrm{d}\W_t(\theta_{-k}\omega)\Big) , &\FzeroPrimeZero>0
        \end{cases}.
    \end{align*}
\end{example}

{\footnotesize
\bibliographystyle{alpha}
\bibliography{mainFiles/refs}
}

\appendix

\section{Appendix}

\addtocontents{toc}
    {\protect\setcounter{tocdepth}{1}}
    \subsection{Notational Index}\label{appendix notation}

{\small
\renewcommand{\arraystretch}{1.2}
\begin{longtable}{p{0.20\textwidth} p{0.50\textwidth} p{0.22\textwidth}}
\caption{Generic but frequently used notation, together with the first place where each notion is introduced or used systematically.}\label{tab:notation}\\
\toprule
\textbf{Notation} & \textbf{Meaning} & \textbf{First Occurrence} \\
\midrule
\endfirsthead
\toprule
\textbf{Notation} & \textbf{Meaning} & \textbf{First occurrence} \\
\midrule
\endhead
\bottomrule
\endfoot
\bottomrule
\endlastfoot
$\N$, $\N_0$
& Natural numbers $\N=\{1,2,\dots\}$ and non-negative integers $\N_0=\N\cup\{0\}$.
& Theorem \ref{main thm in introduction}.\\

$\R$, $\Q$, $\R_{\geq0},\dots$
& Real numbers, rational numbers, non-negative reals, and similar.
& Section~\ref{sec:introduction}.\\

$\V,\V_0,\V_1,\V_2,\dots$
& Finite-dimensional Banach spaces endowed with norm $\abs{\cdot}$.
& Definition \ref{sde def random dynamical system}.\\


$\linearMaps{\V_1}{\V_2}$
& Space of linear maps $\V_1\to\V_2$.
& Theorem \ref{main thm in introduction}.\\

$v^d$, $v_d$
& Components of $v\in\V$ in a chosen basis $\{b_1,\dots,b_D\}$.
& Definition \ref{rde def multi-indices}, Remark \ref{rde rem matrix-valued rdes}.\\

$\mathrm{dom}(f)$, $\range{(f)}$
& Domain and range/image of a map $f$.
& Appendix \ref{rde appendix proof of lemma with polynomial composition}.\\

$C^k(\V_1,\V_2)$; $C^k_{\mathrm{b}}(\V_1,\V_2)$
& Space of $k\in\N_0\cup\{\infty\}$ times continuously differentiable maps $\V_1\to\V_2$; with bounded derivatives.
& Theorem \ref{main thm in introduction}.\\

$C^\alpha(\V_1,\V_2)$
& Space of $\alpha$-Hölder continuous maps $\V_1\to\V_2$.
& Definition \ref{def controlled rough path}.\\

$gf$
& Composition $g\circ f$ of linear maps, omitting "$\circ$".
& Lemma \ref{rde lem product rule}.\\

$\theta_t\theta_s$
& Composition $\theta_t\circ\theta_s$ of base shifts for $s,t\in\T$, given a metric dynamical system $\metricDynamicalSystemT$, omitting "$\circ$".
& Definition~\ref{sde def metric dyn system}.\\

$\simeq$
& Identification/isomorphism of spaces.
& Lemma \ref{rde lem polynomials simeq R^D}.\\

$v_1\formallyEqual v_2$
& Formal equality of expressions.
& Theorem \ref{main thm in introduction}.\\

$\V_1\otimes\V_2$
& Tensor product of $\V_1$ and $\V_2$ equipped with a compatible norm.
& Definition \ref{rde def driver rough path}.\\

$(v_1\otimes v_2)^T$
& Transposed tensor $v_2\otimes v_1$ for $v_1\in\V_1$ and $v_2\in\V_2$.
& Definition \ref{rde def geometric rough path}.\\



$\phi(\cdot,v)$, $\phi(\cdot)v$
& Evaluation of a flow $\phi$ at $v\in\V$.
& Definition~\ref{sde def random dynamical system}, Theorem \ref{thm MET}.\\

$C$
& Generic constant $\in\R_{\geq0}$ whose value may and will change from line to line. We write $C=C(p_1,\dots,p_L)$ to indicate that the constant $C$ depends upon parameters $p_1,\dots,p_L$.
& Theorem \ref{theorem gubinelli estimates}.\\




$\mathrm{P}$
& Finite partition of an interval with mesh size $|\mathrm{P}|:=\max\{\abs{v-u}:[u,v]\in\mathrm{P}\}$.
& Theorem~\ref{theorem gubinelli estimates}.\\


$X_{s,t}$
& Increment $X_t-X_s$ of a path $X:\R\to\V$.
& Definition \ref{rde def driver rough path}.\\

$\mathrm{O}(\abs{t-s}^\beta)$; $\mathrm{o}(\abs{t-s}^\beta)$
& Standard Landau notation: For $s,t\in \R$, $\beta\in\R_{>0}$, we write $f= \mathrm{O}(\abs{t-s}^\beta)$ if $\limsup_{t\to s}\abs{f(t)-f(s)}/\abs{t-s}^\beta<\infty$; and $f= \mathrm{o}(\abs{t-s}^\beta)$ if $\limsup_{t\to s}\abs{f(t)-f(s)}/\abs{t-s}^\beta=0$.
& Theorem \ref{rde lem ito-wentzell}.\\
\end{longtable}
}

 Let us add a subtle yet fundamental remark: To us, the identifications of main interest are $\R^{D_1}\otimes\R^{D_2}\simeq \R^{D_1\times D_2}\simeq\linearMaps{\R^{D_1}}{\R^{D_2}}$ and $\linearMaps{\R^{D_1}}{\linearMaps{\R^{D_2}}{\R^{D_3}}}\simeq\linearMaps{\R^{D_2}}{\linearMaps{\R^{D_1}}{\R^{D_3}}}$. These identifications are in fact canonical, which enables us to meaningfully identify elements of these spaces. Therefore, if we have, say, $A_1\in\linearMaps{\R^{D_1}}{\linearMaps{\R^{D_2}}{\R^{D_3}}}$ canonically identified with $A_2\in\linearMaps{\R^{D_2}}{\linearMaps{\R^{D_1}}{\R^{D_3}}}$, we write $A_1=A_2$, although $A_1$ and $A_2$ are not really equal, and a priori not even comparable via an equality at all. In any other similar situation, we also proceed like this.

 Also, when we speak of "elementary inequalities", we mean basic facts like the triangle inequality, or, for $p\in[1,\infty)$, the estimate $(a+b)^p\leq C(a^p+b^p)$ that holds for all $a,b\in\R_{\geq0}$, or $\norm{W}{\alpha}\leq \normCalpha{\W}$.

\subsection{Proof of Theorem \ref{rde lem ito-wentzell}: Two Technical Assertions}\label{appendix rde proof of ito-wentzell}
    Recall that, for $s,t\in[0,T]$ with $s<t$, we intend to establish
    \begin{enumerate}
        \item \begin{align*}\mathrm{D}_xH(t,Y_s)Y_{s,t} =\mathrm{D}_xH(s,Y_s)Y_{s,t} +\mathrm{D}_x\eta(s,Y_s)Y_s'(W_{s,t}\otimes W_{s,t}) + \mathrm{O}(\abs{t-s}^{3\alpha}) \text{ and}\end{align*}
        \item \begin{align*}\mathrm{D}_x^2H(t,Y_s)(Y_{s,t}\otimes Y_{s,t})=2 \mathrm{D}_x^2H(s,Y_s)(Y_s'\otimes Y_s')\mathbb{W}_{s,t} +\mathrm{D}_x^2H(s,Y_s)(Y_s'\otimes Y_s')[\W]_{s,t}+\mathrm{O}(\abs{t-s}^{3\alpha}).\end{align*}
    \end{enumerate}

    \underline{\textit{Regarding }(i)}. Performing two Gubinelli expansions, exploiting symmetry of $W_{s,t}\otimes W_{s,t}$ with $(\mathrm{D}_xH(\cdot,x))'=\mathrm{D}_x\eta(\cdot,x)$,  and successively collecting terms of order $=\mathrm{O}(\abs{t-s}^{3\alpha})$, we derive
    \begin{align*}
        \mathrm{D}_xH(t,x)Y_{s,t}
        &=\mathrm{D}_xH(s,x)Y_{s,t}+(\mathrm{D}_xH(\cdot,x))'_sW_{s,t}Y_{s,t}+\mathrm{O}(\abs{t-s}^{3\alpha})\\
        &=\mathrm{D}_xH(s,x)Y_{s,t}+\mathrm{D}_x\eta(s,x)Y_s'(W_{s,t}\otimes W_{s,t})+\mathrm{O}(\abs{t-s}^{3\alpha}).
    \end{align*}
    Completing the proof of (i) requires the choice $x=Y_s$.\\

    \underline{\textit{Regarding }(ii)}. By inserting the Gubinelli expansion $Y_{s,t}=Y_s'W_{s,t}+\mathrm{O}(\abs{t-s}^{2\alpha})$, it is easy to see that
    $
        Y_{s,t}\otimes Y_{s,t}=(Y_s'\otimes Y_s')(W_{s,t}\otimes W_{s,t})+\mathrm{O}(\abs{t-s}^{3\alpha}).
    $
     Hence,
    \begin{align*}
        \frac{1}{2}\mathrm{D}_x^2H(t,Y_s)(Y_{s,t}\otimes Y_{s,t})
        &=\frac{1}{2}\mathrm{D}_x^2H(t,Y_s)(Y_s'\otimes Y_s')(W_{s,t}\otimes W_{s,t})+\mathrm{O}(\abs{t-s}^{3\alpha})\\
        &=\frac{1}{2}\mathrm{D}_x^2H(t,Y_s)(Y_s'\otimes Y_s')[\W]_{s,t}+\mathrm{D}_x^2H(t,Y_s)(Y_s'\otimes Y_s')\mathrm{Sym}(\mathbb{W}_{s,t})\\
        &\quad+\mathrm{O}(\abs{t-s}^{3\alpha}),
    \end{align*}
    where we employed $W_{s,t}\otimes W_{s,t}=[\W]_{s,t}+2\mathrm{Sym}(\mathbb{W}_{s,t})$. Now, it is important to remember that the contraction of a symmetric matrix with an antisymmetric matrix vanishes. Thus, as the Hessian matrix $\mathrm{D}_x^2H(t,Y_s)$ is symmetric, it eliminates any antisymmetric part, which leads to
    \begin{align*}
        \mathrm{D}_x^2H(t,Y_s)(Y_s'\otimes Y_s')\mathbb{W}_{s,t}=\mathrm{D}_x^2H(t,Y_s)(Y_s'\otimes Y_s')\mathrm{Sym}(\mathbb{W}_{s,t}).
    \end{align*}
    With this fact in mind, we finally arrive at
    \begin{align*}
        \frac{1}{2}\mathrm{D}_x^2H(t,Y_s)(Y_{s,t}\otimes Y_{s,t})
        &=\frac{1}{2}\mathrm{D}_x^2H(s,Y_s)(Y_s'\otimes Y_s')[\W]_{s,t}+\mathrm{D}_x^2H(s,Y_s)(Y_s'\otimes Y_s')\mathbb{W}_{s,t}\\
        &\quad+\frac{1}{2}(\mathrm{D}_x^2H(t,Y_s)-\mathrm{D}_x^2H(s,Y_s))(Y_s'\otimes Y_s')[\W]_{s,t}\\
        &\quad+(\mathrm{D}_x^2H(t,Y_s)-\mathrm{D}_x^2H(s,Y_s))(Y_s'\otimes Y_s')\mathbb{W}_{s,t}+\mathrm{O}(\abs{t-s}^{3\alpha})\\
        &=\frac{1}{2}\mathrm{D}_x^2H(s,Y_s)(Y_s'\otimes Y_s')[\W]_{s,t}+\mathrm{D}_x^2H(s,Y_s)(Y_s'\otimes Y_s')\mathbb{W}_{s,t}\\
        &\quad+\mathrm{O}(\abs{t-s}^{3\alpha}),
    \end{align*}
    where we made use of the facts that $t\mapsto H(t,x)\in C^\alpha([0,T],\R^D)$ for every $x$, and that $[\W],\mathbb{W}\in C^{2\alpha}([0,T],\R^M\otimes\R^M)$. The last equation is precisely statement (ii).

    \subsection{Further Explanation Regarding Remark \ref{rde rem regarding generation of cocycles and backward integration}}\label{appendix rde further details on generation and backward integration}

    Remember that, most notably, a major difference is the occurrence of backward integration in the SDE setting, as opposed to usual forward integration in the RDE framework. This mystery is encoded in the fact that controlled integration contains additional information due to the second-order process that is an integral part of the rough path $\W(\omega)$, resulting in the choice of the evaluation point carrying less weight. For more details on controlled backward integration, see \cite[Section 5.4]{FrizHairer2020}.

Before we come to the main discussion about whether there is a fundamental difference between Lemma \ref{rde Lem rdes cocycles but mild formulation} and \cite[Theorem 2.3.39]{Arnold1998RDS}, let us illuminate two elementary claims regarding time-reflected rough paths.

\begin{comment}
\textbf{Assertion 1. }\textit{Let $\W\in\driverRoughPath{[-T,0]}{\R^M} $ be a weakly geometric rough path. Then, also its time-reflection $\W^\leftarrow\in\driverRoughPath{[0,T]}{\R^M}$ is weakly geometric, where, for $s,t\in[0,T]$, $s<t$, we define
    \begin{align*}
        (W^\leftarrow_{s,t},\mathbb{W}_{s,t}^\leftarrow):=(-W_{-t,-s},-\mathbb{W}_{-t,-s}+W_{-t,-s}\otimes W_{-t,-s}).
    \end{align*}}
    \end{comment}

    \begin{proposition}\label{appendix prop 1}
        Let $\W\in\driverRoughPath{[-T,0]}{\R^M} $ be a weakly geometric rough path. Then, also its time-reflection $\W^\leftarrow\in\driverRoughPath{[0,T]}{\R^M}$ is weakly geometric, where, for $s,t\in[0,T]$, $s<t$, we define
    \begin{align*}
        (W^\leftarrow_{s,t},\mathbb{W}_{s,t}^\leftarrow):=(-W_{-t,-s},-\mathbb{W}_{-t,-s}+W_{-t,-s}\otimes W_{-t,-s}).
    \end{align*}
    \end{proposition}
\begin{proof}
    The proof is a straightforward computation analogous to \cite[Exercise 2.4]{FrizHairer2020}.
\end{proof}

\begin{comment}
\textbf{Assertion 2. }\textit{Let $\W\in\driverRoughPath{[-T,0]}{\R^M}$ and $\W^\leftarrow\in\driverRoughPath{[0,T]}{\R^M}$ be as in Assertion 1. In addition to that, let $\Y\in\controlledRoughPath{\W}{[-T,0]}{\R^D}$. Then, putting $\Y^\leftarrow_t:=(-Y_{-t},-Y'_{-t})$ yields a time-reflected controlled rough path $\Y^\leftarrow\in\controlledRoughPath{\W^\leftarrow(\omega)}{[0,T]}{\R^D}$.}
\end{comment}

\begin{corollary}\label{appendix cor 2}
    Let $\W\in\driverRoughPath{[-T,0]}{\R^M}$ and $\W^\leftarrow\in\driverRoughPath{[0,T]}{\R^M}$ be as in Proposition \ref{appendix prop 1}. In addition to that, let $\Y\in\controlledRoughPath{\W}{[-T,0]}{\R^D}$. Then, putting $\Y^\leftarrow_t:=(-Y_{-t},-Y'_{-t})$ yields a time-reflected controlled rough path $\Y^\leftarrow\in\controlledRoughPath{\W^\leftarrow(\omega)}{[0,T]}{\R^D}$.
\end{corollary}
\begin{proof}
    The Hölder-continuity requirements are obviously satisfied. Regarding controlledness, we observe, for $s,t\in[0,T]$, $s<t$, that
    \begin{align*}
        Y^\leftarrow_{s,t}=Y_{-t,-s}=Y_{-t}'W_{-t,-s}+\mathrm{O}(\abs{t-s}^{2\alpha})=-Y_{-s}'W_{s,t}^\leftarrow+\mathrm{O}(\abs{t-s}^{2\alpha}),
    \end{align*}
    which ends the proof.
\end{proof}

 The next step is to compare the situation from Lemma \ref{rde Lem rdes cocycles but mild formulation} to the SDE situation explained in \cite[Theorem 2.3.39]{Arnold1998RDS}. To this end, let $W$ be a two-sided $\R^M$-valued Brownian motion with $W_0=0$ almost surely, and let $(\W(\omega))_\omega\subseteq\driverRoughPath{[-T,0]}{\R^M}$ be the weakly geometric Stratonovich enhancement of $W$. On the one hand, it is clear that the time-reflection $W^\leftarrow$ given by $W_t^\leftarrow:=W_{-t}$ is still a Brownian motion. On the other hand, according to Proposition \ref{appendix prop 1}, for almost every $\omega$, the rough path $\W^\leftarrow(\omega)\in\driverRoughPath{[0,T]}{\R^M}$ is well-defined and has first-order path $W^\leftarrow(\omega)$.

In order to ensure that the notions of the integrals are consistent for our purposes, which is essential given our expectation of parallelism between the theories, let us calculate that both versions reduce to the same object. We take an arbitrary suitable integrand process $Y:\R\times \Omega\to\linearMaps{\R^M}{\R^D}$ that we also interpret as being controlled by $\W$, so that $\Y(\omega):=(Y(\omega),Y'(\omega))\in\controlledRoughPath{\W(\omega)}{[-T,0]}{\linearMaps{\R^M}{\R^D}}$. Furthermore, we define, as proposed in Corollary \ref{appendix cor 2}, $Y^\leftarrow_t(\omega):=-Y_{-t}(\omega)$ and $(Y^\leftarrow)'_t(\omega):=-Y'_{-t}(\omega)$ to obtain $\Y^\leftarrow(\omega)\in\controlledRoughPath{\W^\leftarrow(\omega)}{[0,T]}{\linearMaps{\R^M}{\R^D}}$. It remains to check whether both classical Stratonovich integration and controlled Gubinelli integration produce the same output.

For simplicity of presentation, we work with concrete partitions of $[0,T]$ which, for $N\in\N$, are given by $\{0=t_0<t_1<\dots<t_{N-1}<t_N=T\}$.
Obviously, the time-reflected trial points $\{-T=-t_N<-t_{N-1}<\dots<-t_1<-t_0=0\}$ partition $[-T,0]$. In the SDE case, we recall
 the definitions provided in \cite[Section 5.4]{FrizHairer2020} to manipulate
 \begin{align*}
     \int_0^{T}Y^\leftarrow_t\circ \,\mathrm{d}^+W^\leftarrow_r
     &=\int_0^{T}Y_r^\leftarrow \,\mathrm{d}^+W^\leftarrow_r +\frac{1}{2}\langle Y^\leftarrow,W^\leftarrow\rangle_{T}^+\\
     &=\int_{-T}^0Y_r\,\mathrm{d}^-W_r-\frac{1}{2}\langle Y,W\rangle^-_{-T}=\int_{-T}^0Y_r\circ \,\mathrm{d}^-W_r. \mytag\label{appendix tag backward integration SDe}
 \end{align*}
 Yet, be aware that equation (\ref{appendix tag backward integration SDe}) only holds true algebraically, whereas on the martingale level, the LHS is forward-adapted while the RHS is backward-adapted.

The same objective in mind, we compute in the RDE case that
 \begin{align*}
     \int_0^{T}Y_r^\leftarrow(\omega)\,\mathrm{d}\W^\leftarrow_r(\omega)
     &=\lim_{N\toinf} \sum_{n=0}^{N-1} \Big(Y_{-t_n}(\omega) W_{-t_{n+1},-t_{n}}(\omega)\\
     &\quad\quad+Y'_{-t_n} (\omega)\left(\mathbb{W}_{-t_{n+1},-t_{n}}(\omega) -W_{-t_{n+1},-t_n}(\omega)\otimes W_{-t_{n+1},-t_n}(\omega) \right)\Big) \\
     &=\lim_{N\toinf} \sum_{n=0}^{N-1} \Big(Y_{-t_{n+1}}(\omega) W_{-t_{n+1},-t_{n}}(\omega)\\
     &\quad\quad+Y'_{-t_{n+1}} (\omega)\left(\mathbb{W}_{-t_{n+1},-t_{n}}(\omega) -W_{-t_{n+1},-t_n}(\omega)\otimes W_{-t_{n+1},-t_n}(\omega) \right)\\
     &\quad\quad +Y_{-t_{n+1},-t_n}(\omega) W_{-t_{n+1},-t_{n}}(\omega)\\
     &\quad\quad+Y'_{-t_{n+1},-t_n}(\omega) \left(\mathbb{W}_{-t_{n+1},-t_{n}}(\omega) -W_{-t_{n+1},-t_n}(\omega)\otimes W_{-t_{n+1},-t_n}(\omega) \right)\Big)\\
     &=\lim_{N\toinf} \sum_{n=0}^{N-1} Y_{-t_{n+1}}(\omega) W_{-t_{n+1},-t_{n}}(\omega)+Y'_{-t_{n+1}}(\omega) \mathbb{W}_{-t_{n+1},-t_{n}}(\omega)+\mathrm{O}(\abs{t_{n+1}-t_n}^{3\alpha})\\
     &=\int_{-T}^0Y_r(\omega)\,\mathrm{d}\W_r(\omega).\mytag\label{appendix tag backward integration RDE}
 \end{align*}
In the second-to-last equality, we Gubinelli-expanded $Y_{-t_{n+1},-t_n}(\omega)=Y'_{-t_{n+1}}(\omega)W_{-t_{n+1},-t_n}(\omega)+\mathrm{O}(\abs{t_{n+1}-t_n}^{2\alpha})$ and collected terms of order $=\mathrm{O}(\abs{t_{n+1}-t_n}^{3\alpha})$. Furthermore, in the last identity, we made use of $\lim_{N\toinf} \sum_{n=0}^{N-1}\mathrm{O}(\abs{t_{n+1}-t_n}^{3\alpha})=0$. Keep in mind the key aspect that adaptedness constitutes no issue here.

Standard results, such as \cite[Corollary 5.2]{FrizHairer2020}, tell us that the LHSs of (\ref{appendix tag backward integration SDe}) and (\ref{appendix tag backward integration RDE}) coincide almost surely, which is why the RHSs do, as well. This, finally, convinces us that the "loss" of backward integration in Lemma \ref{rde Lem rdes cocycles but mild formulation} compared to \cite[Theorem 2.3.39]{Arnold1998RDS} is merely a subtlety hidden in the definitions and does not really take place.

\subsection{Complete Proof of Lemma \ref{rde lem Lemma 8.5.6}}\label{rde appendix proof of lemma with polynomial composition}
Recall that $f(\Z(\omega,y_1))=(f(Z(\omega,y_1)),\mathrm{D}f(Z(\omega,y_1))Z'(\omega,y_1))$. Smooth mappings pass on adaptedness. Let $p\in[2,\infty)$ and choose a compact $K\subseteq\R^{D_1}$ with $0\in K$ with two fixed elements $y_1,\tildeY_1\in K$.

    Let $i\in\{1,\dots,D_3\}$, so that $f^i:\R^{D_2}\rightarrow\R$ is a polynomial that we write as $f^i=\sum_{\abs{\sigma}=0}^{\deg f^i}f^i_\sigma X^\sigma$ for $\sigma\in\N^{D_2}_0$ and coefficients $f_\sigma^i\in\R$. We utilize the general calculus identity $x^\sigma-y^\sigma=\sum_{d=1}^D(x_d^{\sigma_d}-y_d^{\sigma_d})\prod_{k=1}^{d-1}x_k^{\sigma_k}\prod_{l=d+1}^Dy_l^{\sigma_l}$ that holds for $x,y\in\R^D$ and $\sigma\in\N_0^D$. Adding elementary inequalities as well as Hölder's inequality, we obtain
    \begin{align*}
            &\quad\expectationPathwise{\sup_{t\in[0,1]}\abs{f^i(Z_t(\omega,y_1))-f^i(Z_t(\omega,\tildeY_1))}^p}\\
            &\leq \expectationPathwise{\sup_{t\in[0,1]}\abs{\sum_{\abs{\sigma}=0}^{\deg f^i}f_\sigma^i(Z_t(\omega,y_1)^\sigma-Z_t(\omega,\tildeY_1)^\sigma)}^p}\\
            &\leq C\sum_{\abs{\sigma}=0}^{\deg f^i}\sum_{d=1}^{D_2} \left(\expectationPathwise{\sup_{t\in[0,1]}\abs{(Z_t^d(\omega,y_1)^{\sigma_d}-Z_t^d(\omega,\tildeY_1)^{\sigma_d})}^{3p}}\right)^{1/3}\\
            &\quad\quad\cdot\left(\expectationPathwise{\sup_{t\in[0,1]}\abs{\prod_{k=1}^{d-1}Z_t^k(\omega,y_1)^{\sigma_k}}^{3p}}\right)^{1/3}
            \left(\expectationPathwise{\sup_{t\in[0,1]}
            \abs{\prod_{l=d+1}^{D_2}Z_t^l(\omega,\tildeY_1)^{\sigma_l}}^{3p}}\right)^{1/3}.
    \end{align*}
    The last two factors are easy to control: A simple application of Hölder's inequality shows that it suffices to bound $\expectationPathwise{\sup_{t\in[0,1]}\abs{Z_t^k(\omega,y_1)^{\sigma_k}}^q}$ for a suitable $q\in \{np:n\in\N\}$ and $k=1,\dots,D_2$. We recall that $y_1\in K$, where $K$ is a compact set with $0\in K$. With $\abs{Z_t^k(\omega,y_1)^{\sigma_k}}^q\leq \abs{Z_t(\omega,y_1)}^{\sigma_k q}$, we get by means of Lemma \ref{rde lem compactness lemma} that the two last factors are bounded by $C$. To study the first factor, we recycle the trick of rewriting \begin{align*}Z_t^d(\omega,y_1)^{\sigma_d}-Z_t^d(\omega,\tildeY_1)^{\sigma_d}=(Z_t^d(\omega,y_1)-Z_t^d(\omega,\tildeY_1))\sum_{r=0}^{\sigma_d-1} Z_t^d(\omega,y_1)^rZ_t^d(\omega,\tildeY_1)^{\sigma_d-1-r}.
    \end{align*}
    With this identity, a similar procedure as before and elementary considerations, we see
    \begin{align*}
        &\quad\expectationPathwise{\sup_{t\in[0,1]}\abs{(Z_t^d(\omega,y_1)^{\sigma_d}-Z_t^d(\omega,\tildeY_1)^{\sigma_d})}^{3p}}\\
        &\leq C\left(\expectationPathwise{\sup_{t\in[0,1]}\abs{Z_t^d(\omega,y_1)-Z_t^d(\omega,\tildeY_1)}^{6p}}\right)^{1/2}\\
        &\leq C\left(\expectationPathwise{\sup_{t\in[0,1]}\abs{Z_t(\omega,y_1)-Z_t(\omega,\tildeY_1)}^{6p}}\right)^{1/2} \leq C\abs{y_1-\tildeY_1}^{3p/2}.
    \end{align*}
    Consequently, it holds
    \begin{align*}
        \expectationPathwise{\sup_{t\in[0,1]}\abs{f^i(Z_t(\omega,y_1))-f^i(Z_t(\omega,\tildeY_1))}^p}\leq C\abs{y_1-\tildeY_1}^{p/2}.
    \end{align*}
    Since $i$ was arbitrary, we arrive at
    \begin{align*}
        &\quad\expectationPathwise{\sup_{t\in[0,1]}\abs{f(Z_t(\omega,y_1))-f(Z_t(\omega,\tildeY_1))}^p}\\
        &\leq C\sum_{i=1}^{D_3}\expectationPathwise{\sup_{t\in[0,1]}\abs{f^i(Z_t(\omega,y_1))-f^i(Z_t(\omega,\tildeY_1))}^p}\leq C\abs{y_1-\tildeY_1}^{p/2}.
    \end{align*}

    As a second step, we must derive \begin{align*}\expectationPathwise{\sup_{t\in[0,1]}\abs{\mathrm{D}f(Z_t(\omega,y_1))Z_t'(\omega,y_1)-\mathrm{D}f(Z_t(\omega,\tildeY_1))Z_t'(\omega,\tildeY_1)}^p}\leq C\abs{y_1-\tildeY_1}^{p/2}.
    \end{align*}
    Note that elementary inequalities give
    \begin{align*}
        &\quad\expectationPathwise{\sup_{t\in[0,1]}\abs{\mathrm{D}f(Z_t(\omega,y_1))Z_t'(\omega,y_1)-\mathrm{D}f(Z_t(\omega,\tildeY_1))Z_t'(\omega,\tildeY_1)}^p}\\
        &\leq C\expectationPathwise{\sup_{t\in[0,1]}\abs{(\mathrm{D}f(Z_t(\omega,y_1))-\mathrm{D}f(Z_t(\omega,\tildeY_1)))Z_t'(\omega,y_1)}^p}\\
        &\quad +C\expectationPathwise{\sup_{t\in[0,1]}\abs{\mathrm{D}f(Z_t(\omega,\tildeY_1))(Z_t'(\omega,y_1)-Z_t'(\omega,\tildeY_1))}^p}.
    \end{align*}
    Let us control the second summand. It is obvious that
    \begin{align*}
        \abs{\mathrm{D}f(Z_t(\omega,\tildeY_1))}
        &\leq \sum_{i=1}^{D_3}\sum_{j=1}^{D_2}\abs{\partial_j f^i(Z_t(\omega,\tildeY_1))}.
    \end{align*}
    But, $\partial_j f^i$ remains a polynomial $\R^{D_2}\rightarrow\R$. Thus, as we have already seen a few times in this proof, we deduce a constant bound $C$ such that
        $\expectationPathwise{\sup_{t\in[0,1]} \abs{\mathrm{D}f(Z_t(\omega,\tildeY_1))}^{2p}} \leq C$.
    By means of Hölder's inequality, this results in
    \begin{align*}
        \expectationPathwise{\sup_{t\in[0,1]}\abs{\mathrm{D}f(Z_t(\omega,\tildeY_1))(Z_t'(\omega,y_1)-Z_t'(\omega,\tildeY_1))}^p}\leq C\abs{y_1-\tildeY_1}^{p/2},
    \end{align*}
    for $\expectationPathwise{\sup_{t\in[0,1]}\abs{Z_t'(\omega,y_1)-Z_t'(\omega,\tildeY_1)}^{2p}}\leq C\abs{y_1-\tildeY_1}^p$. Still, we need to treat the first summand. As before, we know $\expectationPathwise{\sup_{t\in[0,1]} \abs{Z_t'(\omega,y_1)}^{2p}}\leq C$, so that Hölder's inequality tells us to focus on $\expectationPathwise{\sup_{t\in[0,1]}\abs{\mathrm{D}f(Z_t(\omega,y_1))-\mathrm{D}f(Z_t(\omega,\tildeY_1))}^{2p}}$. For a suitable expansion $\partial_jf^i=\sum_{\abs{\sigma}=0}^{\deg\partial_jf^i} f^{i,j}_\sigma X^\sigma$, where $\sigma\in\N_0^{D_2}$ and all $f^{i,j}_\sigma\in\R$, we compute
    \begin{align*}
        &\quad\expectationPathwise{\sup_{t\in[0,1]}\abs{\mathrm{D}f(Z_t(\omega,y_1))-\mathrm{D}f(Z_t(\omega,\tildeY_1))}^{2p}}\\
        &\leq \sum_{i=1}^{D_3}\sum_{j=1}^{D_2}\expectationPathwise{\sup_{t\in[0,1]}\abs{\partial_jf ^i(Z_t(\omega,y_1))-\partial_jf^i(Z_t(\omega,\tildeY_1))}^{2p}}\\
        &\leq\sum_{i=1}^{D_3}\sum_{j=1}^{D_2}\expectationPathwise{\sup_{t\in[0,1]}\sum_{\abs{\sigma}=0}^{\deg\partial_jf^i}\abs{f_\sigma^{i,j}(Z_t(\omega,y_1))^\sigma-Z_t(\omega,\tildeY_1)^\sigma)}^{2p}}.
    \end{align*}
    We see that we end up in the same situation as in the above calculation. Therefore, we proceed as before to conclude in summary that \begin{align*}\expectationPathwise{\sup_{t\in[0,1]}\abs{\mathrm{D}f(Z_t(\omega,y_1))Z_t'(\omega,y_1)-\mathrm{D}f(Z_t(\omega,\tildeY_1))Z_t'(\omega,\tildeY_1)}^p}\leq C\abs{y_1-\tildeY_1}^{p/2}.
   \end{align*}

    It remains to establish the third bound
    \begin{align*}
        \expectationPathwise{\normDalphaOmega{f(\Z(\omega,y_1))-f(\Z(\omega,\tildeY_1))}^p}\leq C\abs{y_1-\tildeY_1}^{p/2}.
    \end{align*}
    It is crucial to recall the continuity of composing controlled rough paths with regular functions as seen in \cite[Theorem 7.5]{FrizHairer2020}. More precisely, we refer to \cite[Proposition 2.3]{Geng2021_RoughPathsNotes}. Upon close inspection of the proof of \cite[Proposition 2.3]{Geng2021_RoughPathsNotes}, we first discover that the constant $M(\dots)$ is in fact a polynomial of its inputs, and second that it is enough to substitute $\norm{F}{C^3(\zeta)}$ for $\norm{F}{C^3_{\mathrm{b}}}$, both following the notation of \cite[Proposition 2.3]{Geng2021_RoughPathsNotes}, where $\zeta:=\zeta_1\cup \Tilde{\zeta}_1:= \conv(\range(Z(\omega,y_1)))\cup\conv(\range(Z(\omega,\tildeY_1)))$ and, for any set $S\subseteq\R^{D_2}$, we write $\conv(S):=\{(1-\lambda )s_1+\lambda s_2:s_1,s_2\in S,\lambda\in[0,1]\}$.
    We apply the estimate from \cite[Proposition 2.3]{Geng2021_RoughPathsNotes}, as well as Hölder's inequality, to obtain
    \begin{align*}
        &\quad\expectationPathwise{\normDalphaOmega{f(\Z(\omega,y_1))-f(\Z(\omega,\tildeY_1))}^p}\\
        &\leq\left(\expectationPathwise{\norm{f}{C^3(\zeta)}^{3p}}\right)^{1/3}\\
        &\quad\cdot\left(\expectationPathwise{M\left(\abs{Z_0'(\omega,y_1)},\abs{Z_0'(\omega,\tildeY_1)},\normDalphaOmega{\Z(\omega,y_1)},\normDalphaOmega{\Z(\omega,\tildeY_1)},\normCalpha{\W(\omega)}\right)^{3p}}\right)^{1/3}\\
        &\quad\cdot\left(\expectationPathwise{\left(\normDalphaOmega{\Z(\omega,y_1)-\Z(\omega,\tildeY_1)}+\abs{Z_0(\omega,y_1)-Z_0(\omega,\tildeY_1)}+\abs{Z_0'(\omega,y_1)-Z_0'(\omega,\tildeY_1)}\right)^{3p}}\right)^{1/3}.
    \end{align*}
    We must control all three factors. The second factor is $\leq C$ by standard procedures and Lemma \ref{rde lem compactness lemma} and Proposition \ref{rde prop consequences of heavy machinery}, since $M(\dots)$ is a polynomial of its inputs. Moreover, the third factor is $\leq C\abs{y_1-\tildeY_1}^{p/2}$ by means of usual arguments and the assumption that $\Z$ satisfies $\conditionRDE$. To finish the proof, it is left to establish $\expectationPathwise{\norm{f}{C^3(\zeta)}^{3p}}\leq C$. Notice that $\norm{f}{C^3(\zeta)}\leq C\max_{|\gamma|=0,\dots,3}\sup_{z\in\zeta}\abs{\partial^\gamma f(z)} \leq C\max_{\abs{\gamma}=0,\dots,3}(\sup_{z\in\zeta_1}\abs{\partial^\gamma f(z)}+\sup_{z\in\Tilde{\zeta}_1}\abs{\partial^\gamma f(z)})$ for $\gamma\in\N_0^{D_2}$, which gives that the first factor is bounded via
    \begin{align*}
        \expectationPathwise{\norm{f}{C^3(\zeta)}^{3p}}&\leq C\max_{\abs{\gamma}=0,\dots,3} \Bigg(
        \expectationPathwise{\sup_{s,t\in[0,1]}\sup_{\lambda\in[0,1]}\abs{\partial^\gamma f((1-\lambda)Z_s(\omega,y_1)+\lambda Z_t(\omega,y_1))}^{3p}}\\
        &\quad\quad+\expectationPathwise{\sup_{s,t\in[0,1]}\sup_{\lambda\in[0,1]}\abs{\partial^\gamma f((1-\lambda)Z_s(\omega,\tildeY_1)+\lambda Z_t(\omega,\tildeY_1))}^{3p}}\Bigg).
    \end{align*}
    Of course, it suffices to show that, for a fixed $\gamma\in\N_0^{D_2}$ with $|\gamma|\in\{0,\dots,3\}$, both of the two remaining summands are bounded by $C$.
    The integrands coincide syntactically upon disregarding the tilde, which is why it is enough to study the first summand. As $\partial^\gamma f:\R^{D_2}\to\R^{D_3}$ and $[0,1]\to \R^{D_2},\,\lambda\mapsto \lambda x+(1-\lambda)y$ are polynomials, so too is $[0,1]\to \R^{D_3},\,\lambda\mapsto \partial^\gamma f(\lambda x+(1-\lambda)y)$, where $x,y\in\R^{D_2}$.  We utilize a standard Bernstein basis expansion
    that leads to
    \begin{align*}
        \partial^\gamma f(\lambda x+(1-\lambda)y)=\sum_{i=0}^{\deg \partial^\gamma f}\sum_{j=0}^{\deg \partial^\gamma f-i}\lambda^i(1-\lambda)^j(\partial^\gamma f)_{i,j}(x,y),
    \end{align*}
    where each $(\partial^\gamma f)_{i,j}:\R^{D_2}\times \R^{D_2}\to \R^{D_3}$ is some suitable polynomial.
    Inserting this formula into the first summand entails
    \begin{align*}
        &\quad\expectationPathwise{\sup_{s,t\in[0,1]}\sup_{\lambda\in[0,1]}\abs{\partial^\gamma f((1-\lambda)Z_s(\omega,y_1)+\lambda Z_t(\omega,y_1))}^{3p}}\\
        &=\expectationPathwise{\sup_{s,t\in[0,1]}\sup_{\lambda\in[0,1]}\abs{\sum_{i=0}^{\deg\partial^\gamma f}\sum_{j=0}^{\deg \partial^\gamma f-i} \lambda^i(1-\lambda)^j(\partial^\gamma f)_{i,j}(Z_s(\omega,y_1),Z_t(\omega,y_1))}^{3p}}.
    \end{align*}
    With $\lambda^i(1-\lambda)^j\in [0,2]$ and the fact that $(\partial^\gamma f)_{i,j}$ is a polynomial, standard procedures via elementary inequalities, Hölder's inequality and Lemma \ref{rde lem compactness lemma} let us conclude that the first summand is $\leq C$, which completes the proof.

\subsection{Proof of Proposition \ref{rde prop Thm 8.5.15}, \textit{(iii)}: The Invariance Property}\label{appendix rde proof of Thm 8.5.15}
    For the existence, we have to verify that $\phi_1(t,\omega)\xi_1(\omega)=\xi_1(\theta_t\omega)$, or equivalently that $\projection^s_1(\theta_t\omega)\phi_1(t,\omega)\xi_1(\omega)=\xi_1^s(\theta_t\omega)$ and $\projection^u_1(\theta_t\omega)\phi_1(t,\omega)\xi_1(\omega)=\xi_1^u(\theta_t\omega)$ for all $t\in\R$. For brevity, we only deal with the latter, but the former follows by a similar argument.
    First, we present the decisive calculation that achieves the invariance claim. This manipulation is quite delicate, which is why we ought to treat it very carefully and slowly. Still, some equality signs might raise some further questions that we will resolve afterwards. Also, we focus on $t\in\R_{\geq0}$; the case $t\in\R_{<0}$ works analogously. We compute
        \begin{align*}
    &\quad\xi^u_1(\theta_t\omega) \\
    &\overset{(1)}{=} - \sum_{k=0}^\infty \projection^u_1(\theta_t\omega) \left(\Phi_1(k, \theta_t\omega)^{-1} X^+_1(\theta_{t+k}\omega)\right) \\
    &\overset{(2)}{=} - \sum_{k=0}^\infty \projection^u_1(\theta_t\omega) \left(\int_{0}^1\Phi_1(u+k,\theta_t\omega)^{-1}P^1_0\,\mathrm{d}u+\int_0^{1} \Phi_1(u+k, \theta_t\omega)^{-1} P^1\,\mathrm{d}\W_u(\theta_{t+k}\omega)\right) \\
    &\overset{(3)}{=} - \sum_{k=0}^\infty \projection^u_1(\theta_t\omega) \left(\int_{t}^{t+1}\Phi_1(u+k-t,\theta_t\omega)^{-1}P^1_0\,\mathrm{d}u+\int_t^{t+1} \Phi_1(u+k-t, \theta_t\omega)^{-1} P^1\,\mathrm{d}\W_u(\theta_{k}\omega)\right) \\
    &\overset{(4)}{=} - \sum_{k=0}^\infty \projection^u_1(\theta_t\omega) \left(\int_{t}^{t+1}\Phi_1(t,\omega)\Phi_1(u+k,\omega)^{-1}P^1_0\,\mathrm{d}u+\int_t^{t+1}\Phi_1(t,\omega) \Phi_1(u+k,\omega)^{-1} P^1\,\mathrm{d}\W_u(\theta_{k}\omega)\right) \\
    &\overset{(5)}{=} - \sum_{k=0}^\infty  \Phi_1(t, \omega)\projection^u_1(\omega)\left(\int_{t}^{t+1}\Phi_1(u+k,\omega)^{-1}P^1_0\,\mathrm{d}u+\int_t^{t+1} \Phi_1(u+k,\omega)^{-1} P^1\,\mathrm{d}\W_u(\theta_{k}\omega)\right) \\
    &\overset{(6)}{=} \Phi_1(t, \omega) \Bigg( -\sum_{k=0}^\infty\projection^u_1(\omega)\left(\int_{0}^{1}\Phi_1(u+k,\omega)^{-1}P^1_0\,\mathrm{d}u+\int_0^{1} \Phi_1(u+k,\omega)^{-1} P^1\,\mathrm{d}\W_u(\theta_{k}\omega)\right)\\
    &\quad\quad+ \projection^u_1(\omega)\left(\int_0^t\Phi_1(u,\omega)^{-1}P^1_0\,\mathrm{d}u+\int_0^t \Phi_1(u, \omega)^{-1} P^1\,\mathrm{d}\W_u(\omega)\right) \Bigg)\\
    &\overset{(7)}{=} \Phi_1(t, \omega) \left( \xi^u_1(\omega) + \projection^u_1(\omega) \left(\int_0^t\Phi_1(u,\omega)^{-1}P^1_0\,\mathrm{d}u+\int_0^t \Phi_1(u, \omega)^{-1} P^1\,\mathrm{d}\W_u(\omega)\right) \right)\\
    &\overset{(8)}{=}\projection^u_1(\theta_t\omega)\phi_1(t,\omega)\xi_1(\omega).
    \end{align*}
    Apart from the fact that eight justifications remain, this computation establishes the desired invariance assertion. Finally, we supply these explanations.\\

    In (1), we just write down the definition of $\xi^u_1(\theta_t\omega)$.

        For (2), given partitions $\mathrm{P}\subseteq[0,1]$, the required manipulation reads
        \begin{align*}
            &\quad\Phi_1(k, \theta_t \omega)^{-1} X_1^+(\theta_{t+k} \omega) \\
            &\overset{\mathclap{(2.1)}}{=} \Phi_1(k, \theta_t \omega)^{-1} \bigg(\int_0^1\Phi_1(u,\theta_{t+k}\omega)^{-1}P^1_0\,\mathrm{d}u+\int_0^1 \Phi_1(u, \theta_{t+k} \omega)^{-1} P^1\,\mathrm{d}\W_u(\theta_{t+k} \omega)\bigg) \\
            &\overset{\mathclap{(2.2)}}{=} \Phi_1(k, \theta_t \omega)^{-1} \lim_{|\mathrm{P}| \rightarrow 0} \Bigg(\sum_{[u,v]\in\mathrm{P}}\Phi_1(u,\theta_{t+k}\omega)^{-1}P^1_0(v-u)+\sum_{[u,v] \in \mathrm{P}} \Big(\Phi_1(u, \theta_{t+k} \omega)^{-1} P^1 W_{u,v}(\theta_{t+k} \omega) \\
            &\quad\quad- \Phi_1(u,\theta_{t+k}\omega)^{-1}\Phi_1'(u,\theta_{t+k}\omega)\Phi_1(u,\theta_{t+k}\omega)^{-1}P^1\mathbb{W}_{u,v}(\theta_{t+k}\omega)\Big)\Bigg)\\
            &\overset{\mathclap{(2.3)}}{=}  \lim_{|\mathrm{P}| \rightarrow 0}\Bigg(\sum_{[u,v]\in\mathrm{P}}\Phi_1(k,\theta_t\omega)^{-1}\Phi_1(u,\theta_{t+k}\omega)^{-1}P^1_0(v-u) \\
            &\quad+\sum_{[u,v] \in \mathrm{P}} \Big(\Phi_1(k, \theta_t \omega)^{-1}\Phi_1(u, \theta_{t+k} \omega)^{-1} P^1 W_{u,v}(\theta_{t+k} \omega) \\
            &\quad\quad- \Phi_1(k, \theta_t \omega)^{-1}\Phi_1(u,\theta_{t+k}\omega)^{-1}\Phi_1'(u,\theta_{t+k}\omega)\Phi_1(u,\theta_{t+k}\omega)^{-1}P^1\mathbb{W}_{u,v}(\theta_{t+k}\omega)\Big)\Bigg)\\
            &\overset{\mathclap{(2.4)}}{=} \int_0^1\Phi_1(k,\theta_t\omega)^{-1}\Phi_1(u,\theta_{t+k}\omega)^{-1}P^1_0\,\mathrm{d}u+  \int_0^1\Phi_1(k, \theta_t \omega)^{-1} \Phi_1(u, \theta_{t+k} \omega)^{-1} P^1\,\mathrm{d}\W_u(\theta_{t+k} \omega)  \\
            &\overset{\mathclap{(2.5)}}{=} \int_0^1\Phi_1(u+k,\theta_t\omega)^{-1}P^1_0\,\mathrm{d}u+\int_0^{1} \Phi_1(u+k, \theta_t \omega)^{-1} P^{1}\,\mathrm{d}\W_u(\theta_{t+k} \omega).
        \end{align*}
        Identity (2.1) is just the definition of $X_1^+(\theta_{t+k}\omega)$, and (2.2) follows from inserting the definitions of the Riemann and the controlled integral paired with Lemma \ref{rde lem phi inverse is also controlled rough path}. Since $\Phi_1(k,\theta_t\omega)^{-1}$ is independent of the partition $\mathrm{P}$, we pull it inside the limit process in (2.3). Moreover, $\Phi_1(k,\theta_t\omega)^{-1}$ is independent of the integration variable (which is the same as being independent of $\mathrm{P}$), so that the Gubinelli derivative of $u\mapsto\Phi_1(k, \theta_t \omega)^{-1} \Phi_1(u, \theta_{t+k} \omega)^{-1} P^1 $ is really $u\mapsto\Phi_1(k, \theta_t \omega)^{-1} (\Phi_1^{-1})'(u, \theta_{t+k} \omega) P^1 $. Thus, the definitions of the Riemann and the controlled integral entail (2.4). Equality (2.5) simply uses the cocycle identity.

        The time-shift taking place in (3) was justified  in Lemma \ref{rde lem time-shift}.

        The cocycle identity $\Phi_1(u+k,\omega)=\Phi_1(u+k-t,\theta_t\omega)\Phi_1(t,\omega)$ readily yields $\Phi_1(u+k-t,\theta_t\omega)^{-1}=\Phi_1(t,\omega)\Phi_1(u+k,\omega)^{-1}$, which explains (4).

        To reason
        \begin{align*}
            &\quad\int_{t}^{t+1} \Phi_1(t, \omega)\Phi_1(u+k,\omega)^{-1} P^1_0\,\mathrm{d}u+\int_{t}^{t+1} \Phi_1(t, \omega)\Phi_1(u+k,\omega)^{-1} P^1\,\mathrm{d}\W_u(\theta_k\omega)\\
            &=\Phi_1(t, \omega)\left(\int_{t}^{t+1} \Phi_1(u+k,\omega)^{-1} P^1_0\,\mathrm{d}u+\int_{t}^{t+1} \Phi_1(u+k,\omega)^{-1} P^1\,\mathrm{d}\W_u(\theta_k\omega)\right),
        \end{align*}
        adjust the steps performed in (2.2), (2.3) and (2.4) above. After that, employ $\Phi_1(t,\omega)\projection^u_1(\omega)=\projection^u_1(\theta_t\omega)\Phi_1(t,\omega)$, as provided by the MET, Theorem \ref{thm MET}, to obtain (5).

        In (6), we first pull out $\Phi_1(t,\omega)$. Then, we calculate
        \begin{align*}
            &\quad-\sum_{k=0}^\infty\projection^u_1(\omega)\left(\int_{0}^{1}\Phi_1(u+k,\omega)^{-1}P^1_0\,\mathrm{d}u+\int_0^{1} \Phi_1(u+k,\omega)^{-1} P^1\,\mathrm{d}\W_u(\theta_{k}\omega)\right)\\
            &\quad\quad+ \projection^u_1(\omega)\left(\int_0^t\Phi_1(u,\omega)^{-1}P^1_0\,\mathrm{d}u+\int_0^t \Phi_1(u, \omega)^{-1} P^1\,\mathrm{d}\W_u(\omega)\right)\\
            &\overset{\mathclap{(6.1)}}{=}-\sum_{k=0}^{\lfloor t\rfloor-1 }\projection^u_1(\omega)\left(\int_{0}^{1}\Phi_1(u+k,\omega)^{-1}P^1_0\,\mathrm{d}u+\int_0^{1} \Phi_1(u+k,\omega)^{-1} P^1\,\mathrm{d}\W_u(\theta_{k}\omega)\right)\\
            &\quad- \projection^u_1(\omega)\left(\int_{0}^{1}\Phi_1(u+\lfloor t\rfloor,\omega)^{-1}P^1_0\,\mathrm{d}u+\int_0^{1} \Phi_1(u+\lfloor t\rfloor,\omega)^{-1} P^1\,\mathrm{d}\W_u(\theta_{\lfloor t\rfloor}\omega)\right)\\
             &\quad-\sum_{k=\lceil t\rceil}^{\infty}\projection^u_1(\omega)\left(\int_{0}^{1}\Phi_1(u+k,\omega)^{-1}P^1_0\,\mathrm{d}u+\int_0^{1} \Phi_1(u+k,\omega)^{-1} P^1\,\mathrm{d}\W_u(\theta_{k}\omega)\right)\\
             &\quad +\projection^u_1(\omega)\left(\int_{0}^{t}\Phi_1(u,\omega)^{-1}P^1_0\,\mathrm{d}u+\int_0^{t} \Phi_1(u,\omega)^{-1} P^1\,\mathrm{d}\W_u(\omega)\right)\\
             &\overset{\mathclap{(6.2)}}{=}-\sum_{k=0}^{\lfloor t\rfloor-1 }\projection^u_1(\omega)\left(\int_{k}^{k+1}\Phi_1(u,\omega)^{-1}P^1_0\,\mathrm{d}u+\int_k^{k+1} \Phi_1(u,\omega)^{-1} P^1\,\mathrm{d}\W_u(\omega)\right)\\
            &\quad- \projection^u_1(\omega)\left(\int_{\lfloor t\rfloor}^{t}\Phi_1(u,\omega)^{-1}P^1_0\,\mathrm{d}u+\int_{\lfloor t\rfloor}^{t} \Phi_1(u,\omega)^{-1} P^1\,\mathrm{d}\W_u(\omega)\right)\\
            &\quad- \projection^u_1(\omega)\left(\int_{t}^{\lceil t\rceil}\Phi_1(u,\omega)^{-1}P^1_0\,\mathrm{d}u+\int_t^{\lceil t\rceil} \Phi_1(u,\omega)^{-1} P^1\,\mathrm{d}\W_u(\omega)\right)\\
             &\quad-\sum_{k=\lceil t\rceil}^{\infty}\projection^u_1(\omega)\left(\int_{0}^{1}\Phi_1(u+k,\omega)^{-1}P^1_0\,\mathrm{d}u+\int_0^{1} \Phi_1(u+k,\omega)^{-1} P^1\,\mathrm{d}\W_u(\theta_{k}\omega)\right)\\
             &\quad +\projection^u_1(\omega)\left(\int_{0}^{t}\Phi_1(u,\omega)^{-1}P^1_0\,\mathrm{d}u+\int_0^{t} \Phi_1(u,\omega)^{-1} P^1\,\mathrm{d}\W_u(\omega)\right)\\
             &\overset{\mathclap{(6.3)}}{=}- \projection^u_1(\omega)\left(\int_{t}^{\lceil t\rceil}\Phi_1(u,\omega)^{-1}P^1_0\,\mathrm{d}u+\int_t^{\lceil t\rceil} \Phi_1(u,\omega)^{-1} P^1\,\mathrm{d}\W_u(\omega)\right)\\
             &\quad-\sum_{k=\lceil t\rceil}^{\infty}\projection^u_1(\omega)\left(\int_{0}^{1}\Phi_1(u+k,\omega)^{-1}P^1_0\,\mathrm{d}u+\int_0^{1} \Phi_1(u+k,\omega)^{-1} P^1\,\mathrm{d}\W_u(\theta_{k}\omega)\right)\\
             &\overset{\mathclap{(6.4)}}{=}-\sum_{k=0}^\infty\projection^u_1(\omega)\left(\int_t^{t+1}\Phi_1(u+k,\omega)^{-1}P^1_0\,\mathrm{d}u+\int_{t}^{t+1}\Phi_1(u+k,\omega)^{-1}P^1\,\mathrm{d}\W_u(\theta_k\omega)\right).
        \end{align*}
        In (6.1), we split the sum, and, in (6.2), we performed a time-shift similar to the time-shift in (3). In (6.3), we applied linearity of both Riemann and controlled integration in order to detect cancellations. In (6.4), we performed a routine time-shift and re-ordering, so that an index shift of the sum-variable produces the result. Note that we employed the common notation that $\lfloor t\rfloor:=\max\{ r\in\N_0: r\leq t\}$ and $\lceil t\rceil:=\min\{r\in\N_0:r>t\}$.

        Identity (7) is argued analogously to (2) and (1) backwards.

        For (8), we remember that the projection $\projection^u_1(\omega)$ is a linear map. Then, we again utilize the rule $\Phi_1(t,\omega)\projection^u_1(\omega)=\projection^u_1(\theta_t\omega)\Phi_1(t,\omega)$ to conclude with Duhamel's formula.\\

    This ends the proof of the existence. To treat uniqueness, copy the uniqueness part of the proof of \cite[Theorem 4]{Arnold1992}.
\end{document}